   \pdfoutput=1

\documentclass{article}
\usepackage{excludeonly}

\usepackage{savesym}
\usepackage{scrextend}  % \begin{labeling}{longest} \item   ..  \end{labeling}
\usepackage{amsthm,amsmath}
\usepackage{amssymb,latexsym,graphicx}
\usepackage{amscd}
 \usepackage[all,cmtip]{xy}
\usepackage{tikz-cd}
  \usetikzlibrary{arrows}
\tikzset{
  commutative diagrams/.cd,
  arrow style=tikz
  }
\usepackage{tikz}

\usepackage{accents}
\usepackage{cite}
\usepackage{mathtools}
\usepackage{stmaryrd} % provides \llbracket (formal power series)

\usepackage{calligra}
\DeclareMathAlphabet{\mathcalligra}{T1}{calligra}{m}{n}
\DeclareMathAlphabet{\mathpzc}{OT1}{pzc}{m}{it}
\usepackage{hycolor}
\usepackage{xcolor}
\usepackage[
                       colorlinks=true,
                       linkcolor=black, %red,
                       citecolor=black, %green,
                       urlcolor=blue,
                     ]{hyperref}
\usepackage{soul} %provides strikethrough-command st{.. text ..}

\usepackage{stackengine} 
\usepackage{booktabs} % tables

\newtheorem{theoremABC}{Theorem}

\newtheorem{theorem}{Theorem}[section]

\newtheorem{corollary}[theorem]{Corollary}

\newtheorem{lemma}[theorem]{Lemma}
\newtheorem{proposition}[theorem]{Proposition}

\theoremstyle{definition}
\newtheorem{definition}[theorem]{Definition}

\newtheorem{remark}[theorem]{Remark}

\newtheorem{example}[theorem]{Example}

\theoremstyle{remark}

\newcommand{\C}{{\mathbb{C}}}

\newcommand{\LL}{{\mathbb{L}}}

\newcommand{\R}{{\mathbb{R}}}
\renewcommand{\SS}{{\mathbb{S}}}

\newcommand{\Z}{{\mathbb{Z}}}
\newcommand{\Aa}{{\mathcal{A}}}   % connections
\newcommand{\Bb}{{\mathcal{B}}}
\newcommand{\Hh}{{\mathcal{H}}}
\newcommand{\Ii}{{\mathcal{I}}}

\newcommand{\Kk}{{\mathcal{K}}}
\newcommand{\Ll}{{\mathcal{L}}}   % Lagrangian planes
\newcommand{\Mm}{{\mathcal{M}}}   % moduli space

\newcommand{\Qq}{{\mathcal{Q}}}

\newcommand{\Ss}{{\mathcal{S}}}
\newcommand{\Tt}{{\mathcal{T}}}
\newcommand{\Uu}{{\mathcal{U}}}
\newcommand{\Vv}{{\mathcal{V}}}

\newcommand{\Xx}{{\mathcal{X}}}
\newcommand{\Yy}{{\mathcal{Y}}}
\newcommand{\Zz}{{\mathcal{Z}}}
\newcommand{\AAA}{\mbf{A}}       % magnetic vector potential
\newcommand{\aaa}{\mbf{a}}       % 
\newcommand{\AAAdot}{\mbf{\dot A}} % dot magnetic vector potential
\newcommand{\BBB}{\mbf{B}}       % magnetic field B=rot A=\nabla x A
\newcommand{\rot}{\mathrm{rot}}  % rot
\newcommand{\rrr}{\mbf{r}}       %
\newcommand{\im}{{\rm im\, }}             % image
\newcommand{\id}{{\rm id}}                % identity
\newcommand{\Id}{{\rm Id}}
\newcommand{\grad}{\mathop{\mathrm{grad}}}    % gradient
\newcommand{\cgraph}[1]{\Gamma_{\kern-.5ex{}#1}}     % contact graph map
\renewcommand{\Re}{{\rm Re}}       % real part
\renewcommand{\Im}{{\rm Im}}       % imaginary part
\newcommand{\Diff}{{\rm Diff}}        % Diffeomorphisms
\newcommand{\Fix}{{\rm Fix}}            % Fixed points
\newcommand{\Crit}{{\rm Crit}}        % Critical points
\newcommand{\Omegacan}{\Omega_{\rm can} }  % T^*M+TM: can.sympl.form

\newcommand{\norm}{{\rm norm}}

\newcommand{\eps}{{\varepsilon}}

\newcommand{\Qfrak}{{\mathfrak Q}}

\newcommand{\zfrak}{{\mathfrak z}}
\newcommand{\Zfrak}{{\mathfrak Z}}
\newcommand{\inner}[2]{\langle #1, #2\rangle}   % inner product < , >
\newcommand{\INNER}[2]{\left\langle #1, #2\right\rangle}

\newcommand{\mbf}[1]{\text{\boldmath $#1$}}  % mbf=mathboldface

\def\Nablatop#1{\nabla^{#1}\kern-.5ex{}}

\def\NABLA#1{{\mathop{\nabla\kern-.5ex\lower1ex\hbox{$#1$}}}}
\def\Nabla#1{\nabla\kern-.5ex{}_{#1}}
\def\Tabla#1{\Tilde\nabla\kern-.5ex{}_{#1}}
\def\Babla#1{\widebar\nabla\kern-.5ex{}_{#1}}
\def\abs#1{\mathopen|#1\mathclose|}   
\def\Abs#1{\left|#1\right|}            
\def\norm#1{\mathopen\|#1\mathclose\|}
\def\Norm#1{\left\|#1\right\|}

\renewcommand{\Tilde}{\widetilde}

\newcommand{\p}{{\partial}}

\renewcommand{\1}{{{\mathchoice {\rm 1\mskip-4mu l} {\rm 1\mskip-4mu l}
{\rm 1\mskip-4.5mu l} {\rm 1\mskip-5mu l}}}}
\newcommand{\pretop}[1]{{^#1\mspace{-1mu}}}

\newlength\eqshift
\renewcommand\theequation{\thesection.\arabic{equation}}
\let\savetheequation\theequation

\makeatletter
\renewcommand*\env@matrix[1][\arraystretch]{%
  \edef\arraystretch{#1}%
  \hskip -\arraycolsep
  \let\@ifnextchar\new@ifnextchar
  \array{*\c@MaxMatrixCols c}}
\makeatother

\makeatletter
\let\save@mathaccent\mathaccent
\newcommand*\if@single[3]{%
  \setbox0\hbox{${\mathaccent"0362{#1}}^H$}%
  \setbox2\hbox{${\mathaccent"0362{\kern0pt#1}}^H$}%
  \ifdim\ht0=\ht2 #3\else #2\fi
  }
\newcommand*\rel@kern[1]{\kern#1\dimexpr\macc@kerna}
\newcommand*\widebar[1]{\@ifnextchar^{{\wide@bar{#1}{0}}}{\wide@bar{#1}{1}}}
\newcommand*\wide@bar[2]{\if@single{#1}{\wide@bar@{#1}{#2}{1}}{\wide@bar@{#1}{#2}{2}}}
\newcommand*\wide@bar@[3]{%
  \begingroup
  \def\mathaccent##1##2{%
    \let\mathaccent\save@mathaccent
    \if#32 \let\macc@nucleus\first@char \fi
    \setbox\z@\hbox{$\macc@style{\macc@nucleus}_{}$}%
    \setbox\tw@\hbox{$\macc@style{\macc@nucleus}{}_{}$}%
    \dimen@\wd\tw@
    \advance\dimen@-\wd\z@
    \divide\dimen@ 3
    \@tempdima\wd\tw@
    \advance\@tempdima-\scriptspace
    \divide\@tempdima 10
    \advance\dimen@-\@tempdima
    \ifdim\dimen@>\z@ \dimen@0pt\fi
    \rel@kern{0.6}\kern-\dimen@
    \if#31
      \overline{\rel@kern{-0.6}\kern\dimen@\macc@nucleus\rel@kern{0.4}\kern\dimen@}%
      \advance\dimen@0.4\dimexpr\macc@kerna
      \let\final@kern#2%
      \ifdim\dimen@<\z@ \let\final@kern1\fi
      \if\final@kern1 \kern-\dimen@\fi
    \else
      \overline{\rel@kern{-0.6}\kern\dimen@#1}%
    \fi
  }%
  \macc@depth\@ne
  \let\math@bgroup\@empty \let\math@egroup\macc@set@skewchar
  \mathsurround\z@ \frozen@everymath{\mathgroup\macc@group\relax}%
  \macc@set@skewchar\relax
  \let\mathaccentV\macc@nested@a
  \if#31
    \macc@nested@a\relax111{#1}%
  \else
    \def\gobble@till@marker##1\endmarker{}%
    \futurelet\first@char\gobble@till@marker#1\endmarker
    \ifcat\noexpand\first@char A\else
      \def\first@char{}%
    \fi
    \macc@nested@a\relax111{\first@char}%
  \fi
  \endgroup
}
\makeatother

\def\XXint#1#2#3{{\setbox0=\hbox{$#1{#2#3}{\int}$}
     \vcenter{\hbox{$#2#3$}}\kern-.5\wd0}}

\long\def\symbolfootnote[#1]#2{\begingroup%
\def\thefootnote{\fnsymbol{footnote}}\footnote[#1]{#2}\endgroup}
\tikzset{
  symbol/.style={
    draw=none,
    every to/.append style={
      edge node={node [sloped, allow upside down, auto=false]{$#1$}}}
  }
}

\begin{document}
\sloppy
\author{\quad Urs Frauenfelder \quad \qquad\qquad
             Joa Weber\footnote{
  Email: urs.frauenfelder@math.uni-augsburg.de
  \hfill
  joa@unicamp.br
  }
    \\
    Universit\"at Augsburg \qquad\qquad%\quad
    UNICAMP
}

\title{Towards a Floer theory for Mars \\ 
         \Large I - Twisted Zeeman systems}

\date{\today}

%\begin{titlepage}
\maketitle %(to set the title page and copyright page; see note)
                 %\include files (e.g., preface, introduction)
%\thispagestyle{empty}
%\newpage

%\end{titlepage}

%%%%%%%%%%%%%%%%%%%%%%%%%%%%%%%%%%%
%%%%%%% Abstract %%%%%%%%%%%%%%%%%%%%%
%%%%%%%%%%%%%%%%%%%%%%%%%%%%%%%%%%%
\begin{abstract}
In this article we study periodic orbits of an electron attracted by a
proton subject to Lorentz, electric, and Euler forces
where each of them is allowed to depend periodically on time.
This setup is motivated by the elliptic restricted three-body-problem
where the Lorentz force corresponds to Coriolis force, 
the Coulomb force is replaced by the gravitational force,
and the electric force of an external source is a combination of
centrifugal forces and gravitational forces of other bodies.
This is a singular version of a Euler-Hamilton system as discussed
in~\cite{Frauenfelder:2026a}.
The singularity is due to collisions of the electron with the proton,
respectively of two masses.
Due to the possibility of collisions this problem has to be
regularized.

We show how periodic collisional solutions of this problem
can be detected variationally in a non-local Lagrangian setup
as well as in a non-local Hamiltonian setup.
\end{abstract}

\tableofcontents

\newpage 
%%%%%%%%%%%%%%%%%%%%%%%%%%%%%%%%%%%
%%%%%%%%%%%%%%%%%%%%%%%%%%%%%%%%%%%
%%%%%%% Section  %%%%%%%%%%%%%%%%%%%%%%
%%%%%%%%%%%%%%%%%%%%%%%%%%%%%%%%%%%
%%%%%%%%%%%%%%%%%%%%%%%%%%%%%%%%%%%
\section{Introduction}

In 1609 the emperor of the Holy Roman Empire of the German Nation
% (962–1806) 
Rudolf~II received a dedication of a rather special book.
In it the author claimed that he had taken the God of War captive.
This author was Johannes Kepler and the book was the 
\emph{Astronomia Nova}~\cite{Kepler:1609a}\footnote{
  Deutsche \"Ubersetzung~\cite{Kepler:1937a}.
  English translation~\cite{Kepler:1992a}.
  }
in which Kepler showed that the orbit of Mars is an ellipse.
In fact, Mars was the Roman god of war
and in the \emph{Astronomia Nova} Kepler 
overcame the Aristotelian Worldview that on heaven
everything is moving on perfect circles.

\medskip
\textsc{Mars and Floer theory.}
Studying periodic orbits in the vicinity of Mars
also leads to new structures in Floer theory
which we describe in this paper.
In view of the high eccentricity of Mars
the system consisting of sun, Mars and a satellite
cannot well be described by the circular restricted\footnote{
  ``Restricted'' means that the satellite is considered massless and
  does not attract the sun and mars, but on the other hand is
  attracted by sun and Mars.
  }
three body problem, but has
to be attacked by the elliptic restricted three body problem.
Because of the eccentric movement of Mars, the Coriolis force in the
elliptic restricted three body problem is not constant, but
\emph{depends periodically on time}.
\newline
A basic system where the Coriolis force is time-dependent is a
merry-go-round which accelerates and decelerates.
How such systems can be described using time-dependent symplectic
forms was the subject of the recent work~\cite{Frauenfelder:2026a}.
Like the Lorentz force of a magnetic field, the Coriolis force depends
linearly on the velocity and therefore can be modeled by twisting the
standard symplectic form on the cotangent bundle.
In the case where the magnetic field is time-dependent, a new force
shows up, the Euler force,
which in contrast to the Lorentz force depends on the choice of
time-dependent primitive of the time-dependent symplectic form.
\newline
In the case of a satellite around Mars a new issue shows up
and these are collisions with Mars.
Although one does not want to put a satellite on a collisional orbit
with Mars, to obtain the global picture of periodic orbits
one has to take into account these collisional orbits as well.
In fact, considering a homotopy periodic orbits appear in families
and such families can go through collisions.

\medskip
\textsc{Variational approach to collisional orbits.}
For global theories like Floer theory
the crucial ingredient is a variational approach
to periodic orbits.
In order to obtain a variational approach 
as well to collisional orbits,
we apply in this paper the new regularization technique of
Barutello, Ortega, and  Verzini~\cite{Barutello:2021b}
of blowing up the loop space, see also~\cite{Frauenfelder:2025b},
to the case of time-dependent magnetic fields.

\smallskip
To treat these kind of problems we develop a general setup
which we refer to as twisted Zeeman systems.\footnote{
  A \textbf{planar twisted Zeeman system} describes the motion of an
  electron in the plane attracted by a proton in a time-dependent  
  magnetic field which admits a primitive depending twisted periodically
  on time.
  Such a magnetic field is necessarily time periodic, see
  Remark~\ref{rem:magn-2-form}.
  }
A \textbf{Stark-Zeeman system}, as introduced in~\cite{Cieliebak:2017c},
describes the motion of an electron attracted by a proton and subject
to a magnetic and electric field.
In a \textbf{twisted Zeeman system} the magnetic field is allowed to depend on
time and, in contrast to a time-independent case, the dynamics also
depends on the choice of the time-dependent primitive of the
time-dependent symplectic form.
While the magnetic field is assumed to depend periodically on time,
the time-dependence of the primitive is not necessarily periodic,
but is allowed to be twisted-periodic.
The additional freedom of allowing this twist makes it possible
to incorporate the electric field in the twist as well, as observed
in~\cite[\S 3.3]{Frauenfelder:2026a}.
Therefore the contribution of Stark, i.e.~the electric field,
is not needed any more, so that every Stark-Zeeman system can
be re-interpreted as a twisted Zeeman system.

\smallskip
Blowing up the loop space requires a reparametrization
of the loop which depends on the loop.
Therefore in the regularized system the symplectic form
on the loop space becomes non-local.
In the present article we introduce a Lagrangian and a Hamiltonian action
functional which are related by a non-local Legendre transformation,
see also~\cite{Frauenfelder:2021e,Cieliebak:2023a},
and show that the critical points correspond to periodic solutions of
twisted Zeeman systems allowing collisions of the electron with the proton.
In part II~\cite{Frauenfelder:2026d} we then show that the linearized
$L^2$-gradient flow equation of the non-local Hamiltonian action functional for
this non-local symplectic form is a Fredholm operator by showing
that its Hessian field almost extends by applying our general
Fredholm result from the recent article~\cite{Frauenfelder:2025g}.

%%%%%%%%%%%%%%%%%%%%%%%%%%%%%%%%%%%
%%%%%%% Subsection  %%%%%%%%%%%%%%%%%%%
%%%%%%%%%%%%%%%%%%%%%%%%%%%%%%%%%%%
\subsection{Main results}

We twist the Lagrangian of the Kepler problem with a $1$-form
on the plane $\C$ depending periodically on time
$\theta_{t+1}=\theta_t$.
The physical interpretation of this $1$-form is the following.
The exterior derivative with respect to space
gives rise to a time-dependent magnetic field $d\theta_t$.
The time derivative $\dot\theta_t$ of the $1$-form gives rise to a force
known as Euler force which can for example be felt on
an accelerating merry-go-round~\cite{Frauenfelder:2026a}.
We, more generally, consider $1$-forms which are only twisted-periodic
in time, and not periodic.
This additional twist allows to model electric forces as
well~\cite[\S 3.3]{Frauenfelder:2026a}.
In this introduction, for simplicity of exposition,
we do not consider this additional twist 
and also do not consider the case where the $1$-form
is defined on open subsets of the plane.

Periodic orbits of this problem can be detected variationally
as the critical points of the classical Lagrangian action functional
on loop space
\begin{equation*}%\label{eq:class-action-intro}
\begin{gathered}
   \Ss
   \colon \Ll\C^\times 
   \to\R
   ,\quad
   q\mapsto
   \int_0^1
   \Bigl(
   \tfrac12\Abs{\dot q_t}^2+\frac{1}{\abs{q_t}}
   +\theta_t|_{q_t}\dot q_t
   \Bigr)
   dt
   .
\end{gathered}
\end{equation*}
Since we have to exclude collisions with the singularity of
the Kepler potential $-1/\abs{q_t}$ at the origin,
this functional is only defined on loops that avoid
the origin, i.e. loops in $\C^\times:=\C\setminus\{0\}$.
In order to also allow collisional solutions variationally
we regularize this functional.
For that purpose we consider a non-local map
$\Qq\colon \Ll\C^\times\to \Ll\C^\times$.
This map $\Qq$ re-parametrizes the complex squaring map
by a circle diffeomorphism $\tau_z\colon\SS^1_t\to\SS^1_\tau$
which depends on the loop $z$ itself and is therefore not local.
The map $\Qq$ was discovered by
Barutello, Ortega, and  Verzini~\cite{Barutello:2021b}.
It is not smooth in the usual sense, but scale-smooth
in the sense of Hofer-Wysocki-Zehnder\cite{Frauenfelder:2026c}.
Pulling back the functional $\Ss$ under $\Qq$ we obtain the following
sum of three terms $\Bb=\Kk-\Uu+\Mm\colon \Ll\C^\times\to\R$
where
\begin{equation*}\label{eq:Bb-intro}
   \Bb(z)
   =
   2\norm{z}^2\norm{z^\prime}^2
   +\frac{1}{\norm{z}^2}
   +\int_0^1 \vartheta_{t_z(\tau)}|_{z(\tau)} 
   z^\prime(\tau)\, d\tau
   .
\end{equation*}
Here $t_z$ is the inverse of $\tau_z$ and
$\inner{\cdot}{\cdot}$ is the $L^2$-inner product with associated
norm $\norm{\cdot}$ and $\vartheta$ is the pull-back of $\theta$ by
the complex squaring map.
Because the map $\Qq$ is non-local, the periodic orbits pulled back
by $\Qq$ do not satisfy a second order ODE~(\ref{eq:q}) any more, but
a second order delay differential equation, the DDE~(\ref{eq:z})
which actually characterizes the critical points of $\Bb$.

The functional $\Bb$ naturally smoothly extends,
via the formula provided by $\Qq^*\Ss$
as illustrated in Figure~\ref{fig:fig-Qq-intro},
to the space $\Ll^\times\C=\Ll\C\setminus\{0\}$, i.e. the loop is
allowed to cross the origin, the only thing that is forbidden is that the
loop stays for all times at the origin.
Times where the loop crosses the origin
are interpreted as collision times.
After the extension, critical points of the functional consist of two
kinds, namely 
non-collisional solutions (Section~\ref{sec:reg-sol}),
i.e. solutions which never cross the origin,
and solutions admitting collisions (Section~\ref{sec:collisions}),
i.e. solutions which cross the origin.
\\
Non-collisional critical points 
are in 1-to-1 correspondence with critical points of $\Ss$, i.e.
periodic orbits of a particle subject to
firstly Newton's, respectively Coulomb's, force 
secondly a time-dependent Lorentz, respectively Coriolis, force
as well as thirdly the Euler force.
If the $1$-form is even allowed to be only twisted-periodic in time,
we can twist the $1$-form with an additional electric potential
and therefore the particle can be subject additionally to an electric force.

\begin{figure}%[b]
  \centering
\begin{equation*}
\begin{tikzcd} [row sep=tiny] %[column sep=small] %
{\color{gray} \text{extension}}
  &{\color{gray} \text{regularized}}
  &&
    {\color{gray} \text{classical}}
\\
\R
  &\R
  &&\R
\\ \mbox{} \\ \mbox{} \\
{\color{red}
   \Ll\C\setminus\{0\}
}
\arrow[uuu, "\Bb:=", "\Qq^*\Ss"']
\arrow[r, phantom, "\supset"]
  &
{\color{cyan}
   \Ll\C^\times
}
    \arrow[rr, shift left=2, "\Qq", "1:1"']
    \arrow[uuu, dashed, "\Qq^*\Ss"]
  &&{\color{cyan}\Ll \C^\times}
      \arrow[uuu, "\Ss"' ]
      \arrow[ll, shift left=2, "\Zz"]
\\
  &z
    \arrow[rr, mapsto]
  &&
    q_z:=z^2\circ \tau_z
\\
  &z_q:=q^\frac12\circ \tau_{1/\sqrt{q}}
  &&
    q
    \arrow[ll, mapsto]
\\
  &{\color{gray} \text{time $\tau\in\SS^1_\tau$}}
  &&
    {\color{gray} \text{$\SS^1_t\ni t$ time}}
    \arrow[ll, darkgray, "\tau_z=t_z^{-1}"',"\cong"]
\end{tikzcd}
\end{equation*}
  \caption{Regularization $\Bb$ on blow-up of loop space $\Ll\C^\times$}
   \label{fig:fig-Qq-intro}
\end{figure}
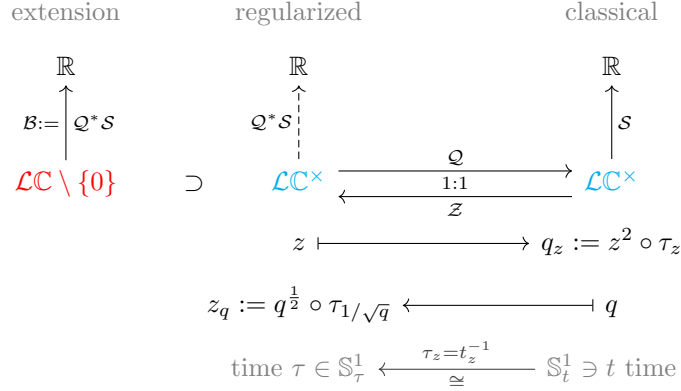

In Section~\ref{sec:correspondence}
we explain what precisely collisional solutions are on the $q$-side
and show the following.

\begin{theoremABC}\label{thm:A}
Collisional critical points $z$ of $\Bb$
are in $1$-to-$1$ correspondence to classical collisional solutions $q$.
\end{theoremABC}

Theorem~\ref{thm:A} is proved in Theorem~\ref{thm:bijection} 
by constructing a parcial inverse map
$\Zz$ to $\Qq$ which is illustrated in Figures~\ref{fig:fig-Qq-intro}
and~\ref{fig:fig-Qq-ext}.

\medskip
The second contribution of this paper is that,
in Section~\ref{sec:non-loc-Ham},
we provide as well a Hamiltonian formulation of the regularized solutions.
For this purpose we apply a non-local Legendre transformation to the
non-local functional $\Bb$ which is motivated by~\cite{Frauenfelder:2021e}.
The functional we obtain is given by
$$
   \Aa =i_\Vv\left(\Lambda +\pi^*\Theta\right) -\Hh
   \colon T^*\Ll^\times\C\to\R
   ,\qquad
   T^*\Ll^\times\C
   =\Ll^\times\C\times \Ll\C
   .
$$
The ingredients on the right hand side of $\Aa$ are the following.
Firstly there is the canonical vector field $\Vv=\p_\tau$ along
$T^*\Ll^\times\C$, namely $\Vv(z,\eta)=(z^\prime,\eta^\prime)$.
Secondly $\Lambda$ is a $1$-form on $T^*\Ll^\times\C$
obtained by integrating the Liouville form $\lambda$ on $T^*\C$.
Thirdly, in contrast to $\Lambda$, the $1$-from $\Theta$ on
$\Ll^\times\C$ is non-local.
If $z\in\Ll^\times\C$, then for a tangent vector
$\xi\in T_z\Ll^\times\C=\Ll\C$ we define
$$
   \Theta_z\xi
   :=\int_0^1 \vartheta_{t_z(\tau)}|_{z(\tau)} 
   \xi(\tau)\, d\tau
   .
$$
In particular, the last term $\Mm$ of $\Bb$
is given by $\Mm(z)=\Theta_z z^\prime$.
The base point projection and a
canonical injection are given by
\begin{equation}\label{eq:proj-inj}
\begin{aligned}
   \pi\colon T^*\Ll^\times\C&\to \Ll^\times\C\qquad&
   \iota\colon \Ll^\times\C&\to T^*\Ll^\times\C
   \\
   (z,\eta)&\mapsto z\qquad&
   z&\mapsto\left( z,{\color{brown} 4\norm{z}^2} z^\prime\right)
\end{aligned}
\end{equation}

The non-local Hamiltonian $\Hh\colon T^*\Ll^\times\C\to\R$
can be thought of as the non-local Legendre transform
of the first two terms $\Kk-\Uu$ of $\Bb$ and reads
$$
   \Hh(z,\eta)
   =\frac{\norm{\eta}^2}{8\norm{z}^2}-\frac{1}{\norm{z}^2}
   .
$$
Our second main result is the following.

\begin{theoremABC}\label{thm:B}
There is a bijection given by
\begin{equation*}
\begin{tikzcd} [row sep=tiny] %[column sep=small] %
\Crit\,\Aa
\arrow[rr, shift left=2,"\pi", "1:1"']
  &&\Crit\,\Bb
      \arrow[ll, shift left=2,"\iota"]
\end{tikzcd} 
.
\end{equation*}
\end{theoremABC}

\begin{proof}
Section~\ref{sec:EH=Lag}.
\end{proof}

The critical points of $\Aa$ are solutions of the following
second order delay equation.

\begin{theoremABC}\label{thm:C}
Suppose that $(z,\eta)\in\Ll^\times\C\times \Ll\C$
is a critical point of $\Aa$.
Then $(z,\eta)$ is a solution of the following problem
\begin{equation*}
  \left\{
   \begin{aligned}
     z^\prime
     &=\tfrac{1}{4\norm{z}^2}\;\eta
\\
     \eta^\prime
     &=\tfrac{\norm{\eta}^2-8}{4\norm{z}^4} z +\grad\Mm(z)
   \end{aligned}
   \right.
\end{equation*}
where the $L^2$-gradient of $\Mm$ is of the following form.
Writing the $1$-form $\vartheta_t= a_t^1\, dx+ a_t^2\, dy$
in the form of a pair $\aaa_t=( a_t^1,  a_t^2)\in\R^2$
there is the formula
\begin{equation*}%\label{eq:grad-Mm}
\begin{split}
   (\grad\Mm|_z)_\tau
   &=
% T3 magnetic M2
   -\frac{2z_\tau}{\norm{z}^4}
   \int_0^1
{\textstyle
   \int_0^\sigma\Abs{z_\rho}^2 d\rho
}
   \cdot
   \inner{\dot{\aaa}_{t_z(\sigma)}|_{z_\sigma}}{z^\prime_\sigma}_0\; d\sigma
% T5 magnetic M3
   -\frac{\Abs{z_\tau}^2}{\norm{z}^2}\dot{\aaa}_{t_z(\tau)}|_{z_\tau}
   \\
   &\quad
% T4 magnetic M1
   +\frac{2 z_\tau}{\norm{z}^2}
   \int_{\sigma=\tau}^1
   \inner{\dot{\aaa}_{t_z(\sigma)}|_{z_\sigma}}{z^\prime_\sigma}_0\;
   d\sigma
% T6 magnetic M4
   -(\rot\,\aaa_{t_z}|_z)_\tau\; i z^\prime_\tau
\end{split}
\end{equation*}
for each time $\tau\in\SS^1$ and abbreviating $z_\tau:=z(\tau)$.
\end{theoremABC}

\begin{proof}
Theorem~\ref{thm:Crit-Aa_Theta} and Lemma~\ref{le:gradient-Mm}.
\end{proof}

In  Appendix~\ref{sec:non-loc-Ham-twisted} we show that the exterior
derivative of the non-local $1$-form $\Lambda+\pi^*\Theta$
is a weak symplectic form on $T^*\Ll^\times\C$.
This uses an abstract result on weak symplectic forms in
Appendix~\ref{sec:tw-sympl-form}
which is of independent interest.
Hence the critical point equation in
Theorem~\ref{thm:C} can be interpreted as the Euler-Hamilton equation
of $\Hh$ with respect to the $1$-form $\Lambda+\pi^*\Theta$
as in~\cite{Frauenfelder:2026a}.
This is explained in Theorem~\ref{thm:Crit-Aa_Theta}.

The $1$-form $\Theta$ only exists for time-dependent $1$-forms
$\vartheta$ which are periodic in time.
However, the weak symplectic form on $T^*\Ll^\times\C$ still makes
sense in the twisted-periodic case and remains a weak symplectic form,
as discussed as well in Appendix~\ref{sec:non-loc-Ham-twisted}.

\medskip
\textbf{Notation.}
Working with functions on function spaces easily triggers
excesses of parentheses, which harms legibility.
Therefore we often write variables either as subscripts
$\Mm_z$ or in the form $\Mm|_z$, as opposed to $\Mm(z)$.
For time-dependence subscript has priority, for example if $t\mapsto
q(t)$ is a loop then $\theta_t|_{q_t} \dot q_t$ denotes a
time-dependent $1$-form at time $t$ and at the spatial point $q(t)$
evaluated on the velocity vector $\dot q(t)$.
Let $\norm{\cdot}$ and $\INNER{\cdot}{\cdot}$ be the
$L^2(\SS^1,\R^2)$ norm and inner~product.
There are two circles, the quotient circle $\SS^1=\R/\Z$ and the
unit circle $\SS^1_\C\subset \C$.
A map $f\colon\R\to X$ with $f_{t+1}=f_t$
$\forall t\in\R$ is called \textbf{periodic}, notation
$f\colon\SS^1\to X$.
While, as is common, ODE abbreviates ordinary differential equation,
\textbf{DDE} stands for \textbf{delay differential equation}.

\medskip\noindent
{\bf Acknowledgements.}
UF~acknowledges
support by DFG grant FR~2637/5-1.

\newpage
%%%%%%%%%%%%%%%%%%%%%%%%%%%%%%%%%%%
%%%%%%%%%%%%%%%%%%%%%%%%%%%%%%%%%%%
%%%%%%% Section  %%%%%%%%%%%%%%%%%%%%%%
%%%%%%%%%%%%%%%%%%%%%%%%%%%%%%%%%%%
%%%%%%%%%%%%%%%%%%%%%%%%%%%%%%%%%%%
\section{Twisted Zeeman system}
\label{sec:tw-Zeeman}

Let $0\in\Qfrak\subset\C$ be an open subset containing the origin.
Without $0$ we write
$$
   \Qfrak^\times:=\Qfrak\setminus\{0\}
{\color{gray}\;
   \subset \C\simeq\R^2 .
}
$$
We freely identify $\C$ with $\R^2$ using whichever is convenient.
For $q\in \Qfrak$ we write
\begin{equation}\label{eq:i-j_0}
{\color{gray}
   \C\ni 
{\color{black}
   q_1+iq_2\simeq (q_1,q_2)
}
   \in\R^2
   \qquad
   \Ll_\R(\C)\ni 
{\color{black}
   i\simeq j_0:=
   \begin{pmatrix}0&-1\\1&0\end{pmatrix}
}
   \in\Ll(\R^2)
   .
}
\end{equation}
For brevity we often write $i$. If $i$ is in front of an element
of $\R^2$ it means $j_0$.
By $\INNER{\cdot}{\cdot}_0$ we denote the Euclidean inner product on
$\R^2$, by $\abs{\cdot}$ the induced norm.

\begin{definition}[twisted-periodic]
\label{def:tw-per-1-form}
A \textbf{twisted-periodic \boldmath$1$-form}
$\theta=\{\theta_t\}_{t\in\R}$
is a real smooth family of $1$-forms on $\Qfrak$, notation
\begin{equation}\label{eq:tw-per-1-form}
   \theta_t=A_t^1\, dq_1+A_t^2\, dq_2
   ,\qquad
   A_t^1, A_t^2\colon\Qfrak\to\R ,
\end{equation}
with $\theta$ \textbf{twisted-periodic} in the sense that
a)~the time-derivative is periodic
\begin{equation}\label{eq:df_t}
   \text{a) $\dot\theta_{t+1}=\dot\theta_t$,}
   \qquad
   \text{b) $\theta_{t+1}=\theta_t+df$,}
\end{equation}
and b)~there is a smooth function
$f\colon\Qfrak\to \R$ satisfying the above,
cf.~\cite[\S5]{Frauenfelder:2026a},
called a \textbf{twist function} of $\theta=\{\theta_t\}$.
In the periodic case $f=0$.
As a consequence the exterior derivative is periodic, in symbols
$
   d\theta_{t+1}=d\theta_t
$.
The coefficients of the twisted $1$-form $\theta$
yield a vector field along $\Qfrak$ of the form
\begin{equation}\label{eq:AAA}
   \AAA_t|_q:=\left(A_t^1|_{(q_1,q_2)},
   A_t^2|_{(q_1,q_2)}\right)
   ,\qquad
   \AAAdot_{t+1}
   \stackrel{\text{(\ref{eq:df_t})}_{\mathrm{a}}}{=}\AAAdot_t
   ,
\end{equation}
called a \textbf{vector potential}.
\end{definition}

\begin{remark}\label{rem:magn-2-form}
Magnetic fields are described by closed $2$-forms $\sigma$,
where closedness $d\sigma=0$ encodes the fact that there are no
magnetic charges; see e.g.~\cite[\S\,2.4.1]{Weber:2017b}.
For a twisted-periodic $1$-form $\theta$ the $2$-form
\begin{equation}\label{eq:d-theta}
\begin{split}
   \sigma_t:=d\theta_t
   &=d\left(A_t^1\, dq_1+A_t^2\, dq_2\right)
   =\underbrace{\left(\p_1A^2_t-\p_2A^1_t\right)}_{=: \rot\,\AAA_t=:B_t}
    \underbrace{dq_1\wedge dq_2}_{=:\omega_0}
\end{split}
\end{equation}
is time-periodic 
$
   \sigma_t:=d\theta_t=d\theta_{t+1}=\sigma_{t+1}
$
and closed $d\sigma_t=dd\theta_t=0$.
\end{remark}

\boldmath
%%%%%%% Subsection  %%%%%%%%%%%%%%%%
\subsection{Classical Lagrangian functional $\Ss$}
\unboldmath

\begin{definition}\label{def:class-action}
Let $\theta$ be a twisted-periodic $1$-form~(\ref{eq:df_t})
and $f$ a twist function.
Define the \textbf{classical Lagrangian action} functional on the
\textbf{free loop space}
$\Ll\Qfrak^\times:=C^\infty(\SS^1,\Qfrak^\times)$
of $\Qfrak^\times:=\Qfrak\setminus\{0\}$ by
\begin{equation}\label{eq:class-action}
\begin{gathered}
   \Ss=\Ss_{L^\theta}
   \colon{\color{cyan}\Ll\Qfrak^\times}
   \to\R
\\
   q\mapsto
   \int_0^1
   \Bigl(
   \tfrac12\Abs{\dot q_t}^2+\tfrac{1}{\abs{q_t}}
   +\theta_t|_{q_t}\dot q_t
   \Bigr)
   dt
   -f(q_0)
{\color{gray}\;\,
   =\int_0^1 L_t^\theta(q_t,\dot q_t)\, dt -f(q_0)
}
\end{gathered}
\end{equation}
where $L_t ^\theta\colon T\Qfrak^\times\to\R$ is the function, called
\textbf{Lagrangian}, defined 
by
$$
   L_t^\theta(q,v)
   :=\tfrac12\Abs{v}^2+\tfrac{1}{\Abs{q}}+\theta_t|_{q} v
   .
$$
\end{definition}

\begin{remark}[twist term and integration interval, 
{\cite[Rmk.\,3.4]{Frauenfelder:2026a}}]
\label{rem:interval-not-SS^1}

(i)~Since we integrate over $[\underline{0},1]$, we subtract specifically
the \textbf{twist term \boldmath$f(q_{\underline{0}})$}
in~(\ref{eq:class-action}) in order to
get as critical points the periodic solutions of the
$(\AAA,\phi)$-equation~(\ref{eq:q}).

(ii)~If $\theta_t$ is twisted-periodic, and not periodic, then $L_t^\theta$
is not periodic in time.
Thus the integral~(\ref{eq:class-action}) is not over the circle
$\SS^1=\R/\Z$,
but over the interval~$[0,1]$; see~\cite[Rmk.\,3.4]{Frauenfelder:2026a}
which shows that integration $\int_k^{k+1}$
leads to the same value $\Ss(q)$ in~(\ref{eq:class-action})
for any integer $k\in\Z$.
\\
(iii)~If $\theta_{t+1}=\theta_t$ is periodic, then there is no twist
term ($f\equiv 0$) and any integration interval  $[r,r+1]$ for $r\in\R$
leads to the same value $\Ss(q)$ in~(\ref{eq:class-action}).\footnote{
  schematically
  $\int_r^{r+1}\theta_t=\int_r^1 \theta_t+\int_1^{r+1}\theta_t
  =\int_r^1 \theta_t+\int_0^r \theta_{t+1}=\int_0^1 \theta_t$
  as $\theta_{t+1}=\theta_t$ is periodic
  }
\end{remark}

\begin{remark}[classical critical point ODE]
\label{rem:class-action}
By~\cite[Prop.\,3.7]{Frauenfelder:2026a}, the critical points of the
functional $\Ss$ are solutions of the
$(\AAA,\phi)$-equation with $\phi(q)=-\frac{1}{\Abs{q}}$
\begin{equation}\label{eq:q}
\begin{split}
   \Crit\, \Ss
   &=
   \left\{q\in{\color{cyan}\Ll\Qfrak^\times}\mid    
   \ddot q
   =
   -B_t|_q j_0\dot q
   -\AAAdot_t|_q
   -\tfrac{q}{\Abs{q}^3}
   \right\}
   ,\quad
{\color{gray}\text{\small
   $B_t|_q:=\rot\,\AAA_t|_q$}
    .
}
\end{split}
\end{equation}
The critical points $q$ are called \textbf{classical} or
\textbf{physical solutions}.
A solution to this ODE describes the motion $t\mapsto q(t)$ of an
electron attracted by a proton at the origin and subject to a magnetic field.
The origin is a singularity of the potential $\phi$
and $q(t)$ approaching $0$ is called a \textbf{collision}. In the
classical description the origin is a forbidden locus.
\end{remark}

%\newpage %.\newpage
%%%%%%%%%%%%%%%%%%%%%%%%%%%%%%%%%%%
%%%%%%%%%%%%%%%%%%%%%%%%%%%%%%%%%%%
%%%%%%% Section  %%%%%%%%%%%%%%%%%%%%%%
%%%%%%%%%%%%%%%%%%%%%%%%%%%%%%%%%%%
%%%%%%%%%%%%%%%%%%%%%%%%%%%%%%%%%%%
\section{Loop space blow-up}\label{sec:BOV-blow-up}

\boldmath
%%%%%%% Subsection  %%%%%%%%%%%%%%%%
\subsection{Complex squaring and twisted loop spaces}
\unboldmath

\boldmath
%%%%%%% Subsubsection:  %%%%%%%%%%%%%%%%
\subsubsection*{Standard notions on the Euclidean plane $\R^2\simeq\C$}
\unboldmath

It is convenient to identify $\R^2$ with $\C$
via $(x,y)\mapsto x+iy=z$.
We recall how standard notions look like on each side.
The complex side lends itself
for shorter formulas and quicker calculations.

\smallskip
On $\R^2$ three natural players are
the Euclidean inner product, the counter-clockwise quarter rotation,
and the natural symplectic form.
They are defined~by
$$
   \INNER{z}{\zeta}_0
   :=x\xi+y\eta
   ,\qquad
   i\stackrel{\text{(\ref{eq:i-j_0})}}{\simeq} j_0
   ,\qquad
   \omega_0(z,\zeta):=x\eta-y\xi ,
$$
for $z=(x,y)$ and $\zeta=(\xi,\eta)$.
The three are compatible in the sense that
\begin{equation}\label{eq:g_0-omega_0-J_0}
   \omega_0(\cdot,\cdot):=\INNER{j_0\cdot}{\cdot}_0
   \qquad
   {\color{gray}\small\text{or, equivalently}}
   \qquad
   \INNER{\cdot}{\cdot}_0:=\omega_0(\cdot, j_0\cdot).
\end{equation}
As $j_0j_0=-\1$ and $j_0^t=-j_0$ is skew-symmetric,
the previous identity tells that
$$
   \omega_0(j_0\cdot,j_0\cdot)=\omega_0(\cdot,\cdot)
   ,\qquad
   \INNER{j_0\cdot}{j_0\cdot}_0=\INNER{\cdot}{\cdot}_0 ,
$$
are both invariant under $j_0$.
In natural coordinates $\omega_0=dx\wedge dy$.

\begin{remark}\label{rmk:3-natural-players-complex}
On $\C$ the three natural players appear as follows\footnote{
  Of course, on the left hand sides $z$ stands for $x+iy$ and
  on the right hand sides for $(x,y)$.
  }
\begin{equation}\label{eq:bar-z-zeta}
   \bar z\zeta
   =\INNER{z}{\zeta}_0+i\omega_0(z,\zeta)
   ,\qquad
   \Re(\bar z\zeta)=\INNER{z}{\zeta}_0
   ,\quad
   \Im(\bar z\zeta)=\omega_0(z,\zeta) ,
\end{equation}
and
\begin{equation*}\label{eq:bar-iz-zeta}
\begin{split}
   \widebar{iz}\zeta
   &=\INNER{j_0 z}{\zeta}_0+i\omega_0(j_0 z,\zeta)
\\
   &=\omega_0(z,\zeta)-i\INNER{z}{\zeta}_0
\end{split}
\end{equation*}
for $z=x+iy$ and $\zeta=\xi+i\eta$
and where $\bar z:=x-iy$ is the complex conjugate.
On the complex side multiplication by $i$
corresponds on the real plane to applying the rotation matrix $j_0$.
Note that
$   \Re(\widebar{i z} \zeta)
   =\Im(\bar z \zeta)
$
and 
$
   \Im(\widebar{i z} \zeta)
   =-\Re(\bar z \zeta)
$.
Observe that
$
   x=(z+\bar z)/2
$
and
$
   y=(z-\bar z)/2i
$.
The one-forms $dz=dx+i\, dy$ and $d\bar z=dx-i\, dy$
reproduce
$
   dx=\frac{dz+d\bar z}{2}
$
and
$
   dy=\frac{dz-d\bar z}{2i}
$.
There are the identities
$$
   d\bar z\wedge dz
{\color{gray}\,\text{\small
   $=(dx+idy)\wedge(dx-idy)$}
}
   =2i\, dx\wedge dy
   ,\qquad
   z\bar z=x^2+y^2=\Abs{z}^2 ,
$$
where $\Abs{\cdot}$ is the Euclidean norm on $\R^2$.
\end{remark}

\boldmath
%%%%%%% Subsubsection:  %%%%%%%%%%%%%%%%
\subsubsection*{Complex squaring map and sign involution
$\mathfrak{i}$}
\unboldmath

In polar coordinates $z=r e^{i\varphi}$ on $\C^\times:=\C\setminus\{0\}$ the
complex squaring map is
$
   \varsigma\colon \C^\times
   \to \C^\times
$,
$
   {\color{gray}re^{i\varphi}=\;} z
   \mapsto
   z^2{\color{gray}\; =r^2e^{i2\varphi}}
$.
Let $0\in\Qfrak\subset\C$ be an open subset containing the origin.
Without $0$ we write
$$
   \Zfrak^\times:=\varsigma^{-1}(\Qfrak^\times)
   =\{z\in\C^\times\mid z^2\in\Qfrak^\times\}
   ,\qquad
   {\color{gray}\varsigma(\Zfrak^\times)
   =\Qfrak^\times:=\Qfrak\setminus\{0\} ,}
$$
where the pre-image $\Zfrak^\times$ is open since $\varsigma$ is
continuous. 
While $\Zfrak^\times$ does not contain the origin, its closure does.
For later use we define
\begin{equation}\label{eq:Zfrak}
   \Zfrak:=\varsigma^{-1}(\Qfrak)=\Zfrak^\times\cup\{0\}\subset \C.
\end{equation}
A useful observation is that $\Zfrak^\times$ is invariant under the
\textbf{sign involution}
\begin{equation}\label{eq:I-Zfrak}
   \mathfrak{i}\colon \Zfrak^\times\to \Zfrak^\times
   ,\quad
   z\mapsto -z
\end{equation}
indeed
$z\in \Zfrak^\times\Leftrightarrow -z\in\Zfrak^\times$
as $z^2=(-z)^2$.
Hence the \textbf{complex squaring~map}
\begin{equation}\label{eq:I-Zfrak-i}
   \varsigma\colon\Zfrak^\times
   \stackrel{2 : 1}{\longrightarrow}
   \Qfrak^\times
   ,\quad
   z\mapsto z^2
   ,\qquad
   \varsigma\circ \mathfrak{i}=\varsigma
   ,
\end{equation}
is a double cover and invariant under sign involution.
By linearization
\begin{equation}\label{eq:TI-Zfrak}
   T\mathfrak{i}\colon T\Zfrak^\times\to T\Zfrak^\times
   ,\quad
   (z,\xi)\mapsto (-z,-\xi)
   .
\end{equation}

\begin{remark}[square root of complex numbers]
\label{rem:cx-square-root}
A non-zero complex number $z=re^{i\phi}$ has two square roots
$
   \sqrt{r} e^{i\phi/2}
$
and
$
   \sqrt{r} e^{i(\phi/2+\pi)}=\sqrt{r} e^{i\phi/2}e^{i\pi}=-\sqrt{r} e^{i\phi/2}
$.
Thus, the square root is only well defined modulo sign.
We need to take square roots in order to define the inverse $\Zz$
of the rescale-square map $\Qq$ later on in~(\ref{eq:Rr-inverse-C}),
as illustrated by Figure~\ref{fig:fig-Qq}.
This is why in~(\ref{eq:involution-quotients}) we will divide by the
sign involution and work on the quotient space.
\end{remark}

\boldmath
%%%%%%% Subsubsection:  %%%%%%%%%%%%%%%%
\subsubsection*{Twisted loop spaces and sign involution $I$}
\unboldmath

Considering twisted loop space allows for defining square roots.
The open subsets $\Zfrak^\times\subset\C\setminus\{0\}$
and $\Zfrak=\Zfrak^\times\cup\{0\}$
are defined by~(\ref{eq:Zfrak}).

\begin{definition}[regularization loop space $\bar\Ll^\times\Zfrak$]
(i)~Both
the $\boldmath\pretop{\pm}\textbf{loop spaces}$
\begin{equation*}%\label{eq:twisted-loop-space}
   \Ll_\pm\Zfrak
   :=\{z\in C^\infty (\R,\Zfrak)\mid\forall t\in\R\colon
   z(t+1)=\pm z(t)\}
   ,
\end{equation*}
for $\Zfrak$ from~(\ref{eq:Zfrak})
are invariant under \textbf{sign involution} 
\begin{equation}\label{eq:inv-I}
   I=-\Id_1\colon \Ll_\pm\Zfrak\to \Ll_\pm\Zfrak
   ,\quad
   z\mapsto -z
   ,\qquad
{\color{gray}
   \Id_k:=\Id_{\C^k}
   ,
}
\end{equation}
because the target $\Zfrak$ is by~(\ref{eq:I-Zfrak}).
For simplicity we call the elements of both spaces still \textbf{loops}.
The intersection of both spaces consists of the zero loop, the unique
fixed point of $I$, in symbols
$
   \Ll_+\Zfrak\cap \Ll_-\Zfrak=\{0\}=\Fix\, I
$.
Taking away the zero loop $z\equiv 0$ we get
the pointed $\pretop{\pm}\text{loop spaces}$
$$
   \Ll_\pm^\times\Zfrak
   :=\Ll_\pm\Zfrak\setminus\{0\}
   ,\qquad
{\color{gray}
   \Ll_+^\times\Zfrak\cap \Ll_-^\times\Zfrak=\emptyset
   .
}
$$
These are disjoint and the sign involution $I$ acts freely on each one
of them, so on their union.
We introduce the quotient spaces (indicated by a bar)
\begin{equation}\label{eq:involution-quotients}
   \bar\Ll_\pm^\times\Zfrak
   :=\frac{\Ll_\pm^\times\Zfrak}{I}
   ,\qquad
{\color{red}
   \bar\Ll^\times\Zfrak
}
   :=\bar\Ll_+^\times\Zfrak\cup
   \bar\Ll_-^\times\Zfrak
   ,
\end{equation}
and call their union $\bar\Ll^\times\Zfrak$
\textbf{regularization loop space}.

The elements of quotient space $\bar\Ll^\times\Zfrak$
are still called \textbf{loops} and denoted by~$z$.
We keep in mind that, in fact, each element has two representatives $\pm z$
and our constructions must be independent of choosing $z$ or $-z$.
The elements $t_*$ of the set $z^{-1}(0)$
are called \textbf{collision times} or simply \textbf{collisions}.

\smallskip
(ii)~The elements $z$ of $\bar\Ll^\times\Zfrak$ that have
no collisions, in symbols $z^{-1}(0)=\emptyset$, form the
\textbf{non-collisional part} of regularization loop space, notation
\begin{equation*}%\label{eq:involution-quotients-times}
{\color{cyan}
   \bar\Ll\Zfrak^\times
}
   =\bar\Ll_+\Zfrak^\times\cup
   \bar\Ll_-\Zfrak^\times
\end{equation*}
with $\bar\Ll_\pm\Zfrak^\times$ defined as earlier, just with
$\Zfrak^\times=\Zfrak\setminus\{0\}$
in place of $\Zfrak$.
\end{definition}

\begin{remark}[tangent bundles]\label{rem:T-bdls}
Both tangent bundles
\begin{equation}\label{eq:TLlZ}
   T\Ll_\pm\Zfrak
   :=\{(z,\xi)\in C^\infty (\R,\Zfrak\times\C)\mid
   \forall t\in\R\colon (z_{\tau+1},\xi_{\tau+1})=\pm (z_\tau,\xi_\tau)\}
\end{equation}
are invariant under \textbf{sign involution}
$$
   TI=-\Id_2
   \colon T\Ll_\pm\Zfrak\to T\Ll_\pm\Zfrak
   ,\quad
   (z,\xi)\mapsto (-z,-\xi)
$$
because the target $\Zfrak\times\C$ is, see~(\ref{eq:I-Zfrak}).
The intersection of both tangent bundles consists of the zero
element $(0,0)$, the unique fixed point of $TI$.
Taking away the fiber over the zero loop we get the two tangent
bundles\footnote{
  The uppercase greek letter $\Xi$ is called ``Xi''.
  }
$$
   T\Ll_\pm^\times\Zfrak
   :=\{\Xi=(z,\xi)\in T\Ll_\pm\Zfrak\colon \norm{z}\not=0\}
   ,\qquad
{\color{gray}
   T\Ll_+^\times\Zfrak\cap T\Ll_-^\times\Zfrak=\emptyset
   .
}
$$
To put it differently
$$
   T\Ll_+^\times\Zfrak
   =\Ll_+^\times\Zfrak \times \Ll_+\C
   ,\qquad
   T\Ll_-^\times\Zfrak
   =\Ll_-^\times\Zfrak \times \Ll_-\C
   .
$$
These are disjoint and the sign involution $TI$ acts freely
on each one of them, so on their union.
We unite the two quotient spaces (indicated by bars)
\begin{equation*}%\label{eq:involution-quotients-*}
   T\bar\Ll_\pm^\times\Zfrak
   :=\frac{T\Ll_\pm^\times\Zfrak}{TI}
   ,\qquad
   T{\color{red}\bar\Ll^\times\Zfrak}
   :=T\bar\Ll_+^\times\Zfrak\cup
   T\bar\Ll_-^\times\Zfrak
   .
\end{equation*}
The elements of $T \bar\Ll^\times\Zfrak$
are still denoted by
\begin{equation}\label{eq:bar-Ll-C}
   \Xi=(z,\xi)\in \Ll_\pm^\times\Zfrak \times\Ll_\pm\C
\end{equation}
actually representing $\pm(z,\xi)$.
This concludes Remark~\ref{rem:T-bdls}.
\end{remark}

\boldmath
%%%%%%%%%%%%%%%%%%%%%%%%%%%%%%%%%%%
%%%%%%% Subsubsection  %%%%%%%%%%%%%%%%
%%%%%%%%%%%%%%%%%%%%%%%%%%%%%%%%%%%
\subsubsection*{Canonical vector field on loop space and induced flow}
\unboldmath

The \textbf{canonical vector field} is generated by time derivative
\begin{equation}\label{eq:can-vf-loops}
   \nu\colon \bar\Ll^\times\Zfrak
   \to T\bar\Ll^\times\Zfrak
   ,\quad
   z\mapsto (z,z^\prime)
\end{equation}
For the principal part we use the same notation $\nu(z)=z^\prime$.
Note that both $\nu$ is well defined on quotients, indeed
$\nu(-z)=(-z,-z^\prime)=-(z,z^\prime)=-\nu(z)$.

\smallskip
The \textbf{flow} induced by $\nu$ on $\bar\Ll^\times\Zfrak$
is time shift, in symbols
\begin{equation}\label{eq:shift-nu}
   \phi^r_\nu z=r_* z
   ,\qquad
   \left(\phi^r_\nu z\right)_\tau=z_{\tau+r} ,
\end{equation}
where $(r_* z)(\tau)=z(\tau+r)$ for every time $\tau\in\R$.
Given $\xi\in T_z \bar\Ll^\times\Zfrak$,
pick a smooth map $h\colon\R\to \bar\Ll^\times\Zfrak$
with $h(0)=z$ and $h^\prime(0)=\xi$. Then
$$
   \bigl(d \phi^r_\nu|_z \xi\bigr)_\tau
   =\bigl(\left.\tfrac{d}{ds}\right|_{s=0}\phi^r_\nu(h(s))\bigr)_\tau
   =\left.\tfrac{d}{ds}\right|_{s=0} h(s)_{\tau+r}
   =\xi_{\tau+r} .
$$

\boldmath
%%%%%%% Subsection  %%%%%%%%%%%%%%%%
\subsection{Rescale-square map $\Qq$ and inverse $\Zz$
(non-collisional)}
\label{sec:resc-square-non-coll}
\unboldmath

In this section we consider loops avoiding the origin (no collisions),
i.e. $z\in {\color{cyan}\bar\Ll\Zfrak^\times}$
and $q\in {\color{cyan}\bar\Ll\Qfrak^\times}$.
The constructions do not depend on choosing $z$ or $-z$.
We construct two maps $\Qq$ and its inverse $\Zz$, as illustrated by
Figure~\ref{fig:fig-Qq}.

\begin{remark}\label{rem:times}
Pick $z\in \bar\Ll\Zfrak^\times$,
so there are no collisions $z^{-1}(0)=\emptyset$,
see Figure~\ref{fig:fig-Qq}.
We call the variable $\tau$ of $z\colon\SS^1\to\Zfrak^\times$,
equivalently $z\colon\R\to\Zfrak^\times$ with $z_{\tau+1}=\pm z_\tau$,
\textbf{regularized time}.
\textbf{Classical time} we call the values of the map
$t_z\colon\SS^1_\tau\to\SS^1_t$
\begin{equation}\label{eq:t_z}
   \forall\tau\in\R\colon\qquad
   t_z(\tau)
   :=\frac{\int_0^\tau \Abs{z(s)}^2\, ds}{\norm{z}^2}
   ,\qquad
   {\color{gray} t_z=t_{-z}}
   .
\end{equation}
Classical time has the following obvious properties
\begin{equation}\label{eq:class-time-tau-deriv-C}
   t_z^\prime(\tau)
   =\frac{\Abs{z(\tau)}^2}{\norm{z}^2}> 0
   ,\quad
   t_z\in C^\infty,\quad t_{\rho z}=t_z,\quad
   t_z(0)=0,\quad t_z(1)=1 ,
\end{equation}
for every real $\rho\not=0$.
Moreover, classical time is indeed equivariant with respect to the
$\Z$-action $\Z\times\R\to\R$, $(k,\tau)\mapsto
\tau+k$, in symbols
$
   t_z(\tau+k)=t_z(\tau)+ k
$.
As the map $t_z\colon\SS^1_\tau\to\SS^1_t$ is continuous and strictly
monotone increasing (as $z$ avoids the origin),
it has a continuous strictly
monotone increasing inverse\footnote{
  see e.g.~\cite[\S12 Satz\,1]{Forster:2011a}
  }
\begin{equation}\label{eq:t-inverse}
   \tau_z:={t_z}^{-1}
   \colon\SS^1_t\to\SS^1_\tau
   ,\quad
   t\mapsto
   \tau_z(t)
\end{equation}
called \textbf{regularized time}.
Since $z$ avoids the origin $\dot \tau_z$ is continuous. Indeed
\begin{equation}\label{eq:inverse-C}
   \dot \tau_z(t)
   =\tfrac{1}{t_z^\prime(\tau_z(t))}
   =\tfrac{\norm{z}^2}{\Abs{z(\tau_z(t))}^2}>0
   ,\quad
   \tau_z\in C^\infty,\quad
   \tau_z(0)=0,\quad \tau_z(1)=1
   ,
\end{equation}
and $\tau_{\rho z}=\tau_z$ $\forall \rho>0$.
In particular, both $t_z$ and $\tau_z$ are \underline{smooth}
diffeomorphisms of $\SS^1$.
This is due to the fact that $z$ \underline{avoids} the origin.
This concludes Remark~\ref{rem:times}.
\end{remark}

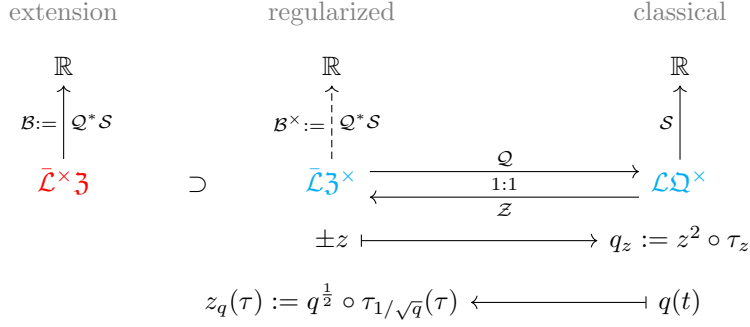
\begin{figure}%[b]
  \centering
\begin{equation*}
\begin{tikzcd} [row sep=tiny] %[column sep=small] %
{\color{gray} \text{extension}}
  &{\color{gray} \text{regularized}}
  &&
    {\color{gray} \text{classical}}
\\
\R
  &\R
  &&\R
\\ \mbox{} \\ \mbox{} \\
{\color{red}
   \bar\Ll^\times\Zfrak
}
\arrow[uuu, "\Bb:=", "\Qq^*\Ss"']
\arrow[r, phantom, "\supset"]
  &
{\color{cyan}
   \bar\Ll\Zfrak^\times
}
    \arrow[rr, shift left=2, "\Qq", "1:1"']
    \arrow[uuu, dashed, "\Bb^\times:=", "\Qq^*\Ss"']
  &&{\color{cyan}\Ll \Qfrak^\times}
      \arrow[uuu, "\Ss"]
      \arrow[ll, shift left=2, "\Zz"]
\\
  &\pm z
    \arrow[rr, mapsto]
  &&
    q_z:=z^2\circ \tau_z
\\
  &z_q(\tau):=q^\frac12\circ \tau_{1/\sqrt{q}}(\tau)
  &&
    q(t)
    \arrow[ll, mapsto]
\end{tikzcd} 
\end{equation*}
  \caption{Pull-back of classical action gives a formula
                 $\Qq^*\Ss$ on $\bar{\Ll}\Zfrak^\times$
                 {\small\color{gray} (loops avoiding $0$)}
                 that makes sense on the larger space
                 $\bar\Ll^\times\Zfrak$
                 {\small\color{gray} (all loops except zero loop $\equiv 0$)}
                 on which $\Bb:=\Qq^*\Ss$ has extra critical points
                 (rescaled classical collision trajectories)}
   \label{fig:fig-Qq}
\end{figure}

\begin{definition}\label{def:reg-path-C}
The \textbf{rescale-square operation}
is defined by
\begin{equation*}%\label{eq:q_z}
\begin{split}
   \Qq\colon
{\color{cyan}
   \bar\Ll\Zfrak^\times
}
   \to
{\color{cyan}
   \Ll\Qfrak^\times
}
   ,\quad
   z
   \mapsto
   q_z:=z^2\circ \tau_z
   {\color{gray}\;=\varsigma\circ z\circ \tau_z,
   \quad
   \forall \rho\not=0\colon \Qq(\rho z)=\rho^2\Qq(z)}
   .
\end{split}
\end{equation*}
\end{definition}

\begin{remark}
By definition of $q_z$ and the chain rule we get the identity
\begin{equation}\label{eq:dot-q_z}
   \dot q_z(t)
   =2z(\tau_z(t))\, z^\prime(\tau_z(t))\, \dot\tau_z(t)
   \stackrel{(\ref{eq:inverse-C})}{=}
   2\norm{z}^2\,\tfrac{z^\prime(\tau_z(t))}{\bar z(\tau_z(t))}
   .
\end{equation}
In equality~2 below
we change the variable to $\sigma:=\tau_z(t)$,
then use~(\ref{eq:class-time-tau-deriv-C}) to get
\begin{equation}\label{eq:fsfcs-C}
\begin{split}
   \norm{\dot q_z}^2
   &\stackrel{\text{(\ref{eq:dot-q_z})}}{=}
   4\norm{z}^4 \int_0^1
     \frac{\abs{z^\prime(\tau_z(t))}^2}{\abs{\bar z(\tau_z(t))}^2} dt
   \stackrel{2}{=} 4\norm{z}^4 \int_0^1
     \frac{\abs{z^\prime(\sigma)}^2}{\abs{z(\sigma)}^2}
     t_z^\prime(\sigma)\, d\sigma
\\
   &\stackrel{\text{(\ref{eq:class-time-tau-deriv-C})}}{=}
   4\norm{z}^4 \int_0^1
     \frac{\abs{z^\prime(\sigma)}^2}{\abs{z(\sigma)}^2}
     \frac{\Abs{z(\sigma)}^2}{\norm{z}^2}
     \, d\sigma
   =
   4\norm{z}^2\INNER{z^\prime}{z^\prime}
   .
\end{split}
\end{equation}
\end{remark}

\begin{lemma}[well defined bijection]\label{le:bij-proj}\label{le:Qq}
For $z\in \bar\Ll\Zfrak^\times$ the image $\Qq(z)$ lies in
$\Ll\Qfrak^\times$.
The map $\Qq\colon\bar\Ll\Zfrak^\times\to \Ll\Qfrak^\times$
is a bijection with inverse~(\ref{eq:Rr-inverse-C}).
\end{lemma}

\begin{proof}
Pick $z\in\bar\Ll\Zfrak^\times$, in particular
$z\colon\SS^1\to\Zfrak^\times$.
Since $z$ and $\tau_z$ are smooth, so is their composition $q_z$.
Non-vanishing of $z$ implies non-vanishing of $q_z$.
To see that $q_z$ is $1$-periodic note that
$q_z(1)=z^2\circ \tau_z(1)=z(1)^2=(\pm z(0))^2=
z(0)^2=z^2\circ t_z(0)=q_z(0)$.
This shows that
$\Qq(z)=q_z\in \Ll\Qfrak^\times=C^\infty(\SS^1,\Qfrak\setminus\{0\})$.

\medskip\noindent
\textbf{Surjective.}
Given $q\in\bar\Ll\Qfrak^\times$,
define $1/\sqrt{q}\in\bar \Ll\C^\times$
and $\Zz(q)\in  \bar\Ll\Zfrak^\times$
by
\begin{equation}\label{eq:Rr-inverse-C}
   \Zz(q):=z_q:=q^\frac12\circ \tau_{1/\sqrt{q}}
   ,\qquad
   1/\sqrt{q}=[t\mapsto 1/q(t)^{\frac12}] ,
\end{equation}
where $\tau_{1/\sqrt{q}}={t_{1/\sqrt{q}}}^{-1}\colon\SS^1_\tau\to\SS^1_t$
is the inverse of the map $t_{1/\sqrt{q}}$ associated
by~(\ref{eq:t_z})
to the loop $1/\sqrt{q}$ in $\C^\times$.
For $\tau\in\R$ we calculate the identities
\begin{equation}\label{eq:rescaling-sqrt-q}
\begin{split}
   t_{z_q}(\tau)
   &\stackrel{\text{def.}}{=}
   \frac{\int_0^\tau\abs{{z_q}^2(\sigma)}\, d\sigma}{\int_0^1\abs{{z_q}^2(\sigma)}\, d\sigma}
   =\frac{\int_0^\tau\abs{q\circ\overbrace{\tau_{1/\sqrt{q}}(\sigma)}^{=:s}}
      \, d\sigma}
      {\int_0^1\abs{q\circ \tau_{1/\sqrt{q}}(\sigma)}\, d\sigma}
   =\frac{\int_0^{\tau_{1/\sqrt{q}}(\tau)} \frac{ds}{\norm{1/\sqrt{q}}^2}}
      {\int_0^1 \frac{ds}{\norm{1/\sqrt{q}}^2}}
\\
   &=\tau_{1/\sqrt{q}}(\tau)
\end{split}
\end{equation}
using change of variables $s(\sigma)=\tau_{1/\sqrt{q}}(\sigma)$, equivalently
$\sigma(s)=t_{1/\sqrt{q}}(s)$, hence
$$
   d\sigma
   =t_{1/\sqrt{q}}^\prime(s) \,ds
   \stackrel{\text{(\ref{eq:class-time-tau-deriv-C})}}{=}
   \tfrac{\abs{1/\sqrt{q}(s)}^2}{\norm{1/\sqrt{q}}^2} \, ds
   =\tfrac{1/\abs{q_s}}{\norm{1/\sqrt{q}}^2} \, ds
   .
$$
With this result we obtain that
\[
   \bigl(\Qq\circ\underbrace{\Zz(q)}_{z_q}\bigr)(t)
   \stackrel{\text{def.\,$\Qq$}}{=}
   {z_q}^2\circ \tau_{z_q}(t)
   \stackrel{\text{def.\,$z_q^2$}}{=}
   q\circ\underbrace{\tau_{1/\sqrt{q}}}_{t_{z_q}}\circ
   \underbrace{\tau_{z_q}}_{{t_{z_q}}^{-1}}(t)
   =q(t) .
\]

\medskip\noindent
\textbf{Injective.}
For $z\in\bar\Ll\Zfrak^\times$ set
$q_z:=\Qq(z):=z^2\circ \tau_z$. Then for $t\in\R$ we get
\[
   \int_0^t\frac{1}{\abs{q_z(s)}}\,ds
   =\int_0^t\frac{1}{\abs{z\circ \tau_z(s)}^2}\,ds
   =\int_0^{\tau_z(t)}\frac{1}{\abs{z(\sigma)}^2}
   \frac{\abs{z(\sigma)}^2\, d\sigma}{\norm{z}^2}
   =\frac{\tau _z(t)}{\norm{z}^2}
\]
by change of variables $\sigma=\tau_z(s)$ using~(\ref{eq:inverse-C}).
Pick $t=1$ to obtain
\begin{equation}\label{eq:1/q^2-C}
   \int_0^1\frac{1}{\abs{q_z(t)}}\, dt
   =\frac{1}{\norm{z}^2} .
\end{equation}
Therefore we get for $t\in\R$ the formula
\begin{equation}\label{eq:tau_z=t1/sqrtq}
   \tau_z(t)
   =\norm{z}^2 \int_0^t\frac{1}{\abs{q_z(s)}}\,ds
   \stackrel{(\ref{eq:1/q^2-C})}{=}
      \frac{\int_0^t\frac{1}{\abs{q_z(s)}}\,ds}
      {\int_0^1\frac{1}{\abs{q_z(s)}}\,ds}
   \stackrel{(\ref{eq:t_z})}{=}
   t_{1/\sqrt{q_z}} (t) .
\end{equation}
With this result we obtain
\[
   \bigl(\Zz\circ\underbrace{\Qq(z)}_{q_z}\bigr)(\tau)
   \stackrel{\text{def.\,$\Zz$}}{=}
      {q_z}^{\frac12}\circ \tau_{1/\sqrt{q_z}}(\tau)
   \stackrel{\text{def.\,$q_z$}}{=}
      z\circ\underbrace{\tau_z}_{{t_{1/\sqrt{q_z}}}}\circ
      \underbrace{\tau_{1/\sqrt{q_z}}}_{{t_{1/\sqrt{q_z}}}^{-1}}  (\tau)
   =z(\tau) .
\]
This proves Lemma~\ref{le:bij-proj}.
\end{proof}

\begin{proposition}[extendable formula]\label{prop:Bb}
For $\Ss$ in~(\ref {eq:class-action}) and
$z\in {\color{cyan}\bar\Ll\Zfrak^\times}$ it holds
\begin{equation}\label{eq:Bb}
   \Ss (\Qq(z))
   =2\norm{z}^2\norm{z^\prime}^2
   +\frac{1}{\norm{z}^2}
   +
   \int_0^1(\varsigma^*\theta)_{t_z(\sigma)}|_{z(\sigma)}
      z^\prime(\sigma)\, d\sigma
    -f(z^2_0)
\end{equation}
as illustrated by Figure~\ref{fig:fig-Qq}.
The functional $\Qq^*\Ss$ has the same value on $\pm z$.
\end{proposition}

\noindent
\textbf{Note.} Formula~(\ref{eq:Bb}) makes sense even if $z$ takes on
the value $0$, even along intervals, as long as $\norm{z}\not=0$,
i.e. as long as $z$ is not constantly zero.

\begin{proof}
Set $q_z:=\Qq(z)=\varsigma\circ z\circ \tau_z$.
By definition~(\ref{eq:class-action}) of $\Ss$ we have
\begin{equation*}%\label{eq:Ss=Bb-C}
\begin{split}
   \Ss (q_z):
   &= 
   \tfrac12\norm{\dot q_z}^2
   +\int_{[0,1]} q_z^*\theta
   +\int_0^1\frac{1}{\Abs{q_z(t)}}\, dt
    -f(q_z(0))
\\
   &\stackrel{{\color{white}(2.3)}}{=}
   \tfrac12\, 4\norm{z}^2\INNER{z^\prime}{z^\prime}
   +\int_{[0,1]} q_z^* \theta
   +\frac{1}{\norm{z}^2}
   -f(z^2_0)
   .
\end{split}
\end{equation*}
Here equality two is by~(\ref{eq:fsfcs-C})
and~(\ref{eq:1/q^2-C}). We compute the last two summands
\begin{equation}\label{eq:pull-back-of-theta}
\begin{split}
   &\int_{[0,1]} q_z^*\theta-f(q_z(0))
\\
%1
   &=\int_0^1 \theta_t|_{q_z(t)}\dot q_z(t)\, dt-f(q_z(0))
\\
%2
   &\stackrel{{\color{gray}2}}{=}
   2\norm{z}^2\int_0^1 \theta_{t_z(\sigma)}|_{z^2_\sigma}
   \tfrac{z^\prime_\sigma}{\bar z_\sigma} t_z^\prime(\sigma)\, d\sigma
   -f(z_0^2)
\\
%3
   &\stackrel{{\color{gray}3}}{=}
   2\int_0^1 \theta_{t_z(\sigma)}|_{z^2(\sigma)}
   z_\sigma  z^\prime_\sigma \, d\sigma
   -f(z_0^2)
\\
%4
   &\stackrel{{\color{gray}}}{=}
   2\int_0^1 \theta_{t_z(\sigma)}|_{{z_\sigma}^2}
   \left(x+iy\right)_\sigma
   \left(x^\prime+iy^\prime\right)_\sigma
   d\sigma
   -f(z_0^2)
\\
%5
   &\stackrel{{\color{gray}}}{=}
   2\int_0^1 \theta_{t_z(\sigma)}|_{{z_\sigma}^2}
   \Bigl(
   \left(x x^\prime-y y^\prime\right)_\sigma
   +i \left(x y^\prime-y x^\prime\right)_\sigma
   \Bigr)
   d\sigma
   -f(z_0^2)
\\
\end{split}
\end{equation}
\begin{equation*}
\begin{split}
   &\stackrel{{\color{gray}6}}{=}
   2\int_0^1
   \Bigl(
   A^1_{t_z(\sigma)}|_{{z_\sigma}^2}
   \left(x x^\prime-y y^\prime\right)_\sigma
   +
   A^2_{t_z(\sigma)}|_{{z_\sigma}^2}
   \left(x y^\prime+y x^\prime\right)_\sigma
   \Bigr)
   d\sigma
   -f(z_0^2)
\\
%7
   &\stackrel{{\color{gray}7}}{=}
   2\int_0^1
   \Bigl(
   \left(x A^1_{t_z}|_{z^2}+y A^2_{t_z}|_{z^2}\right) x^\prime
   +
   \left(x A^2_{t_z}|_{z^2}-y A^1_{t_z}|_{z^2}\right) y^\prime
   \Bigr)_{\sigma}
   d\sigma
   -f(z_0^2)
\\
%8
   &\stackrel{{\color{gray}8}}{=}
   \int_0^1(\varsigma^*\theta)_{t_z(\sigma)}|_{z_\sigma}
   z^\prime_\sigma\, d\sigma
   -f(z_0^2)
%9
   \stackrel{\text{(\ref{eq:Mm-BOV})}}{=:} \Mm(z)
\end{split}
\end{equation*}
where we identified $\C\simeq\R^2$
via $z=x+iy\mapsto (x,y)$.
Equality 2 is by~(\ref{eq:dot-q_z})
in combination with variable substitution
$\sigma=\tau_z(t)$ and the identity $t=t_z(\sigma)$
which uses the inverse function $\tau_z(t)=t_z^{-1}(t)$.
Equality 3 uses that $t_z^\prime(\tau)$ is given
by~(\ref{eq:class-time-tau-deriv-C}).
Equality 6 uses that $\theta$ is of the form~(\ref{eq:tw-per-1-form}).
Equality 8 is by definition~(\ref{eq:I-Zfrak-i}) of $\varsigma$ and of
the pull-back by $\varsigma$.
The result we denote by $\Mm(z)$.

\smallskip\noindent
$\Mm$ has the same value on $\pm z$:
By~(\ref{eq:pull-back-of-theta}) it suffices to show $q_z=q_{-z}$.
Indeed $q_z=z^2\circ {t_z}^{-1}=(-z)^2\circ {t_{-z}}^{-1}= q_{-z}$
for $t_z=t_{-z}$ see~(\ref{eq:t_z}).
Alternatively, inspect the right hand side of equality~3
in~(\ref{eq:pull-back-of-theta}).
This proves Proposition~\ref{prop:Bb}.
\end{proof}

\boldmath
%%%%%%%%%%%%%%%%%%%%%%%%%%%%%%%%%%%
%%%%%%%%%%%%%%%%%%%%%%%%%%%%%%%%%%%
%%%%%%% Section  %%%%%%%%%%%%%%%%%%%%%%
%%%%%%%%%%%%%%%%%%%%%%%%%%%%%%%%%%%
%%%%%%%%%%%%%%%%%%%%%%%%%%%%%%%%%%%
\section[Non-local Lagrangian mechanics -- twisted-periodic case]
{Non-local Lagrangian mechanics}
\label{sec:non-loc-Lag-mech}
\unboldmath

\boldmath
%%%%%%%%%%%%%%%%%%%%%%%%%%%%%%%%%%%
%%%%%%% Subsection  %%%%%%%%%%%%%%%%%%%
%%%%%%%%%%%%%%%%%%%%%%%%%%%%%%%%%%%
\subsection{Magnetic functional $\Mm$}
\label{sec:magnetic-Mm}
\unboldmath

Fix a twisted-periodic $1$-form~$\theta$
and a twist function $f$ on $\Qfrak\subset\C$;
see~(\ref{eq:df_t}).
Formula~(\ref{eq:Bb}) for the pull-back
of the classical action functional motivates

\begin{definition}[magnetic functional]\label{def:Mm}
Define the \textbf{magnetic functional}
by
\begin{equation}\label{eq:Mm-BOV}
   \Mm\colon {\color{red}\bar\Ll^\times\Zfrak}\to\R
   ,\quad
   z\mapsto
   \int_0^1 (\varsigma^*\theta)_{t_z(\tau)}|_{z_\tau}
   z_\tau^\prime\, d\tau
    -f(z_0^2)
   ,
\text{\color{gray}\small\;\,
   $f(z_0^2)=(\varsigma^* f)|_{z_0}$
}
\end{equation}
with repara\-metrization $t_z\colon\SS^1_\tau\to\SS^1_t$
defined by~(\ref{eq:t_z}) 
and $\varsigma(z)=z^2$.
For $\Mm$ written as $L^2$-inner product
see~(\ref{eq:Mm-as-inner-product}).
That $\Mm(z)=\Mm(-z)$ one reads off immediately
from formula~(\ref{eq:pull-back-of-theta})$_7$.
Remark~\ref{rem:interval-not-SS^1} explains
the integration interval.
\end{definition}

\begin{remark}[pull-back $\vartheta:=\varsigma^*\theta$]
\label{rem:pull-back}
The pull-back of $\theta$ from the open subset
$\Qfrak\subset\C$ with coordinates $q=q_1+iq_2$ to
$\Zfrak=\varsigma^{-1}(\Qfrak)\subset\C$ with coordinates $z=x+iy$ 
under the map $z\mapsto \varsigma(z):=z^2$
is a smooth real family of $1$-forms
\begin{equation}\label{eq:tw-per-1-form-var}
   \vartheta_t:=\varsigma_*\theta_t
   = a_t^1\, dx+ a_t^2\, dy
   ,\qquad
   a_t^1,  a_t^2\colon\Zfrak\to\R
   ,\qquad
   \aaa_t=( a_t^1,  a_t^2) .
\end{equation}
Analogous to~(\ref{eq:d-theta}) we get
$d\vartheta_t=(\rot\,\aaa_t)\,\omega_0$.
Note that $\vartheta$ is twisted-periodic
\begin{equation*}%\label{eq:vartheta}
\begin{split}
   \dot\vartheta_{t+1}
   &
{\color{gray}
   =\varsigma^*\dot\theta_{t+1}=\varsigma^*\dot\theta_t
}
   =\dot\vartheta_t
\\
   \vartheta_{t+1}
   &
{\color{gray}
   =\varsigma^*\theta_{t+1}=\varsigma^*(\theta_t+df)
}
   =\vartheta_t+dF
   ,\qquad
   F:=\varsigma^*f=f\circ\varsigma
   .
\end{split}
\end{equation*}
To compute the pull-back under the squaring
map $\varsigma$ we write
$$
   q_1+iq_2=q=z^2=(x+iy)^2)=x^2-y^2+2ixy .
$$
Thus the components are
$
   q_1=x^2-y^2
$
and
$
   q_2=2xy
$.
Hence
$$
   dq_1=2x dx-2y dy
   ,\qquad
   dq_2=2ydx+2xdy .
$$
Recall from~(\ref{eq:tw-per-1-form})
that we wrote $\theta_t$ in the form
$\theta_t|_q=A_t^1|_q dq_1+A_t^2|_qdq_2$ for $q\in\Qfrak$.
Calculation shows that the pull-back under $\varsigma$
is of the form
\begin{equation*}
\begin{split}
   \vartheta_t|_z:
   &=\left(\varsigma^*\theta\right)_t|_z
\\
   &=A_t^1|_{\varsigma(z)}(2x dx-2y dy)+A_t^2|_{\varsigma(z)}(2ydx+2xdy)
\\
   &=\left(2xA_t^1|_{z^2}+2yA_t^2|_{z^2}\right) dx
   +\left(2xA_t^2|_{z^2}-2yA_t^1|_{z^2}\right) dy.
\end{split}
\end{equation*}
By~(\ref{eq:tw-per-1-form-var}) the components $ a_t^j|_z$
of $\vartheta$ and the $A_t^j|_{z^2}$ of $\theta$ are related by
\begin{equation}\label{eq:A-tilde-A}
\begin{split}
    a_t^1|_z
   &=2xA_t^1|_{z^2}+2yA_t^2|_{z^2}
{\color{gray}\;
   =2\INNER{z}{\AAA_t|_{z^2}}_0
}
\\
    a_t^2|_z
   &=2xA_t^2|_{z^2}-2yA_t^1|_{z^2}
{\color{gray}\;
   =2\omega_0(z,\AAA_t|_{z^2})
   .
}
\end{split}
\end{equation}
Pointwise at $z=(x,y)$, using the previous two formulas, we calculate
\begin{equation}\label{eq:tilde-A-A}
   \rot\,\aaa_t|_z
   =(\p_x a_t^2-\p_y a_t^1)_z
   =4\Abs{z}^2(\p_1 A^2_t-\p_2 A^1_t)_{z^2}
   =4\bar z z\, \rot\,\AAA_t|_{z^2}
   .
\end{equation}
In particular, we see that $\rot\,\aaa_t|_0=0$
vanishes at the origin singularity at
all times. 
The second equality requires a bit of work.
Hence
\begin{equation}\label{eq:dvartheta}
   (d\vartheta_t)_z
   =(\rot\,\aaa_t)\,\INNER{j_0\cdot}{\cdot}_0
   =4\Abs{z}^2(\rot\,\AAA_t|_{z^2})\, \INNER{j_0\cdot}{\cdot}_0 .
\end{equation}
In complex notation there are the identities
\begin{equation}\label{eq:barzA}
   2\bar z \AAA_t|_{z^2}
   =2(x-iy)(A^1_t+iA^2_t)|_{z^2}
   = a_t^1|_z+i  a_t^2|_z
   =\aaa_t|_z
   .
\end{equation}
Step two uses~(\ref{eq:A-tilde-A});
similarly $2\bar z \AAAdot_t|_{z^2}=\dot{\aaa}_t|_z$.
This concludes Remark~\ref{rem:pull-back}.
\end{remark}

By~(\ref{eq:Mm-BOV}) and~(\ref{eq:tw-per-1-form-var})
the magnetic term translates to an inner product
\begin{equation}\label{eq:Mm-as-inner-product}
\begin{split}
   \Mm(z)
   &=\int_0^1 \left( a^1_{t_z(\tau)}|_{z_\tau} x^\prime_\tau
   + a^2_{t_z(\tau)}|_{z_\tau} y^\prime_\tau \right) d\tau
   -f(z_0^2)
   =\inner{\aaa_{t_z}|_z}{z^\prime}
   -f(z_0^2)
\end{split}
\end{equation}
for $z\in \bar\Ll^\times\Zfrak$.
The task at hand is to calculate the derivative $d\Mm(z)$.

\boldmath
%%%%%%%%%%%%%%%%%%%%%%%%%%%%%%%%%%%
%%%%%%% Subsubsection  %%%%%%%%%%%%%%%%
%%%%%%%%%%%%%%%%%%%%%%%%%%%%%%%%%%%
\subsubsection[Gradient]{$L^2$-gradient}
\unboldmath

\begin{lemma}[$L^2$-gradient of $\Mm$]\label{le:gradient-Mm}
At $z\in \bar\Ll^\times\Zfrak$ for any time $\tau$ we have
\begin{equation*}%\label{eq:grad-Mm}
\begin{split}
   (\grad\Mm|_z)_\tau
   &=
% T3 magnetic M2
   -\frac{2z_\tau}{\norm{z}^4}
   \int_0^1
{\textstyle
   \int_0^\sigma\Abs{z_\rho}^2 d\rho
}
   \cdot
   \inner{\dot{\aaa}_{t_z(\sigma)}|_{z_\sigma}}{z^\prime_\sigma}_0\; d\sigma
% T5 magnetic M3
   -\frac{\Abs{z_\tau}^2}{\norm{z}^2}\dot{\aaa}_{t_z(\tau)}|_{z_\tau}
   \\
   &\quad
% T4 magnetic M1
   +\frac{2 z_\tau}{\norm{z}^2}
   \int_{\sigma=\tau}^1
   \inner{\dot{\aaa}_{t_z(\sigma)}|_{z_\sigma}}{z^\prime_\sigma}_0\;
   d\sigma
% T6 magnetic M4
   -(\rot\,\aaa_{t_z}|_z)_\tau\; j_0 z^\prime_\tau
   .
\end{split}
\end{equation*}
\end{lemma}

\begin{proof}
We use the representation~(\ref{eq:Mm-as-inner-product}) of $\Mm$ as
$L^2$-inner product to calculate the $L^2$-gradient of $\Mm$ at a loop $z$.
Given a smooth vector field $\xi\in T_z\bar\Ll^\times\Zfrak$
along $z$, pick a smooth path of loops $\eps\mapsto z_\eps$
with $z_0=z$ and $\left.\tfrac{d}{d\eps}\right|_0 z_\eps=\xi$.
Then
\begin{equation}\label{eq:dt(z)}
\begin{split}
   (dt|_z\xi)_\tau :
   =\left.\tfrac{d}{d\eps}\right|_{\eps=0} t(z_\eps)_\tau
   &=
   \frac{2}{\norm{z}^2}\int_0^\tau\INNER{z_\sigma}{\xi_\sigma}_0 d\sigma
   -\frac{2\INNER{z}{\xi}}{\norm{z}^4}\int_0^\tau\Abs{z_\sigma}^2 d\sigma
   .
\end{split}
\end{equation}
Another preparation for the main calculation is to determine the difference
\begin{equation}\label{eq:problem}
\begin{split}
   \inner{\aaa_1|_{z_1}}{\xi_1}_0
   -\inner{\aaa_0|_{z_0}}{\xi_0}_0
%1
   &=
   (\varsigma^*\theta)_1|_{z_1}\xi_1
   -(\varsigma^*\theta)_0|_{z_0}\xi_0
\\
%2
   &
{\color{brown}\,
   =
   (\mathfrak{i}^*\varsigma^*\theta)_1|_{-z_0}(-\xi_0)
   -(\varsigma^*\theta)_0|_{z_0}\xi_0
}
\\
%3
   &=
   (\varsigma^*\theta)_1|_{z_0} \xi_0
   -(\varsigma^*\theta)_0|_{z_0}\xi_0
\\
%4
   &=
   (\varsigma^*\theta_0+\varsigma^*df)|_{z_0} \xi_0
   -(\varsigma^*\theta)_0|_{z_0}\xi_0
\\
%5
   &=d(\varsigma^*f)|_{z_0} \xi_0
   .
\end{split}
\end{equation}
{\color{brown}Equality~2} is only for the twisted case, skip it in the
untwisted case.
Equality~2 uses $\varsigma^*=\mathfrak{i}^*\varsigma^*$,
by~(\ref{eq:I-Zfrak-i}), and $(z_1,\xi_1)=-(z_0,\xi_0)$ by the twist
hypothesis.
Equality~3 uses $Ti(-z_0,-\xi_0)=-(-z_0,-\xi_0)$, by~(\ref{eq:TI-Zfrak}).
Equality 4 is by twisted-periodicity of~$\theta$ and linearity of pull-back.
Equality 5 uses that pull-back and $d$ commute and, moreover, two
terms cancel each other.

\smallskip\noindent
The principal calculation is as follows
\begin{equation*}\label{eq:hgjgj77}
\begin{split}
   &\INNER{\grad\Mm|_z}{\xi}
%1
   =d\Mm|_z\xi
%2
   =\left.\tfrac{d}{d\eps}\right|_{\eps=0} \Mm(z_\eps)
\\
%3
   &=\left.\tfrac{d}{d\eps}\right|_{\eps=0}
   \INNER{\aaa_{t_{z_\eps}}|_{z_\eps}}{z_\eps^\prime}
   -\left.\tfrac{d}{d\eps}\right|_{\eps=0}
   f\circ\varsigma(z_\eps(0))
\\
%4
   &=\INNER{\left.\tfrac{d}{d\eps}\right|_{\eps=0}\aaa_{t_{z_\eps}}|_{z_\eps}}{z^\prime}
   +\INNER{\aaa_{t_z}|_z}{\xi ^\prime}
   -d(\varsigma^*f)|_{z_0}\xi_0
\\
%5
   &\stackrel{{\color{gray}5}}{=}
{\small
   \INNER{\left.\tfrac{d}{d\eps}\right|_{\eps=0}\aaa_{t_{z_\eps}}|_{z_\eps}}{z^\prime}
   +
\underline{\inner{\aaa_1|_{z_1}}{\xi_1}_0
   -\inner{\aaa_0|_{z_0}}{\xi_0}_0
}
   -\INNER{\left(\aaa_{t_z}|_z\right)^\prime}{\xi}
   -
\underline{ 
   d(\varsigma^*f)|_{z_0}\xi_0
}
}
\\
%6
   &\stackrel{{\color{gray}6}}{=}
   \INNER{\dot\aaa_{t_z}|_z \left(dt|_z\xi\right)+d\aaa_{t_z}|_z\xi}{z^\prime}
   -\INNER{\dot\aaa_{t_z}|_z (t_z)^\prime+d\aaa_{t_z}|_z z^\prime}{\xi}
\\
%7
   &\stackrel{{\color{gray}7}}{=}
{\small
   \inner{\dot\aaa_{t_z}|_z\left(dt|_z\xi\right)}{z^\prime}
   -\inner{\dot\aaa_{t_z}|_z (t_z)^\prime}{\xi}
   +\inner{(d\aaa_{t_z}|_z^T-d\aaa_{t_z}|_z) z^\prime}{\xi}
   .
}
\end{split}
\end{equation*}
Identity~5 uses integration by parts.
Identity~6 is by the chain and product rules,
the underlined terms cancel due to~(\ref{eq:problem}).
Identity~7 moves $d\aaa$ to the other side of the inner product
where it arrives as transpose $d\aaa^T$.

Now we are in position to identify the three resulting summands
with the summands in Lemma~\ref{le:gradient-Mm}.
Both second summands are equal by
formula~(\ref{eq:class-time-tau-deriv-C}) for $(t_z)^\prime$.
Abbreviating $\aaa=\aaa_{t_z}|_z$ both final summands
correspond, since
\begin{equation*}
\begin{split}
   (d\aaa)^T-d\aaa
   &=
   \begin{pmatrix}
      \p_1\aaa_1&\p_1\aaa_2\\
      \p_2\aaa_1&\p_2\aaa_2
   \end{pmatrix}
   -
   \begin{pmatrix}
      \p_1\aaa_1&\p_2\aaa_1\\
      \p_1\aaa_2&\p_2\aaa_2
   \end{pmatrix}
   =-\left(\rot\,\aaa\right) j_0
   .
\end{split}
\end{equation*}
It remains to calculate the first summand
\begin{equation*}
\begin{split}
   \INNER{\dot\aaa_{t_z}|_z\left(dt|_z\xi\right)}{z^\prime}
   &=\int_0^1
   \INNER{\dot\aaa_{t_z(\tau)}|_{z_\tau}\,(dt|_z\xi)_\tau}{z_\tau^\prime}_0\, d\tau
\\
   &\stackrel{{\color{gray}1}}{=}
   \frac{2}{\norm{z}^2}
   \int_0^1
   {\textstyle\int_0^\tau\INNER{z_\sigma}{\xi_\sigma}_0 d\sigma}
   \cdot
   \INNER{\dot\aaa_{t_z(\tau)}|_{z_\tau}}{z_\tau^\prime}_0\, d\tau
\\
   &\quad
   -
   \frac{2\INNER{z}{\xi}}{\norm{z}^4}
   \int_0^1
   {\textstyle\int_0^\tau\Abs{z_\sigma}^2 d\sigma}
   \INNER{\dot\aaa_{t_z(\tau)}|_{z_\tau}}{z_\tau^\prime}_0\, d\tau
\\
   &\stackrel{{\color{gray}2}}{=}
   \frac{2}{\norm{z}^2}\int_{\tau=0}^1\INNER{z_\tau}{\xi_\tau}_0
   \int_{\sigma=\tau}^1\INNER{\dot\aaa_{t_z(\sigma)}|_{z_\sigma}}{z_\sigma^\prime}_0
   d\sigma\, d\tau
\\
   &\quad
   -
   \frac{2\INNER{z}{\xi}}{\norm{z}^4}
   \int_0^1
   {\textstyle\int_0^\tau\Abs{z_\sigma}^2 d\sigma}
   \INNER{\dot\aaa_{t_z(\tau)}|_{z_\tau}}{z_\tau^\prime}_0\, d\tau
   .
\end{split}
\end{equation*}
Identity~1 spells out the $L^2$-inner product.
Identity~2 inserts formula~(\ref{eq:dt(z)}) for $dt_z\xi$
and pulls real factors out of the integral.
In identity three we interchanged the order of integration
using Fubini's theorem.
Take the $L^2$-inner product of
the identity in Lemma~\ref{le:gradient-Mm}
with $\xi$ to see equality.
This proves Lemma~\ref{le:gradient-Mm}.
\end{proof}

In~(\ref{eq:dMm-BOV}) we calculate $d\Mm$, equivalently $\grad \Mm$,
applying Cartan's formula from finite dimensions formally on the loop
space.\footnote{
  The final summand in Lemma~\ref{le:gradient-Mm}
  arises from the final summand in~(\ref{eq:dMm-BOV}),
  via the identity
  $
     -(d\vartheta_{t_z(\tau)})|_{z_\tau}(z_\tau^\prime,\xi_\tau)
     =-\bigl(\rot\,\aaa_{t_z(\tau)}|_{z_\tau}\bigr)
     \INNER{j_0 z_\tau^\prime}{\xi_\tau}_0
  $
  in~(\ref{eq:dvartheta}).
  }
This yields the same formula as the rigorously proved
Lemma~\ref{le:gradient-Mm} above.

\boldmath
%%%%%%%%%%%%%%%%%%%%%%%%%%%%%%%%%%%
%%%%%%% Subsubsection  %%%%%%%%%%%%%%%%
%%%%%%%%%%%%%%%%%%%%%%%%%%%%%%%%%%%
\subsection{Lagrangian action functional $\Bb$}
\label{sec:Lag-action}
\unboldmath

Formula~(\ref{eq:Bb}) for
$\Bb^\times:=\Qq^*\Ss\colon {\color{cyan}
   \bar\Ll\Zfrak^\times
}\to\R$
makes sense on the larger space $\bar\Ll^\times\Zfrak$
that consists of all smooth loops in $\Zfrak$ not identically zero.
This motivates

\begin{definition}
The \textbf{non-local Lagrangian action functional}
is defined by
\begin{equation}\label{eq:Ss_Ll}
\begin{split}
   \Bb\colon {\color{red}\bar\Ll^\times\Zfrak}
    &\to\R
\\
   z
   &\mapsto
   \underbrace{2\norm{z}^2\norm{z^\prime}^2}_{\Kk(z)}      % kinetic
   +\underbrace{\frac{1}{\norm{z}^2}}_{-\Uu(z)}                  % Kepler
   +\underbrace{\int_0^1(\varsigma^*\theta)_{t_z(\tau)}|_{z_\tau}  % magnetic
   z^\prime_\tau\, d\tau
   -f(z_0^2)}_{\text{$\Mm(z)=\Theta_z z^\prime$ (\ref{eq:Theta-periodic}), (\ref{eq:Mm})}}
   .
\end{split}
\end{equation}
Here $\theta$ is a twisted-periodic $1$-form with twist function $f$
along $\Qfrak$, see~(\ref{eq:tw-per-1-form}), and
$\varsigma\colon\C\supset \Zfrak\to\Qfrak$ is the complex squaring
map~(\ref{eq:I-Zfrak-i}).
\end{definition}

\boldmath
%%%%%%%%%%%%%%%%%%%%%%%%%%%%%%%%%%%
%%%%%%% Subsection  %%%%%%%%%%%%%%%%%%%
%%%%%%%%%%%%%%%%%%%%%%%%%%%%%%%%%%%
\subsubsection[Gradient]{$L^2$-gradient}
\unboldmath

\begin{definition}\label{def:gradient-Bb}
The \textbf{\boldmath$L^2$-gradient} of the action functional
$\Bb=\Kk-\Uu+\Mm$ with respect to
the standard $L^2$-inner product at a loop
$z\in\bar\Ll^\times\Zfrak$ is determined by
$
   \inner{\grad\Bb(z)}{\xi}
   =d\Bb|_z\xi
$
for every $\xi$ with $(z,\xi)\in \Ll_\pm^\times\Zfrak \times\Ll_\pm\C$;
see~(\ref{eq:bar-Ll-C}).
Thus
\begin{equation}\label{eq:Bb=Kk+Mm+Cc}
   \grad\Bb(z)
   =
   \grad\Kk(z)-\grad\Uu(z)+\grad\Mm(z)
\end{equation}
where $\Kk$, $\Uu$, $\Mm$ are defined in~(\ref{eq:Ss_Ll}).
Points where the differential, equivalently the gradient, of a
function vanishes are called \textbf{critical points}.
\end{definition}

With the chain and product rule we calculate
\begin{equation*}%\label{eq:dKk}
\begin{split}
   d\Kk|_z\xi:
   &=\left.\tfrac{d}{d\eps}\right|_{0} \Kk(z+\eps\xi)
\\
   &=\left.\tfrac{d}{d\eps}\right|_{0} 
   2\INNER{z+\eps\xi}{z+\eps\xi}
   \INNER{z^\prime+\eps\xi^\prime}{z^\prime+\eps\xi^\prime}
\\
   &=4\norm{z^\prime}^2\INNER{z}{\xi}
   +4\norm{z}^2\INNER{z^\prime}{\xi^\prime}
\\
   &=4\norm{z^\prime}^2\INNER{z}{\xi}
   +\text{\st{${\color{gray}4\norm{z}^2}{\color{red}0}$}}
   -4\norm{z}^2\INNER{z^{\prime\prime}}{\xi}
\end{split}
\end{equation*}
where the last equality (integration by parts)
holds since $z$ is sufficiently regular,
i.e. has at least two weak derivatives and since the boundary term
vanishes\footnote{
  The derivative of $(z_{\tau+1},\xi_{\tau+1})=\pm (z_\tau,\xi_\tau)$
  in~(\ref{eq:TLlZ}) at $\tau=0$ tells
  $({\color{cyan}z_1^\prime},\xi_1^\prime)
      =({\color{cyan}\pm  z_0^\prime},\pm \xi_0^\prime)$.
  The identity $(z_{\tau+1},\xi_{\tau+1})=\pm (z_\tau,\xi_\tau)$
  in~(\ref{eq:TLlZ}) at $\tau=0$ tells
  $(z_1,{\color{brown}\xi_1})=(\pm  z_0,{\color{brown}\pm \xi_0})$.
  }
$$
   \inner{{\color{cyan}z^{\prime}_1}}{{\color{brown}\xi_1}}_0
   -\inner{z^{\prime}_0}{\xi_0}_0
   =\inner{{\color{cyan}\pm z^{\prime}_0}}{{\color{brown}\pm\xi_0}}_0
   -\inner{z^{\prime}_0}{\xi_0}_0
   =\left((\pm 1)^2-1\right) \inner{z^{\prime}_0}{\xi_0}_0
   ={\color{red}0}
   .
$$
Vanishing is independent of choosing $(z,\xi)$ or $-(z,\xi)$.
Furthermore, we calculate
\begin{equation*}%\label{eq:dCc}
\begin{split}
   d\Uu|_z\xi:
   =\left.\tfrac{d}{d\eps}\right|_{0} \Uu(z+\eps\xi)
   =\left.\tfrac{d}{d\eps}\right|_{0}
   \INNER{z+\eps\xi}{z+\eps\xi}^{-1}
   =\frac{2\INNER{z}{\xi}}{\norm{z}^4} .
\end{split}
\end{equation*}
From these differentials one immediately reads off the $L^2$-gradients
\begin{equation}\label{eq:grad-Kk-Cc}
   \grad\Kk(z)
   =4\norm{z^\prime}^2z-4\norm{z}^2z^{\prime\prime}
   ,\qquad
   \grad\Uu(z)
   =\frac{2z}{\norm{z}^4} .
\end{equation}
The magnetic gradient is much more subtle.
We calculated it in Lemma~\ref{le:gradient-Mm}.

\begin{lemma}[$L^2$-gradient of $\Bb$]\label{le:grad-Bb}
At $z\in \bar\Ll^\times\Zfrak$ the $L^2$-gradient is of the form
\begin{equation}\label{eq:grad-Bb}
\begin{split}
   \grad\Bb|_z
%1
   &=4\norm{z^\prime}^2z-4\norm{z}^2z^{\prime\prime}
   -\frac{2z}{\norm{z}^4}
   +\grad\Mm|_z
\\
%2
   &=4\norm{z^\prime}^2z-4\norm{z}^2z^{\prime\prime}
   -\frac{2z}{\norm{z}^4}
\\
   &\quad
% T3 magnetic M2
   -\frac{2z}{\norm{z}^4}
   \int_0^1
{\textstyle
   \int_0^\sigma\Abs{z_\rho}^2 d\rho
}
   \cdot
   \inner{\dot{\aaa}_{t_z(\sigma)}|_{z_\sigma}}{z^\prime_\sigma}_0\; d\sigma
% T5 magnetic M3
   -\frac{\Abs{z}^2}{\norm{z}^2}\dot{\aaa}_{t_z}|_z
   \\
   &\quad
% T4 magnetic M1
   +\frac{2 z}{\norm{z}^2}
   \int_{\sigma=\tau}^1
   \inner{\dot{\aaa}_{t_z(\sigma)}|_{z_\sigma}}{z^\prime_\sigma}_0\;
   d\sigma
% T6 magnetic M4
   -\bigl(\rot\,\aaa_{t_z}|_z\bigr)\; j_0 z^\prime .
\end{split}
\end{equation}
\end{lemma}

\begin{proof}
(\ref{eq:Bb=Kk+Mm+Cc}),
(\ref{eq:grad-Kk-Cc}),
and Lemma~\ref{le:gradient-Mm}.
\end{proof}

\boldmath
%%%%%%%%%%%%%%%%%%%%%%%%%%%%%%%%%%%
%%%%%%% Subsection  %%%%%%%%%%%%%%%%%%%
%%%%%%%%%%%%%%%%%%%%%%%%%%%%%%%%%%%
\subsubsection{Critical points -- finitely many collision times}
\label{sec:Crit-Bb}
\unboldmath

As the identity $d\Bb|_z=\inner{\grad\Bb(z)}{\cdot}$ determines the
gradient, an equation for the critical points of $\Bb$ is obtained by
setting~(\ref{eq:grad-Bb}) equal zero, then solve for~$z^{\prime\prime}$.

\begin{theorem}[regularized critical point DDE]\label{thm:crit-Ss_Ll}
The critical points of the non-local action
$\Bb\colon \bar\Ll^\times\Zfrak\to\R$
in~(\ref{eq:Ss_Ll}) are the solutions of the 
delay equation
\begin{equation}\label{eq:z}
\begin{split}
   z^{\prime\prime}(\tau)
%1
   &=z(\tau)\left(\frac{\norm{z^\prime}^2}{\norm{z}^2}
   -\frac{1}{2\norm{z}^6}\right)
   +\frac{1}{4\norm{z}^2}(\grad\Mm|_z)(\tau)
\\
%2
   &=z(\tau)
   \Biggl(
   \frac{\norm{z^\prime}^2}{\norm{z}^2}
   -\frac{1}{2\norm{z}^6}
   -\frac{1}{2\norm{z}^6}
   \int_0^1
{\textstyle
   \int_0^\sigma\abs{z_\rho}^2 d\rho
}
   \;
   \inner{\dot{\aaa}_{t_z(\sigma)}|_{z_\sigma}}{z^\prime_\sigma}_0\; d\sigma
   \\
   &\quad\qquad\qquad\qquad\qquad\quad\;
   +\frac{1}{2\norm{z}^4}
   \int_{\tau}^1
   \inner{\dot{\aaa}_{t_z(\sigma)}|_{z_\sigma}}{z^\prime_\sigma}_0\;
   d\sigma
   \Biggr)
   \\
   &\quad
   -
   \frac{\Abs{z(\tau)}^2}{4\norm{z}^4}\cdot
    \dot{\aaa}_{t_z(\tau)}|_{z(\tau)}
   -\frac{\rot\,\aaa_{t_z(\tau)}|_{z(\tau)}}{4\norm{z}^2}
   \; j_0 z^\prime(\tau)
\end{split}
\end{equation}
for smooth loops
$z\not\equiv 0$ in $\Zfrak$.
In symbols, the set of critical points is given by
\begin{equation*}%\label{eq:Crit-Bb}
\begin{split}
   \Crit\,\Bb:
   &=\{z\in \bar\Ll^\times\Zfrak\mid \grad\Bb|_z=0\}
   =\{\text{\rm$z\in \bar\Ll^\times\Zfrak$ solving~(\ref{eq:z})}\}
   .
\end{split}
\end{equation*}
The elements $z\in  \Crit\,\Bb$ are called
\textbf{regularized collision solutions}.\footnote{
  away from collisions, they correspond to classical/physical
  solutions, i.e. of the ODE~(\ref{eq:q})
  }
\end{theorem}

\begin{proof}
Divide~(\ref{eq:grad-Bb}) by $4\norm{z}^2$, resolve for $z^{\prime\prime}$.
\end{proof}

\begin{lemma}\label{le:finite-colls}
For $z\in\Crit\,\Bb$ \textbf{regularization collision times}
form a finite set
$$
   \Tt_z:=\{\tau_*\in[0,1)\colon z(\tau_*)=0 \}
   =\{\tau_1,\dots,\tau_N\}
$$
enumerated increasingly $\tau_j<\tau_{j+1}$.
Hence $[0,1)\setminus \Tt_z$ is a union of
$N$ intervals\footnote{
  In case $\tau_1=0$ the formal final interval $(\tau_N,\tau_1)$
  abbreviates $(\tau_N,1)$.
  In case $\tau_1>0$ the formal final interval $(\tau_N,\tau_1)$
  abbreviates $(\tau_N,1)\cup[0,\tau_1)$.
  }
\begin{equation*}%\label{eq:Tt_z}
   [0,1)\setminus \Tt_z
   =\Ii_1\cup\dots\cup\Ii_N
   ,\qquad
   \Ii_j:=(\tau_j,\tau_{j+1}))
   ,\quad 
   \tau_{N+1}:=\tau_1 .
\end{equation*}
Furthermore, at collisions velocities are non-zero while accelerations vanish
$$
   \forall\tau_*\in\Tt_z=z^{-1}(0)\colon
   \qquad
   z^\prime(\tau_*)\not=0
   ,\qquad
   z^{\prime\prime}(\tau_*)=0
   .
$$
\end{lemma}

\begin{proof}
We define three maps pointwise at $\tau\in\SS^1$, namely
$a\in\Ll\R$ by
$$\small
   a(\tau)
   :=
   \frac{\norm{z^\prime}^2}{\norm{z}^2}
   -
   \frac{(1+\int_0^1\textstyle{\int_0^s
   \Abs{z_\sigma}^2 d\sigma}\;
   \inner{\dot{\aaa}_{t_z(s)}|_{z_s}}{z^\prime_s}_0 ds)}{2\norm{z}^6}
   +
   \frac{\int_{\tau}^1
   \inner{\dot{\aaa}_{t_z(\sigma)}|_{z_\sigma}}{z^\prime_\sigma}_0
   d\sigma}{2\norm{z}^4}
   ,
$$
the map $b\in\Ll\C$ by
$
   b(\tau)
   :=
   -\tfrac{1}{4\norm{z}^4}\;
   \bar z(\tau)\dot{\aaa}_{t_z(\tau)}|_{z(\tau)}
$,
and the map $c\in\Ll i\R$ by
$
   c(\tau)
   :=-\tfrac{1}{4\norm{z}^2}\; i\tilde B_{t_z\tau}|_{z_\tau}
$.
Because $z$ solves the regularized DDE~(\ref{eq:z}),
it is as well a solution of the second order
linear homogeneous ODE with continuous coefficient functions
$
   z^{\prime\prime}(\tau)
   =a(\tau)z(\tau)+b(\tau) z(\tau)+c(\tau)z^\prime(\tau)
$.
This implies that
\begin{equation*}%\label{eq:30}
   z^\prime(\tau_*)\not=0, \quad\forall \tau_*\in \Tt_z .
\end{equation*}
Suppose by contradiction that $z$ is a solution of the second order
ODE with initial condition $z(\tau_*)=z^\prime(\tau_*)=0$. Hence
$z\equiv0$ and $z\in \bar\Ll^\times\Zfrak$. Contradiction.
So $\Tt_z$ is a discrete
subset of the compact space $\SS^1$, hence finite.

That $z(\tau_*)=0$ $\Rightarrow$ $z^{\prime\prime}(\tau_*)=0$ follows
from the critical point equation~(\ref{eq:z}) together with
$\rot\,\aaa_t|_0=0$ by~(\ref{eq:tilde-A-A}).
This proves Lemma~\ref{le:finite-colls}.
\end{proof}

\boldmath
%%%%%%%%%%%%%%%%%%%%%%%%%%%%%%%%%%%
%%%%%%% Subsection  %%%%%%%%%%%%%%%%%%%
%%%%%%%%%%%%%%%%%%%%%%%%%%%%%%%%%%%
\subsection[Non-collisional regularized solutions]{Non-collisional regularized solutions
-- $\Crit\,\Bb^\times$}
\label{sec:reg-sol} 
\unboldmath

We show that those solutions $z$ of the regularized critical point
DDE~(\ref{eq:z}) which actually avoid the singularity at the origin,
in symbols $z\in(\Crit\,\Bb)\cap {\color{cyan}\bar\Ll\Zfrak^\times}$,
correspond simultaneously to \emph{physical solutions}, i.e.
solutions of the ODE~(\ref{eq:q}), namely via the map 
$z\mapsto \Qq(z)=:q\in\Ll\Qfrak^\times$.
Abstractly this is clear, since solutions of~(\ref{eq:q})
are critical points of the classical functional
$\Ss\colon\Ll\Qfrak^\times\to\R$ and the restriction
$\Bb^\times\colon {\color{cyan}\bar\Ll\Zfrak^\times}\to\R$ of the
regularized functional $\Bb\colon {\color{red}\bar\Ll^\times\Zfrak}\to\R$
coincides with the pull-back
$\Ss\circ\Qq\colon{\color{cyan}\bar\Ll\Zfrak^\times}\to\R$,
as illustrated by Figure~\ref{fig:fig-Qq}.
Although clear, it is not obvious by looking at the
equations~(\ref{eq:z}) and~(\ref{eq:q}).

Hence in this section we show this explicitly.
In doing so we find in an intermediate step a second order DDE for $q$
which will help us in the subsequent Section~\ref{sec:collisions} to see
how solutions of the regularized DDE~(\ref{eq:z}) in the complement
of $\bar\Ll\Zfrak^\times$ in $\bar\Ll^\times\Zfrak$
correspond to \emph{collisional} solutions of the classical ODE~(\ref{eq:q}).

\begin{proposition}\label{prop:reg-sol}
Let $z\in \bar\Ll\Zfrak^\times$ solve the regularized
critical point DDE~(\ref{eq:z}). Then $q:=z^2\circ\tau_z$,
Definition~\ref{def:reg-path-C},
solves the classical critical point ODE~(\ref{eq:q}).
\end{proposition}

\begin{proof}
The proof has two steps.

\smallskip\noindent
\textbf{Step~1.}
{\it
If $z\in \bar\Ll\Zfrak^\times$ is a solution
of the regularized critical point DDE~(\ref{eq:z}),
then $q:=\Qq(z):=z^2\circ\tau_z$ is a solution of the $2^{\rm nd}$ order
\textbf{intermediate DDE}
\begin{equation}\label{eq:q-DDE}
\begin{split}
   &\ddot q_t+B_t(q_t)j_0\dot q_t+\AAAdot_t(q_t)
\\
   &=
   \frac{
   \tfrac12 \norm{\dot q}^2
   -\int_0^1\frac{ds}{\abs{q_s}}
   -
   \int_0^1s \inner{\AAAdot_{s}|_{q_s}}{\dot q_s}_0\; ds
   -\tfrac12 \Abs{\dot q_t}^2
   +
   \int_{t}^1\inner{\AAAdot_{s}|_{q_s}}{\dot q_s}_0\; ds
   }{\bar q_t} .
\end{split}
\end{equation}
}

\begin{proof}[Proof of Step~1]
We need to calculate the first two derivatives of
$\Qq(z)\colon\SS^1\to\Qfrak^\times$.
It is convenient to abbreviate and use complex notation
$$
   q:=q_z:=\Qq(z)
   ,\qquad
   q=q_1+iq_2
   ,\qquad
   q_t:=q(t)
   ,\qquad
   q_t=q_{1,t}+iq_{2,t} .
$$
Hence $\bar q:=\widebar{q_z}=\bar z^2\circ \tau_z$.
The derivative $\dot q$ is given by~(\ref{eq:dot-q_z})
and the derivative of $\tau_z(t)$ by~(\ref{eq:inverse-C}).
So, following~\cite[\S 5]{Frauenfelder:2025b},
the second derivative is of the form
\begin{equation*}%\label{eq:ddot-q_z}
\begin{split}
   \ddot q_t
   &=\frac{2\norm{z}^4}{\bar z_{\tau_z(t)} \bar z_{\tau_z(t)} z_{\tau_z(t)}}
   \left(z_{\tau_z(t)}^{\prime\prime}-\frac{\abs{z^\prime_{\tau_z(t)}}^2}
   {\bar z_{\tau_z(t)}}\right)
   =\frac{1}{\bar q_t}\left(
   \frac{2\norm{z}^4 z_{\tau_z(t)}^{\prime\prime}}{z_{\tau_z(t)}}
   -\frac12 \Abs{\dot q_t}^2
   \right)
\end{split}
\end{equation*}
for every $t\in\SS^1$.
To calculate $\Abs{\dot q_t}^2={\dot q}_t\bar{\dot q}_t$
we used~(\ref{eq:dot-q_z}) and its complex conjugate form.
Continuing the calculation of $\ddot q_t$,
we change first the order of the two summands, 
then in equality~2 we replace $z_{\tau_z(t)}^{\prime\prime}$
by~(\ref{eq:z}), hence
\begin{equation*}%\label{eq:}
\begin{split}
   \ddot q(t)
%1
   &=-\frac{\Abs{\dot q_t}^2}{2\bar q_t}
   +\frac{\norm{z}^2}{2\bar q_t}
   \frac{4\norm{z}^2 z_{\tau_z(t)}^{\prime\prime}}{z_{\tau_z(t)}}
\\
%2
   &\stackrel{{\color{gray}2}}{=}
   -\frac{\Abs{\dot q_t}^2}{2\bar q_t}
   +\frac{2}{\bar q_t}\norm{z}^2\norm{z^\prime}^2
   -\frac{1}{\bar q_t\norm{z}^2}
   \\
   &\quad\;
{\color{brown}
   -\frac{1}{\bar q_t\norm{z}^2}
   \int_0^1
\underbrace{
{\textstyle
   \int_0^\sigma\Abs{z_\rho}^2 d\rho}
}_{\stackrel{\text{(\ref{eq:t_z})}}{=}t_z(\sigma)\norm{z}^2} \;
   \inner{\dot{\aaa}_{t_z(\sigma)}|_{z_\sigma}}{z^\prime_\sigma}_0\; d\sigma
}
   \\
   &\quad\;
{\color{cyan}
   +\frac{1}{\bar q_t}
   \int_{\sigma={\tau_z(t)}}^1
   \inner{\dot{\aaa}_{t_z(\sigma)}|_{z_\sigma}}{z^\prime_\sigma}_0\;
   d\sigma
}
   -\frac{\Abs{z_{\tau_z(t)}}^2}{2\bar q_t z_{\tau_z(t)}}
    \dot{\aaa}_{{\color{red}t}}|_{z_{\tau_z(t)}}
   \\
   &\quad
   -\frac{1}{2\bar q_t}
    \norm{z}^2\left(\rot\,\aaa_{{\color{red}t}}|_{z_{\tau_z(t)}}\right)
   \; i \frac{{\color{red}z^\prime_{\tau_z(t)}}}{z_{\tau_z(t)}}
\\
%3
   &\stackrel{{\color{gray}3}}{=}
   -\frac{\Abs{\dot q_t}^2}{2\bar q_t}
   +
   \frac{1}{2\bar q_t}\norm{\dot q}^2
   -\frac{1}{\bar q_t}\int_0^1\frac{1}{\abs{q_s}}\, ds
   \\
   &\quad\;
{\color{brown}
   -\frac{1}{\bar q_t}
   \int_{s=t_z(0)=0}^{t_z(1)=1}
   s\cdot
   \Re\left(\overline{2{\bar z}_\sigma\AAAdot_s}\cdot
   \frac{\dot q_s {\bar z}_\sigma}{2\norm{z}^2}\right)
   \frac{\norm{z}^2ds}{z_\sigma\bar z_\sigma}
}
   \\
   &\quad\;
{\color{cyan}
   +\frac{1}{\bar q_t}
   \int_{s=t_z(\tau_z)=t}^{t_z(1)=1}
   \Re\left(\overline{2{\bar z}_\sigma\AAAdot_s}\cdot
   \frac{\dot q_s {\bar z}_\sigma}{2\norm{z}^2}\right)
   \frac{\norm{z}^2ds}{z_\sigma\bar z_\sigma}
}
   -\AAAdot_t|_{z_{\tau_z(t)}^2}
   \\
   &\quad
   -\frac{1}{2\bar q_t}
   \norm{z}^2 4   \bar z_{\tau_z(t)}   z_{\tau_z(t)}
   \underbrace{\left(\rot\,\AAA_{t}|_{q_t}\right)}_{=B_t(q_t)}
   \; i \frac{{\color{red}\dot q_t \bar z_{\tau_z(t)}}}
      {z_{\tau_z(t)} {\color{red}2\norm{z}^2}}
\\
%4
   &\stackrel{{\color{gray}4}}{=}
   -\frac{\Abs{\dot q_t}^2}{2\bar q_t}
   +
   \frac{1}{2\bar q_t}\norm{\dot q}^2
   -\frac{1}{\bar q_t}\int_0^1\frac{1}{\abs{q_s}}\, ds
   -\frac{1}{\bar q_t}
   \int_0^1
   s\cdot
   \inner{\AAAdot_{s}|_{q_s}}{\dot q_s}_0\; ds
   \\
   &\quad
%{\color{cyan}
   +\frac{1}{\bar q_t}
   \int_{t}^1
   \inner{\AAAdot_{s}|_{q_s}}{\dot q_s}_0\;
   ds
%}
   -\AAAdot_t|_{q_t}
   -B_t(q_t) j_0 \dot q_t .
\end{split}
\end{equation*}
Equality 2 also uses that
$t_z({\tau_z(t)})={\color{red}t}$ by~(\ref{eq:t-inverse}).
\newline
Equality 3: There are seven summands.
Summand one remains unchanged.
In summand two replace $\norm{z}^2\norm{z^\prime}^2$ by
$\frac{1}{4}\norm{\dot q}^2$ according to~(\ref{eq:fsfcs-C}).
In summand three replace $1/\norm{z}^2$ by~(\ref{eq:1/q^2-C}).
In {\color{brown} summand four} we replaced the $\int_0^\sigma$-term
according to~(\ref{eq:t_z}). Then we changed the integration
variable to $s(\sigma):=t_z(\sigma)$. The substitution of the inner
product, of $z_\sigma^\prime$, and of $d\sigma$ is explained next in
summand five.
In
{\color{cyan}
summand five 
}
we rewrote the inner product in the $\R^2$ picture
in terms of the real part~(\ref{eq:bar-z-zeta}) in the $\C$ picture
and we substituted the vector potential according to~(\ref{eq:barzA})
with dots (time derivative).
Then we changed the integration variable to $s:=t_z(\sigma)$:
We replaced $d\sigma$ with the help of~(\ref{eq:inverse-C})
and $z_\sigma^\prime$ according to~(\ref{eq:dot-q_z}).
Now many factors annulate in pairs and (in equality 4) we go back
to the $\R^2$ picture and inner product.
In summand six we also used~(\ref{eq:barzA}) for time derivatives
and that $\bar q_t=\bar z^2_{\tau_z(t)}$.
To the final summand seven the following happened:
We used that $z_{\tau_z(t)}^2=q_t$.
We replaced the rotational according to~(\ref{eq:tilde-A-A})
and ${\color{red}z^\prime_{\tau_z(t)}}$ according to~(\ref{eq:dot-q_z}).
Now certain factors annulate in pairs.
Equality 4 writes down the cleaned up result.
This proves Step~1.
\end{proof}   % Step 1

\smallskip\noindent
\textbf{Step~2.}
{\it
If $q\in \Ll\Qfrak^\times$ is a solution
of the intermediate DDE~(\ref{eq:q-DDE}), then it is also a solution
of the classical ODE~(\ref{eq:q}).
}

\begin{proof}[Proof of Step~2]
Given a solution $q\in \Ll\Qfrak^\times$ of~(\ref{eq:q}),
we define
\begin{equation}\label{eq:beta}
   \beta_t
   :=\frac{\ddot q_t+B_t(q_t) i\dot q_t+\AAAdot_t(q_t)}{q_t}
   ,\quad \forall t\in\R .
\end{equation}
By definition of $\beta$ and since $\Abs{q_t}^2=\bar q_t q_t$,
using~(\ref{eq:q-DDE}) in equality~2, we obtain
\begin{equation}\label{eq:beta-q}
\begin{split}
   &\beta_t\Abs{q_t}^2
   =\bar q_t\left(\ddot q_t+B_t|_{q_t} i\dot q_t+\AAAdot_t|_{q_t}\right)
\\
   &\stackrel{\color{gray}2}{=}
\underbrace{
  \tfrac12 \norm{\dot q}^2
   -\int_0^1\tfrac{ds}{\abs{q_s}}
   -
   \int_0^1s \inner{\AAAdot_{s}|_{q_s}}{\dot q_s}_0\; ds
}_{\text{independent of $t$}}
   -\tfrac12 \Abs{\dot q_t}^2
   +
   \int_{t}^1\inner{\AAAdot_{s}|_{q_s}}{\dot q_s}_0\; ds .
\end{split}
\end{equation}
As the right hand side of~(\ref{eq:beta-q}) is real, so is
$\beta$, in symbols
$
   \beta_t\in\R, \quad \forall t\in\R
$.
Differentiate~(\ref{eq:beta-q}) to get the first equality, 
the definition of $\beta$ gives the second
\begin{equation*}
\begin{split}
   &2\INNER{q_t}{\dot q_t}_0 \beta_t
   +\Abs{q_t}^2\dot \beta_t
\\
   &=-\INNER{\dot q_t}{\ddot q_t}_0
   -\inner{\AAAdot_t|_{q_t}}{\dot q_t}_0
\\
   &=-\inner{\dot q_t}{\beta_t q_t}_0
   +\inner{\dot q_t}{B_t(q_t) j_0\dot q_t}_0
   +\inner{\dot q_t}{\AAAdot_t(q_t)}_0
   -\inner{\AAAdot_t|_{q_t}}{\dot q_t}_0
\\
   &=-\beta_t \inner{q_t}{\dot q_t}_0
   .
\end{split}
\end{equation*}
Equality three uses that $\beta$ is real
and that, since $B$ is real and by~(\ref{eq:g_0-omega_0-J_0}),
the term
$
   \inner{\dot q_t}{B_t(q_t) j_0\dot q_t}_0
   =B_t(q_t)\omega_0(\dot q_t,j_0j_0 \dot q_t)
   =-B_t(q_t)\omega_0(\dot q_t, \dot q_t)
   =0
$
vanishes. By rearrangement
$\Abs{q}^2\dot\beta=-3\inner{q}{\dot q}\beta$, thus
$
   \tfrac{\dot \beta}{\beta}
   =-\tfrac{3}{2}\tfrac{\p_t \Abs{q}^2}{\Abs{q}^2}
$.
A solution of this first order ODE is of the form
$
   \ln\Abs{\beta}
   =-\frac{3}{2}\ln\Abs{q}^2+c
   =-\ln\Abs{q}^3+c
$
for some constant $c\in\R$.
Equivalently $\beta=\tfrac{\mu}{\Abs{q}^3}$
for $\mu=\pm e^c$.
Definition~(\ref{eq:beta}) of $\beta$
shows that $q$ solves the second order ODE
\begin{equation}\label{eq:mu-q}
   \ddot q=\tfrac{\mu q}{\Abs{q}^3}-B_t(q) i\dot q- \AAAdot_t(q) .
\end{equation}
To prove $\mu=-1$, thus proving Step~2,
combine $\beta=\tfrac{\mu}{\Abs{q}^3}$
and~(\ref{eq:beta-q}) to get
\begin{equation*}
\begin{split}
   \frac{\mu}{\Abs{q_t}}
   &=
\underbrace{
  \tfrac12 \norm{\dot q}^2
   -\int_0^1\tfrac{ds}{\abs{q_s}}
   -
   \int_0^1s \inner{\AAAdot_{s}|_{q_s}}{\dot q_s}_0\; ds
}_{\text{independent of $t$}}
   -\tfrac12 \Abs{\dot q_t}^2
   +
   \int_{s=t}^1\inner{\AAAdot_{s}|_{q_s}}{\dot q_s}_0\; ds .
\end{split}
\end{equation*}
Integrate this equation and in equality 2 change the order of integration
to get
\begin{equation}\label{eq:mu=-1}
\small
\begin{split}
   \mu\int_0^1\frac{dt}{\Abs{q_t}}
   &=\tfrac12 \norm{\dot q}^2
   -\int_0^1\frac{ds}{\abs{q_s}}
   -
   \int_0^1s \inner{\AAAdot_{s}|_{q_s}}{\dot q_s}_0\; ds
\\
   &\quad
   -\tfrac12\int_0^1\Abs{\dot q_t}^2\, dt
   +\int_{t=0}^1 \int_{s=t}^1\inner{\AAAdot_{s}|_{q_s}}{\dot q_s}_0\; ds dt
\\
   &\stackrel{{\color{gray}2}}{=}
   -\int_0^1\frac{dt}{\abs{q_t}}
   -
   \int_0^1s \inner{\AAAdot_{s}|_{q_s}}{\dot q_s}_0\; ds
   +\int_{s=0}^1 \int_{t=0}^s\inner{\AAAdot_{s}|_{q_s}}{\dot q_s}_0\; dt ds
\\
   &=-\int_0^1\frac{dt}{\abs{q_t}}.
\end{split}
\end{equation}
Thus $\mu=-1$ and this proves Step~2.
\end{proof}

Step~1 and Step~2 together conclude the proof of
Proposition~\ref{prop:reg-sol}.
\end{proof}

\boldmath
%%%%%%%%%%%%%%%%%%%%%%%%%%%%%%%%%%%
%%%%%%% Subsection  %%%%%%%%%%%%%%%%%%%
%%%%%%%%%%%%%%%%%%%%%%%%%%%%%%%%%%%
\subsection[Collisional regularized solutions]{Collisional regularized solutions -- $\Crit\,\Bb$}
\label{sec:collisions}
\unboldmath

First we extend the reparametrization map $t_z$
to collisional loops $z\in{\color{red}\bar\Ll^\times\Zfrak}$
with finitely many collisions of non-zero speed each.
While such extension is useful later on, in this sections we
focus on critical points $z$ of $\Bb$; cf. Lemma~\ref{le:finite-colls}.

\smallskip\noindent
The goal is to show that applying the rescale-square operation $\Qq$
to $z\in\Crit\,\Bb$ yields a continuous map
$q_z:=z^2\circ\tau_z\colon\SS^1\to\Qfrak$
which, away from \textbf{classical collision times}
$T_{q_z}:=q_z^{-1}(0)$, is smooth and solves the classical 
ODE~(\ref{eq:q}).
\\
As a byproduct we get that
at all collision times the collision velocity is equal
$\abs{z^\prime(\tau_*)}=1/\sqrt{2}\norm{z}^2$
as seen in~(\ref{eq:collision-velocity}) since $\mu_j=-1$.
(Later on, in~(\ref{eq:z-prime-prime=0}),
we show that actually the collision acceleration
$\abs{z^{\prime\prime}(\tau_*)}=0$ vanishes.)

\boldmath
%%%%%%%%%%%%%%%%%%%%%%%%%%%%%%%%%%%
%%%%%%% Subsection  %%%%%%%%%%%%%%%%%%%
%%%%%%%%%%%%%%%%%%%%%%%%%%%%%%%%%%%
\subsubsection{Collisional reparametrization map $t_z$}
\label{sec:coll-BOV-map}
\unboldmath

\begin{lemma}\label{le:t_z-coll}
Let $z\in\bar\Ll^\times\Zfrak$ be a loop
with finite \textbf{regularization collision times set}
$\Tt_z:=z^{-1}(0)$ and
$z^\prime(\tau_*)\not=0$ $\forall\tau_*\in\Tt_z$.
Then the following is true.
\begin{itemize}\setlength\itemsep{0ex}
\item[\rm (i)]
  The map $t_z\colon \SS^1_\tau\to\SS^1_t$ defined (as in the
  non-collisional case) by
  \begin{equation}\label{eq:t_z-prime-coll}
     \forall \tau\in\R\colon
     \qquad
     t_z(\tau):=\frac{\int_0^\tau\Abs{z(\sigma)}^2d\sigma}{\norm{z}^2}
     ,\qquad
  {\color{gray}
     t_z^\prime(\tau)=\frac{\Abs{z(\tau)}^2}{\norm{z}^2}\ge0
     ,
  }
  \end{equation}
  is smooth and a strictly increasing homeomorphism with $t_z(0)=0$.
\item[\rm (ii)]
  The inverse homeomorphism $\tau_z:=t_z^{-1}\colon \SS^1_t\to\SS^1_\tau$
  is smoothly diff\-erentiable away from
  collision times $t_z(\Tt_z)$ where
  $\dot \tau_z(t)=\tfrac{1}{t_z^\prime(\tau_z(t))}
  =\tfrac{\norm{z}^2}{\Abs{z(\tau_z(t))}^2}>0$.
\end{itemize}
\end{lemma}

\begin{proof}
That $t_z$ maps $\SS^1=\R/\Z$ to $\SS^1$ is true, because
$t_z(\tau+1)=t_z(\tau)+1$ $\forall \tau\in\R$, which follows from
$\int_0^{\tau+1}=\int_0^1+\int_1^{\tau+1}$ and since
$z_{\sigma+1}=z_\sigma$ is periodic.
Clearly $t_z$ is smooth and its derivative is of the
form~(\ref{eq:t_z-prime-coll}) which also shows that collision times
precisely coincide with critical points of $t_z$, in symbols
$
   \Tt_z=\Crit\, t_z
$.
The proof completes in two bullets.

$\bullet$
The critical points of $t_z$ are inflection points of $t_z$,
more precisely
\begin{equation}\label{eq:31}
   \forall \tau_*\in\Tt_z
   \qquad
   t_z^\prime(\tau_*)=t_z^{\prime\prime}(\tau_*)=0
   ,\qquad
   t_z^{\prime\prime\prime}(\tau_*)>0
   .
\end{equation}
To see this we compute the following derivatives
\begin{equation*}%\label{eq:tau-primeprimeprime}
   t_z^{\prime\prime}(\tau)
   =\tfrac{2\INNER{z(\tau)}{z^\prime(\tau)}_0}{\norm{z}^2}
   ,\qquad
   t_z^{\prime\prime\prime}(\tau)
   =2\tfrac{\Abs{z^\prime(\tau)}^2+\INNER{z(\tau)}{z^{\prime\prime}(\tau)}}
   {\norm{z}^2}
   ,
\end{equation*}
for $\tau\in\R$.
At any collision $z(\tau_*)=0$, as
$\abs{z^\prime_{\tau_*}}=\tfrac{1}{\sqrt{2}\norm{z}^2}\not=0$,
we conclude that
$$
   t_z^{\prime}(\tau_*)=0
   ,\qquad
   t_z^{\prime\prime}(\tau_*)=0
   ,\qquad
   t_z^{\prime\prime\prime}(\tau_*)
   =\tfrac{2\Abs{z^\prime(\tau_*)}^2}{\norm{z}^2}>0
   .
$$

$\bullet$
By~(\ref{eq:t_z-prime-coll}) and~(\ref{eq:31})
the continuous map $t_z$ is strictly 
monotonically increasing.
So, since $t_z(0)=0$ and $t_z(1)=1$, 
as in the regular case of Remark~\ref{rem:times}, the map
$t_z$ has a continuous inverse $\tau_z:=t_z^{-1}\colon\SS^1_t\to\SS^1_\tau$
strictly monotonically increasing.
Thus $\tau_z$ is still a homeomorphism of the circle,
but in contrast to the regular case $\tau_z$ is not everywhere differentiable.
Away from the finite set $t_z(\Tt_z)$ it is still smoothly differentiable,
the derivative still given by~(\ref{eq:inverse-C}).
\end{proof}

\boldmath
%%%%%%%%%%%%%%%%%%%%%%%%%%%%%%%%%%%
%%%%%%% Subsection  %%%%%%%%%%%%%%%%%%%
%%%%%%%%%%%%%%%%%%%%%%%%%%%%%%%%%%%
\subsubsection{Collisional rescale-square map $\Qq$}
\label{sec:coll-BOV-map}
\unboldmath

\begin{proposition}[$\Qq$ on $\Crit\,\Bb$]\label{prop:q_z-collision}
A solution $z\in\bar\Ll^\times\Zfrak$
to the regularized DDE~(\ref{eq:z}),
equivalently $z\in\Crit\,\Bb$,
yields a $W^{1,2}$-map
\begin{equation}\label{eq:Qq-extension}
   q_z:=\Qq(z)=z^2\circ\tau_z\colon\SS^1\to\Qfrak
\end{equation}
smoothly differentiable away from classical collision
times~${t_j:=t_z(\tau_j)}$. Set
\begin{equation*}%\label{eq:T_q}
   T_{q_z}
   :=q_z^{-1}(0)
   =t_z(\Tt_z)
   =\{t_1,\dots,t_N\}
   ,\quad
   I_j:=t_z(\Ii_j) =(t_j,t_{j+1}) .
\end{equation*}
Restricted to non-collision times
$q_z\colon\SS^1\setminus T_{q_z}\to\Qfrak$
solves the classical ODE~(\ref{eq:q}).
Furthermore, at all collision times of $z$
the collision velocity is equal
$$
   \forall z\in\Crit\,\Bb\quad \forall \tau_*\in \Tt_z=z^{-1}(0)\colon
   \qquad
  \abs{z^\prime(\tau_*)}=\frac{1}{\sqrt{2}\norm{z}^2}
   .
$$
\end{proposition}

\begin{proof}
Pick $z\in\Crit\,\Bb$.
Lemma~\ref{le:t_z-coll} applies by Lemma~\ref{le:finite-colls}.
That the composition $q_z=z^2\circ\tau_z$ is continuous and
smoothly differentiable away from $t_z(\Tt_z)$ holds since $t_z$ has
these properties, see Lemma~\ref{le:t_z-coll}, and $z$ is smooth anyway.
Since $t_z$ fixes both $0$ and $1$, and it is strictly monotonically
increasing the order $\tau_1<\dots<\tau_N$ is inherited by the images
$t_1<\dots<t_N$.

It remains to show that $q:=q_z$ lies in $W^{1,2}$ and restricted to any
of the collision free intervals $I_j$ solves the classical
ODE~(\ref{eq:q}). This takes three steps.

\smallskip\noindent
\textsc{Step~1.}
The same argument as in the regular case---in the proof of
Proposition~\ref{prop:reg-sol} replace $t\in\SS^1$ by
$t\in I_j$---shows a) that $q_z|_{I_j}$ solves the intermediate
DDE~(\ref{eq:q-DDE}) and b) that there exists $\mu_j\in\R$ such that
\begin{equation}\label{eq:mu_j-q}
   \ddot q_t
   =\frac{\mu_j q_t}{\Abs{q_t}^3}-B_t(q_t) i\dot q_t- \AAAdot_t(q_t)
   ,\quad
   \forall t\in I_j
    .
\end{equation}
To see this note the following. a)~The DDE~(\ref{eq:q-DDE}) still
makes sense since $1/\bar q_t$ is finite for $t\in I_j$
and so is the mean value $\int_0^1\frac{ds}{\abs{q_s}}<\infty$,
by~(\ref{eq:1/q^2-C}).
b)~Since $q_t\not=0$ the function $\beta_t$ in~(\ref{eq:beta})
is well defined and so is~(\ref{eq:mu-q}),
respectively~(\ref{eq:mu_j-q}), and the argument to get there.

\smallskip\noindent
\textsc{Step~2.}
We show that $\mu_j=-1$ in~(\ref{eq:mu_j-q}) for any $j$.
To see this pick $t\in I_j$, then
\begin{equation}\label{eq:mu_j(t)}
\begin{split}
   \mu_j
   &\stackrel{\text{(\ref{eq:mu_j-q})}}{=}
   \Abs{q_t}\bar q_t
   \left(
   \ddot q_t+B_t(q_t) i\dot q_t+\AAAdot_t(q_t)
   \right)
\\
   &\stackrel{\text{(\ref{eq:q-DDE})}}{=}
   -\frac{\Abs{q_t}\cdot \Abs{\dot q_t}^2}{2}
   \\
   &\quad
   +
\underbrace{
   \Abs{q_t}
   {\textstyle
   \left(
  \tfrac12 \norm{\dot q}^2
   -\int_0^1\frac{ds}{\abs{q_s}}
   -
   \int_0^1s \inner{\AAAdot_{s}|_{q_s}}{\dot q_s}_0\; ds
   +
   \int_{t}^1\inner{\AAAdot_{s}|_{q_s}}{\dot q_s}_0\; ds
   \right) 
   }
}_{\text{$\to0$, as $t\to t_j$ or $t\to t_{j+1}$ }}
   .
\end{split}
\end{equation}
The limit is zero as the first three summands in the bracket
are constants and summand four remains finite as $t$ goes to $t_j$ or $t_{j+1}$.
By definition of $q_z$ we get
\begin{equation}\label{eq:q-qdot}
   \frac{\Abs{q_t}\cdot \Abs{\dot q_t}^2}{2}
   \stackrel{\text{(\ref{eq:dot-q_z})}}{=}
   \frac{\Abs{z^2(\tau_z(t))}\cdot 4\norm{z}^4 \Abs{z^\prime(\tau_z(t))}^2}
   {2 \Abs{\bar z(\tau_z(t))}^2}
   =
   2\norm{z}^4 \Abs{z^\prime(\tau_z(t))}^2 .
\end{equation}
Hence, if we take in~(\ref{eq:mu_j(t)}) the limit $t\to t_j$,
respectively $t\to t_{j+1}$, we obtain
\begin{equation}\label{eq:collision-velocity}
   \mu_j
   =-2\norm{z}^4 \Abs{z^\prime(\tau_z(t_j))}^2
   =-2\norm{z}^4 \Abs{z^\prime(\tau_z(t_{j+1}))}^2 .
\end{equation}
Since this is true for all $j$, the right hand side is actually equal
to $\mu_{j+1}$ viewed as the left boundary limit of
$I_{j+1}=(t_{j+1},t_{j+2})$. This means that all the $\mu_j$ are equal
to the same real number, say $\mu$.
Now~(\ref{eq:mu_j-q}) tells that
\begin{equation*}%\label{eq:mu_j-q-no-j}
   \ddot q_t
   =\tfrac{\mu q_t}{\Abs{q_t}^3}-B_t(q_t) i\dot q_t- \AAAdot_t(q_t)
   ,\quad
   \forall t\in \SS^1\setminus t_z(\Tt_z)
   .
\end{equation*}
As in the regular case~(\ref{eq:mu=-1}),
using in addition that $t_z(\Tt_z)$ is a finite set,
integration implies that $\mu=-1$.
Hence, by~(\ref{eq:collision-velocity}), 
it holds $\abs{z^\prime(\tau_*)}=1/\sqrt{2}\norm{z}^2$
whenever $\tau_*\in z^{-1}(0)$.
This proves Step~2.

\smallskip\noindent
\textsc{Step~3.}
Since $q$ is continuous it is $L^2$.
To see that $q:=q_z\in W^{1,2}(\SS^1,\Qfrak)$
it remains to show finiteness of
\begin{equation*}
\begin{split}
   \int_0^1\abs{\dot q_t}^2\, dt
%1
   \stackrel{{\color{gray}1}}{=}
   4\norm{z}^4\int_0^1
   \frac{\abs{z^\prime_{\tau_z(t)}}^2}{\abs{z_{\tau_z(t)}}^2}\, dt
%2
   \stackrel{{\color{gray}2}}{=}
   4\norm{z}^2\int_0^1\abs{z^\prime_\sigma}^2\, d\sigma
%3
   <\infty
   .
\end{split}
\end{equation*}
Equality~1 is by~(\ref{eq:q-qdot}) with $q_t$ replaced by $z_{\tau_z(t)}^2$.
Equality~2 is by variable substitution $\sigma:=\tau_z(t)$, thus
$dt=\frac{\norm{z}^2}{\abs{z_{\tau_z(t)}}^2} d\tau$.
Finiteness is true since $z^\prime$ is smooth and its domain $\SS^1$
is compact.
This proves Step~3 and Proposition~\ref{prop:q_z-collision}.
\end{proof}

\boldmath
%%%%%%%%%%%%%%%%%%%%%%%%%%%%%%%%%%%
%%%%%%% Subsection  %%%%%%%%%%%%%%%%%%%
%%%%%%%%%%%%%%%%%%%%%%%%%%%%%%%%%%%
\subsection{Correspondence of regularized and
classical collisions}
\label{sec:correspondence}
\unboldmath

\boldmath
%%%%%%%%%%%%%%%%%%%%%%%%%%%%%%%%%%%
%%%%%%% Subsubsection  %%%%%%%%%%%%%%%%%
%%%%%%%%%%%%%%%%%%%%%%%%%%%%%%%%%%%
\subsubsection{Classical collision spaces and main theorem}
\unboldmath

\begin{definition}[$\Lambda^\times_{\mathrm{coll}}\Qfrak$]
\label{def:class-coll-loops}
The \textbf{classical collision \underline{loop} space}
$\Lambda^\times_{\mathrm{coll}}\Qfrak$
consists of all loops $q\in W^{1,2}(\SS^1,\Qfrak)$
such that
\begin{itemize}\setlength\itemsep{0ex}
\item[(i)]
  there are at most finitely many collision times
   $T_q:=q^{-1}(0)=\{t_1,\dots,t_N\}$;
\item[(ii)]
  there exist (unique) \footnote{
    Since the set of collision times is finite,
    the functions $\alpha=\alpha_q$ and $e=e_q$ are uniquely
    determined by continuity of $q$.
    }
  continuous maps
  $$
     \alpha_q\colon\SS^1_t\to\SS^1_\C
     ,\qquad
     e_q\colon\SS^1_t\to\R
     ,\qquad
     \beta_q\colon\SS^1_t\to\C
     ,\quad
     \beta|_{T_q}\equiv 0
     ,
  $$
  which away from collisions are solution angle and energy, more
  precisely
  \begin{equation}\label{eq:class-coll-loops}
     \alpha|_{\SS^1\setminus T_q}
     =\frac{q}{\abs{q}}
     ,\qquad
     e|_{\SS^1\setminus T_q}
     =\tfrac12 \abs{\dot q}^2-\frac{1}{\abs{q}}
     ,\qquad
     \beta|_{\SS^1\setminus T_q}
     =\dot\alpha_q\abs{q}
     .
\end{equation}
\end{itemize}
\end{definition}

\begin{definition}[$\mathrm{Coll}\,\Ss$]
\label{def:class-coll-solutions}
The \textbf{space of classical periodic collision orbits}
$\mathrm{Coll}\,\Ss$ consists of all classical collision loops
$q\in\Lambda^\times_{\mathrm{coll}}\Qfrak$ such that
\begin{itemize}\setlength\itemsep{0ex}
\item[(iii)]
  away from collision times $q\colon \SS^1\setminus T_q\to\Qfrak$
  satisfies the classical ODE~(\ref{eq:q}).
\end{itemize}
\end{definition}

\begin{remark}[Motivation for~(\ref{eq:class-coll-loops})]
If one of the first two conditions~(\ref{eq:class-coll-loops}) fails,
then already in the Kepler problem the correspondence with critical
points of the regularized functional will not be true.
In the Kepler problem collisional orbits are just rays.
By continuity of $\alpha$, after collision the particle will stay on
the same ray.
Moreover, in the Kepler problem the energy is constant
on each ray, by continuity of the energy it will be globally constant.
\\
Without the continuity assumption on $\alpha$ and $e$ one could
consider rays, parametrized according to the Kepler equation,
which after collision jump to another ray or change the energy.
\end{remark}

\begin{theorem}[bijection]\label{thm:bijection}
The rescale-square map $\Qq$ in~(\ref{eq:Qq-extension}) is a bijection
between $\Crit\,\Bb$ and $\mathrm{Coll}\,\Ss$.
\end{theorem}

\boldmath
%%%%%%%%%%%%%%%%%%%%%%%%%%%%%%%%%%%
%%%%%%% Subsubsection  %%%%%%%%%%%%%%%%%
%%%%%%%%%%%%%%%%%%%%%%%%%%%%%%%%%%%
\subsubsection{Regularized collision spaces}
\unboldmath

\begin{definition}[$\bar\Ll^\times_{\mathrm{coll}}\Zfrak$]
\label{def:reg-coll-loop-space}
Define the \textbf{regularized collision \underline{loop} space} by
$$
   \bar\Ll^\times_{\mathrm{coll}}\Zfrak
   :=\left\{
   z\in{\color{red}\bar\Ll^\times\Zfrak}
   \;\bigl|\; \forall \tau_*\in z^{-1}(0)\colon
   \text{$\abs{z^\prime_{\tau_*}}=\tfrac{1}{\sqrt{2}\norm{z}^2}$
    and $z^{\prime\prime}_{\tau_*}=0$}
   \right\} .
$$
The elements are called \textbf{regularized collision \underline{loops}}.
Each has finitely many zeroes: first derivatives
at zeroes of $z$ have same length and $\SS^1$ is compact.
\end{definition}

\begin{lemma}\label{le:crit->coll}
$\Crit\,\Bb\subset \bar\Ll^\times_{\mathrm{coll}}\Zfrak
\subset\bar\Ll^\times\Zfrak$.
\end{lemma}

\begin{proof}
Proposition~\ref{prop:q_z-collision} and Lemma~\ref{le:finite-colls}.
\end{proof}

\begin{figure}%[t]
  \centering
\begin{equation*}
\begin{tikzcd} [row sep=normal] %[column sep=small] %
% 0-th row
  &&{\color{gray} \text{regularized}}
  &&
    {\color{gray} \text{classical}}
\\
% 1-st row
&&
\Crit\,\Bb
\arrow[rr, shift left=2, dashed, "\Qq", 
            "\text{\qquad\;\;\; 1:1\; \S\ref{sec:main-bij}}"']
  &&
  \mathrm{Coll}\,\Ss
  \arrow[ll, shift left=2, dashed, "\Zz"]
\\
% 2-nd row
{\color{red}\R}
&
{\color{red}\bar\Ll^\times\Zfrak}
\arrow[l, red, "\Bb"']
\arrow[r, phantom, "\supset"]
&
\bar\Ll^\times_{\mathrm{coll}}\Zfrak
\arrow[u, phantom, sloped, allow upside down, "\supset"]
\arrow[rr, dashed, 
            "\text{\qquad\;\;\;\; $\Qq$\;\;\, \S\ref{sec:ext-Q}}"]
  &&
  \Lambda^\times_{\mathrm{coll}}\Qfrak
  \arrow[u, phantom, sloped, allow upside down, "\supset"]
\\
% 3-rd row
&&
{\color{cyan}\bar\Ll\Zfrak^\times}
\arrow[u, phantom, sloped, allow upside down, "\subset"]
\arrow[rr, shift left=2, dashed, "\Qq", 
            "\text{\qquad\; 1:1\;\; \S\ref{sec:resc-square-non-coll}}"']
  &&
  {\color{cyan}\Ll\Qfrak^\times}
  \arrow[u, phantom, sloped, allow upside down, "\subset"]
  \arrow[r, cyan, "\Ss"]
  \arrow[ll, shift left=2, "\Zz"]
    &
     {\color{cyan}\R}
\end{tikzcd} 
\end{equation*}
  \caption{Correspondence 1:1 of regularized and classical collision
    orbits}
   \label{fig:fig-Qq-ext}
\end{figure}
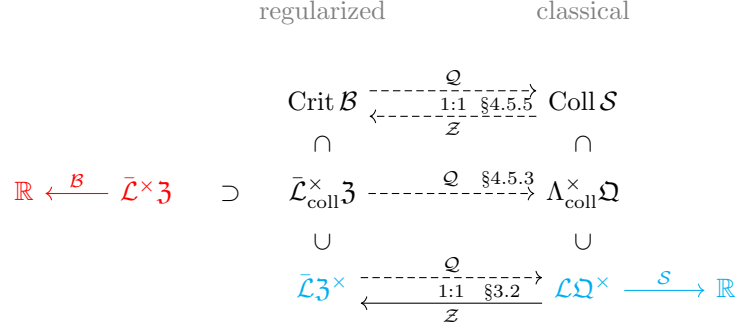

\boldmath
%%%%%%%%%%%%%%%%%%%%%%%%%%%%%%%%%%%
%%%%%%% Subsubsection  %%%%%%%%%%%%%%%%
%%%%%%%%%%%%%%%%%%%%%%%%%%%%%%%%%%%
\subsubsection{Extending $\Qq$}\label{sec:ext-Q}
\unboldmath

The following proposition served us to identify
the proper conditions to define the space
$\Lambda^\times_{\mathrm{coll}}\Qfrak$
to accommodate the critical point equation on the $q$-side.

\begin{proposition}[$\Qq$]\label{prop:extend-Qq}
The rescale-square map $\Qq$ defined on $\Crit\,\Bb$
by~(\ref{eq:Qq-extension})
extends to $\bar\Ll^\times_{\mathrm{coll}}\Zfrak$ and takes
values in $\Lambda^\times_{\mathrm{coll}}\Qfrak$,
as illustrated by Figure~\ref{fig:fig-Qq-ext}.
\end{proposition}

\begin{proof}
Pick $z\in\bar\Ll^\times_{\mathrm{coll}}\Zfrak$ and abbreviate $q=q_z$ where
$q_z:=z^2\circ \tau_z\colon\SS^1\to\Qfrak$.
The zero sets $\Tt_z:=z^{-1}(0)$ and $T_q:=q^{-1}(0)$ are, respectively,
called regularized and classical collision times.
By Lemma~\ref{le:t_z-coll} the map
$t_z\colon \SS^1\to\SS^1$ defined by~(\ref{eq:t_z}) is still smooth,
but if collisions exist only a homeomorphism $\SS^1\to\SS^1$, since
the inverse homeomorphism $\tau_z$ has a derivative~(\ref{eq:inverse-C})
with singularities at the classical collision times $T_q$.
The relation~(\ref{eq:q-qdot}) is still valid
on the finitely many intervals $\SS^1\setminus T_q$.

Note that $q$ is periodic and $T_q$ is finite:
As in Lemma~\ref{le:Qq} one checks that $q_z(1)=q_z(0)$ is $1$-periodic.
Since $q=z^2\circ \tau_z$ and $\tau_z$ is a bijection the zeroes of
$q$ are in bijection to those of $z$, but $z^{-1}$ is a finite set.

\medskip
\noindent
\textbf{Step~1.} Existence of $e_q$.

\begin{proof}
Since $z$ is smooth and vanishes together with its second derivative
at $\tau_*\in \Tt_z$, in symbols $z^{\prime\prime}_{\tau_*}=0$,
by Taylor's theorem there exists a
smooth function $\zeta\colon\R\to\C$ such that for every $\tau$ the
following formula for $z$ holds
\begin{equation}\label{eq:z_tau}
\begin{split}
   z_\tau
   &=z^\prime_{\tau_*}(\tau-\tau_*)+\zeta_\tau (\tau-\tau_*)^3
\\
   z_\tau^\prime
   &=z^\prime_{\tau_*}+3\zeta_\tau (\tau-\tau_*)^2
   +\zeta_\tau^\prime(\tau-\tau_*)^3
   .
\end{split}
\end{equation}
Away from collisions, for $t\in \SS^1\setminus T_q$,
we define $e_q$ as follows, then in equality~3
we use~(\ref{eq:q-qdot}) to get
\begin{equation*}%\label{eq:e_q}
\begin{split}
   e_q(t):
   &=\tfrac12 \abs{\dot q_t}^2-\tfrac{1}{\abs{q_t}}
\\
   &=\tfrac{1}{2\abs{q_t}}
   \left(\abs{q_t}\cdot\abs{\dot q_t}^2-2\right)
\\
   &\stackrel{\text{{\color{gray}3}}}{=}
   \tfrac{1}{2\abs{z^2_{\tau_z(t)}}}
   \left(4\norm{z}^4 \abs{z^\prime_{\tau_z(t)}}^2 -2\right)
\\
   &\stackrel{\text{{\color{gray}4}}}{=}
   \frac{4\norm{z}^4 \Abs{z^\prime_{\tau_*}
   +3\zeta_{\tau_z(t)} (\tau_z(t)-\tau_*)^2
   +\zeta_{\tau_z(t)}^\prime(\tau_z(t)-\tau_*)^3}^2 -2}
      {2\Abs{\left(z^\prime_{\tau_*}(\tau_z(t)-\tau_*)
   +\zeta_{\tau_z(t)} (\tau_z(t)-\tau_*)^3\right)^2}}
\\
   &\stackrel{\text{{\color{gray}5}}}{=}
   \frac{4\norm{z}^4\abs{z^\prime_{\tau_*}}^2
      +(\tau_z(t)-\tau_*)^2\gamma_t -2}
   {2(\tau_z(t)-\tau_*)^2\abs{z^\prime_{\tau_*}}^2+(\tau_z(t)-\tau_*)^4\delta_t}
\\
   &\stackrel{\text{{\color{gray}6}}}{=}
   \frac{2
      +(\tau_z(t)-\tau_*)^2\gamma_t -2}
   {2
     (\tau_z(t)-\tau_*)^2\abs{z^\prime_{\tau_*}}^2+(\tau_z(t)-\tau_*)^4\delta_t}
\\
   &\stackrel{\text{{\color{gray}7}}}{=}
   \frac{\gamma_t \norm{z}^4}
   {1+(\tau_z(t)-\tau_*)^2\delta_t \norm{z}^4}
   .
\end{split}
\end{equation*}
Equality 4 uses~(\ref{eq:z_tau}).
Equality 5 abbreviates two continuous maps $\R\to\R$ by
\begin{equation*}
\begin{split}
   \gamma_t:
   &=4\norm{z}^4
   \biggl(6\inner{z^\prime_{\tau_*}}{\zeta_{\tau_z(t)}}_0
   +2(\tau_z(t)-\tau_*)\inner{z^\prime_{\tau_*}}{\zeta^\prime_{\tau_z(t)}}_0
   \\
   &\quad
   +9(\tau_z(t)-\tau_*)^2\Abs{\zeta_{\tau_z(t)}}^2
   +6(\tau_z(t)-\tau_*)^3
   \inner{\zeta_{\tau_z(t)}}{\zeta^\prime_{\tau_z(t)}}_0
   \\
   &\quad
   +(\tau_z(t)-\tau_*)^4\Abs{\zeta^\prime_{\tau_z(t)}}^2
   \biggr)
\end{split}
\end{equation*}
and
\begin{equation*}
\begin{split}
   \delta_t:
   &=4\inner{z^\prime_{\tau_*}}{\zeta_{\tau_z(t)}}_0
   +2(\tau_z(t)-\tau_*)^2\Abs{\zeta_{\tau_z(t)}}
   .
\end{split}
\end{equation*}
Equality 6 uses that $4\norm{z}^4\abs{z^\prime_{\tau_*}}^2=2$
since $z\in\bar\Ll^\times_{\mathrm{coll}}\Zfrak$.
Equation~7 in the $e_q$ calculation above
extends continuously to $t_*$.
This proves Step~1.
\end{proof}

\medskip
\noindent
\textbf{Step~2.} Existence of $\alpha_q$.

\begin{proof}
Away from collisions, for $t\in \SS^1\setminus T_q$,
we define $\alpha_q$ as follows
\begin{equation*}%\label{eq:alpha_q}
\begin{split}
   \alpha_q(t):
   &=\frac{q_t}{\Abs{q_t}}
   =\tfrac{z^2_{\tau_z(t)}}{\Abs{z^2_{\tau_z(t)}}}
\\
   &\stackrel{\text{\color{gray}3}}{=}
   \frac{
   2\left(z^\prime_{\tau_*}(\tau_z(t)-\tau_*)
      +\zeta_{\tau_z(t)}(\tau_z(t)-\tau_*)^3\right)^2
}
{
   2\Abs{\left(
      z^\prime_{\tau_*}(\tau_z(t)-\tau_*)
      +\zeta_{\tau_z(t)} (\tau_z(t)-\tau_*)^3\right)^2}
}
\\
   &\stackrel{\text{\color{gray}4}}{=}
   \frac{(\tau_z(t)-\tau_*)^2}{(\tau_z(t)-\tau_*)^2}\cdot
   \frac{2\left(z^\prime_{\tau_*}+\zeta_{\tau_z(t)} (\tau_z(t)-\tau_*)^2\right)^2}
   {2\abs{z^\prime_{\tau_*}}^2+(\tau_z(t)-\tau_*)^2\delta_t}
\\
   &\stackrel{\text{\color{gray}5}}{=}
   \frac{2\left(z^\prime_{\tau_*}+\zeta_{\tau_z(t)} (\tau_z(t)-\tau_*)^2\right)^2}
   {\norm{z}^{-4}+(\tau_z(t)-\tau_*)^2\delta_t}
   .
\end{split}
\end{equation*}
Equation 2 is by definition of $q$.
Equation 3 uses~(\ref{eq:z_tau}).
In equation 4 for the denominator we used the previous definition of
$\delta_t$.
Equality 5 uses that $4\norm{z}^4\abs{z^\prime_{\tau_*}}^2=2$
since $z\in\bar\Ll^\times_{\mathrm{coll}}\Zfrak$.

The right hand side
extends continuously to $t_*$.
This proves Step~2.
\end{proof}

\medskip
\noindent
\textbf{Step~3.} Existence of $\beta_q$.

\begin{proof}
Away from collisions, for $t\in \SS^1\setminus T_q$,
we define $\beta_q$ as follows
\begin{equation*}%\label{eq:alpha_q}
\begin{split}
   \beta_q(t):
   &=\abs{q_t}\;\p_t\tfrac{q_t}{\Abs{q_t}}
   =\abs{z_{\tau_z(t)}^2}
   \left(\p_\tau\frac{z_{\tau_z(t)}}{\bar z_{\tau_z(t)}}\right)
{\color{cyan}
   \dot\tau_z(t)
}
\\
   &\stackrel{3}{=}
   \abs{z_{\tau_z(t)}}^2\cdot
{\color{cyan}
   \tfrac{\norm{z}^2}{\Abs{z_{\tau_z(t)}}^2}
}
   \cdot
   \tfrac{z^\prime_{\tau_z(t)}\bar z_{\tau_z(t)}-z_{\tau_z(t)}\bar z_{\tau_z(t)}^\prime}
      {\bar z_{\tau_z(t)}^2}
\\
   &\stackrel{4}{=}
   \norm{z}^2
   \cdot
   \tfrac{(\tau-\tau_*)^32i}{(\tau-\tau_*)^2}
   \tfrac{2\Im(\bar z_*^\prime \zeta_\tau)
   +(\tau-\tau_*)\Im(\bar z_*^\prime\zeta_\tau^\prime)
   +(\tau-\tau_*)^2\Im(\bar \zeta_\tau\zeta_\tau^\prime)}
      {\bar z_*^\prime\bar z_*^\prime
      +2 \bar z_*^\prime\bar\zeta_\tau (\tau-\tau_*)^2
      +\bar\zeta_\tau\bar\zeta_\tau (\tau-\tau_*)^4}
\\
   &\stackrel{5}{=}
   2i\left(\tau-\tau_*\right)\norm{z}^2
   \cdot
   \frac{
   2\Im(\bar z_*^\prime \zeta_\tau)
   +(\tau-\tau_*)\Im(\bar z_*^\prime\zeta_\tau^\prime)
   +(\tau-\tau_*)^2\Im(\bar \zeta_\tau\zeta_\tau^\prime)
   }
      {
      {\color{magenta}\bar z_*^\prime\bar z_*^\prime}
      +2 \bar z_*^\prime\bar\zeta_\tau (\tau-\tau_*)^2
      +\bar\zeta_\tau\bar\zeta_\tau (\tau-\tau_*)^4}
   .
\end{split}
\end{equation*}
Equation 2 is by definition of $q$.
Equality 3 is by the quotient rule and~(\ref{eq:inverse-C}).
In equality 4 we replace $z$ and $z^\prime$ according
to~(\ref{eq:z_tau}), 
multiply out and observe that the terms linear
in $\tau-\tau_*$ cancel and so do those of fifth power. To simplify we
used that $(a+ib)-(a-ib)=2ib=2i\Im(a+ib)$.
The right hand side
extends continuously to $t_*$.
This proves Step~3.
\end{proof}

\medskip
\noindent
\textbf{Step~4.} Given $z\in\bar\Ll^\times_{\mathrm{coll}}\Zfrak$, let
$q:=\Qq(z):=z^2\circ \tau_z$, then
\begin{equation}\label{eq:hgghh5677}
   \int_0^1\tfrac{1}{\abs{q_t}}\, dt=\frac{1}{\norm{z}^2}
   .
\end{equation}

\begin{proof}
The open intervals $I_j=(t_j,t_{j+1})$ between collision times,
Proposition~\ref{prop:q_z-collision}, cover $\SS^1\setminus T_q$.
The corresponding intervals $\Ii_j=(\tau_j,\tau_{j+1})$ are
given in Lemma~\ref{le:finite-colls} where $\tau_j=\tau(t_j)$. Then
\begin{equation*}
\begin{split}
   \int_0^1\tfrac{1}{\abs{q_t}}\, dt
%1
   &=\sum_{j=1}^N \int_{I_j}\tfrac{1}{\abs{q_t}}\, dt
%2
   =\sum_{j=1}^N \int_{I_j}\tfrac{1}{\abs{z\circ \tau_z(t)}^2}\,dt
%3
   \stackrel{\color{gray}3}{=}
   \sum_{j=1}^N \int_{\Ii_j}
   \tfrac{1}{\abs{z(\tau)}^2}
   \tfrac{\abs{z(\tau)}^2\, d\tau}{\norm{z}^2}
%4
   =\tfrac{1}{\norm{z}^2}
   .
\end{split}
\end{equation*}
Equality 3 is change of variables $\tau=\tau_z(t)$
with~(\ref{eq:inverse-C}).
This proves Step~4.
\end{proof}

\medskip
\noindent
\textbf{Step~5.} 
  The map $q\colon\SS^1\to\Qfrak$ is of Sobolev class $W^{1,2}$.

\begin{proof}
Note that $q=z^2\circ \tau_z$ is continuous
since $z$ is smooth and $\tau_z$ is a homeomorphism 
by Lemma~\ref{le:t_z-coll}.
As $q$ is continuous, it suffices to show that $\dot q$ is in~$L^2$.
Using the definition of $e_q$ and Step~4 we compute
\begin{equation*}
\begin{split}
   \norm{\dot q}^2
   &=2\int_0^1 e_q(t)\, dt+2\int_0^1 \tfrac{1}{\abs{q_t}}\, dt
   \stackrel{\text{(\ref{eq:hgghh5677})}}{=}
   2\int_0^1 e_q(t)\, dt+\tfrac{2}{\norm{z}^2}
   .
\end{split}
\end{equation*}
As $e_q$ is $C^0$ by Step~1, the right hand side is finite. This
proves Step~5.
\end{proof}
This proves Proposition~\ref{prop:extend-Qq}.
\end{proof}
\newpage

\begin{corollary}\label{cor:hjghjgh75567}
The rescale-square map $\Qq$ in~(\ref{eq:Qq-extension})
is well defined as a map
$$
   \Qq\colon\Crit\,\Bb
   \to \mathrm{Coll}\,\Ss
   .
$$
\end{corollary}

\begin{proof}
Lemma~\ref{le:crit->coll},
Proposition~\ref{prop:q_z-collision},
Proposition~\ref{prop:extend-Qq}.
\end{proof}

In view of Corollary~\ref{cor:hjghjgh75567},
in order to prove Theorem~\ref{thm:bijection}
it suffices to construct an inverse map,
notation $\Zz\colon \mathrm{Coll}\,\Ss\to \Crit\,\Bb$.

\boldmath
%%%%%%%%%%%%%%%%%%%%%%%%%%%%%%%%%%%
%%%%%%% Subsubsection  %%%%%%%%%%%%%%%%
%%%%%%%%%%%%%%%%%%%%%%%%%%%%%%%%%%%
\subsubsection{Inverse map $\Zz$}\label{sec:inverse-Zz}
\unboldmath

\begin{definition}[$\Zz$]
For a classical periodic collision orbit
$q\in \mathrm{Coll}\,\Ss$
define
\begin{equation*}%\label{eq:Rr-inverse-C-collision}
   \Zz(q):=z_q:=q^\frac12\circ \tau_{1/\sqrt{q}}
   ,\qquad
   1/\sqrt{q}=[t\mapsto 1/\sqrt{q_t}]
   ,
\end{equation*}
where $\tau_{1/\sqrt{q}}\colon \SS^1_\tau\to\SS^1_t$ is to be defined yet
in Corollary~\ref{cor:ghjg67t87gh} below.
Taking square root is well defined on the quotient space
$\bar\Ll^\times\Zfrak$, see Remark~\ref{rem:cx-square-root}.
\end{definition}

The goal of Subsection~\ref{sec:inverse-Zz}
is to show that $z_q\in\Crit\,\Bb$.

\boldmath
%%%%%%%%%%%%%%%%%%%%%%%%%%%%%%%%%%%
%%%%%%% Subsubsection  %%%%%%%%%%%%%%%%
%%%%%%%%%%%%%%%%%%%%%%%%%%%%%%%%%%%
\subsubsection*{Rescaling of collisional loops $1/\sqrt{q}$}
\unboldmath

\begin{lemma}
Let $q\in \Lambda^\times_{\mathrm{coll}}\Qfrak$.
Then the integral $\int_0^1\frac{1}{\abs{q_t}}\, dt$ is finite.
\end{lemma}

\begin{proof}
Since $q\in W^{1,2}$ we have that $\dot q\in L^2$, hence
the following integral (obtained by integrating the restriction of $e_q$)
\begin{equation*}
\begin{split}
   \int_0^1\tfrac{1}{\abs{q_t}}\, dt
%1
   &=\tfrac12 \norm{\dot q}^2-    \int_0^1 e_q(t) \, dt
\end{split}
\end{equation*}
is finite since the function $e_q\colon\SS^1\to\R$ is continuous.
\end{proof}

In analogy with the non-collisional case~(\ref{eq:Rr-inverse-C}),
although $t\mapsto 1/q(t)^{\frac12}$ is not defined
at collision times, hence it is only a 'loop' up to finitely many
points, integration still makes sense and leads to a strictly monotone
increasing function.
Indeed we still rescale time according to this 'loop',
notation $1/\sqrt{q}$, by defining
\begin{equation}\label{eq:class-time-tau-C-collision}
   t_{1/\sqrt{q}}\colon\SS^1_t\to\SS^1_t
   ,\quad
   t\mapsto
   \frac{\int_0^t\frac{1}{\abs{q_s}}\, dt}{\int_0^1\frac{1}{\abs{q_s}}\, ds}
\end{equation}
for every $t\in\R$.
It is well defined on $\SS^1=\R/\Z$ since, by integral properties,
the map $t_{1/\sqrt{q}}$ takes $0$ to $0$ and $1$ to $1$.
Also $t_{1/\sqrt{q}}$ strictly monotone increases:
the integrand is $>0$, even $\infty$ at the
finitely many collisions,
and $\int_0^1\frac{1}{\abs{q_s}}\, ds<\infty$.

\begin{lemma}\label{le:gh877g}
Suppose that $q\in \Lambda^\times_{\mathrm{coll}}\Qfrak$,
notation $q_t=r_t e^{i\vartheta_t}$, has a collision at time $t=0$,
i.e. $q_0=0$. Then for every $\eps>0$ there exists $\delta_\eps>0$
such that
\begin{equation*}%\label{eq:gh877g}
   (\tfrac{3}{2}\sqrt{2-\eps})^{\frac{2}{3}} t^{\frac{2}{3}}
   \le r_t\le (\tfrac{3}{2}\sqrt{2+\eps})^{\frac{2}{3}} t^{\frac{2}{3}}
\end{equation*}
for $t\in[-\delta_\eps,\delta_\eps]$.
\end{lemma}

\begin{proof}
By definition of the energy in polar coordinates
$q_t=r_t e^{i\vartheta_t}$, $e^{i\vartheta_t}=\alpha_t$, we have
the energy equation
$
   \tfrac12 \dot r^2+\tfrac12 r^2\dot\vartheta^2-\tfrac{1}{r}=e
$.

Since $r_0=0$ and $r>0$, there exists $\delta_0>0$ such that $\dot
r_t\ge 0$ for every $t\in[0,\delta_0]$.
Thus, solving for $\dot r$ we obtain the ODE
$$
   \dot r
   =\sqrt{\dot r^2}
   =\sqrt{2e+\tfrac{2}{r}-r^2\dot\vartheta^2}
   =\sqrt{2e+\tfrac{2}{r}-\abs{\beta}^2}
    .
$$
As $e$ and $\beta$ are continuous, 
for every $\eps\in(0,2)$, we can further
choose $\delta_\eps>0$ small enough such that
$$
   -\frac{\eps}{r_t}
   \le 2e_t-\abs{\beta_t}^2
   \le \frac{\eps}{r_t}
$$
whenever $t\in[0,\delta_\eps]$.
So from the above inequality, on $[0,\delta_\eps]$
we get the inequality
$$
   \sqrt{\tfrac{2-\eps}{r}}
   \le\dot r
   \le \sqrt{\tfrac{2+\eps}{r}}
   .
$$
For $a>0$ the ODE $\dot R=\tfrac{a}{\sqrt{R}}$ on $[0,\infty)$
with $R_0=0$ has the solution
$
   R_t^a=(\tfrac{3}{2}a t)^{\frac{2}{3}}
$.
The theory of sub- and super-solutions
for $R_t^{\sqrt{2-\eps}}$ and $R_t^{\sqrt{2+\eps}}$ tells that
\begin{equation*}
   (\tfrac{3}{2}\sqrt{2-\eps})^{\frac{2}{3}} t^{\frac{2}{3}}=R_t^{\sqrt{2-\eps}}
   \le r_t\le R_t^{\sqrt{2+\eps}}=(\tfrac{3}{2}\sqrt{2+\eps})^{\frac{2}{3}} t^{\frac{2}{3}}
\end{equation*}
whenever $t\in[0,\delta_\eps]$; see e.g.~\cite[Le.\,1.2]{teschl:2012a}.

A similar estimate holds for negative times where $\dot r_t\le 0$
so that, maybe after shrinking $\eps>0$, we can assume
that the previous inequalities hold for $t\in[-\delta_\eps,\delta_\eps]$.
This proves Lemma~\ref{le:gh877g}.
\end{proof}

\begin{proposition}\label{prop:gh77877g}
Let $q\in \Lambda^\times_{\mathrm{coll}}\Qfrak$.
Then $t_{1/\sqrt{q}}\colon\SS^1_t\to\SS^1_t$ is continuous.
\end{proposition}

\begin{proof}[Proof of Proposition~\ref{prop:gh77877g}]
Away from collisions, namely for $t\in \SS^1\setminus \Tt_q$,
the map $t_{1/\sqrt{q}}$ is differentiable at $t$, hence
continuous. Indeed the derivative is given by
\begin{equation}\label{eq:bhjbhjbhj7458}
   \dot t_{1/\sqrt{q}}(t)
   =\frac{1}{\abs{q_t}\cdot\norm{\tfrac{1}{q}}_{L^1}}
   .\qquad
   \text{Thus
   $\tau_{1/\sqrt{q}}^\prime(\tau)
   =\Abs{q\circ \tau_{1/\sqrt{q}}(\tau)}\cdot\norm{\tfrac{1}{q}}_{L^1}$}
   .
\end{equation}
It remains to discuss continuity of $t_{1/\sqrt{q}}$ at collision times.
After shifting time
we can assume without loss of generality
that the collision occurs at time $0$.

\smallskip
\noindent
We next show continuity of $f(t):=\int_0^t\frac{1}{\abs{q_s}}\, ds$
at zero from the right.
This means that for every $\eps>0$ there exists $\mu_\eps>0$
such that $f(t)\le\eps$ whenever $t\in[0,\mu_\eps]$.

By Lemma~\ref{le:gh877g} there exists $\delta>0$ such that
$r_t\ge (\frac{3}{2})^{\frac{2}{3}} t^{\frac{2}{3}}\ge t^{\frac{2}{3}}$ 
whenever $t\in[0,\delta]$.
Hence if $t\in[0,\delta]$ we estimate
\begin{equation*}
\begin{split}
   f(t)
   =\int_0^t\frac{1}{r_s}\, ds
   \le \int_0^t s^{-\frac{2}{3}}\, ds
   &=\left. 3 s^{\frac{1}{3}}\right|_{s=0}^{s=t}
   =3 t^{\frac{1}{3}}.
\end{split}
\end{equation*}
Choosing $\mu_\eps=(\eps/3)^3$ completes the proof of continuity from
the right.
Continuity from the left follows similarly.
This proves Proposition~\ref{prop:gh77877g}.
\end{proof}

\begin{corollary}[inverse]\label{cor:ghjg67t87gh}
Let $q\in \Lambda^\times_{\mathrm{coll}}\Qfrak$.
Then $t_{1/\sqrt{q}}$ in~(\ref{eq:class-time-tau-C-collision})
has an inverse
$$
   \tau_{1/\sqrt{q}}={t_{1/\sqrt{q}}}^{-1}\colon \SS^1_\tau\to\SS^1_t
   ,\quad
  \tau\mapsto \tau_{1/\sqrt{q}}(\tau)
$$
which is strictly monotone increasing, too.
\end{corollary}

\begin{proof}
As $t_{1/\sqrt{q}}$ is strictly monotone increasing, $C^0$ by
Proposition~\ref{prop:gh77877g}, 
it admits an inverse with the same properties;
see footnote to~(\ref{eq:t-inverse}).
\end{proof}

\boldmath
%%%%%%%%%%%%%%%%%%%%%%%%%%%%%%%%%%%
%%%%%%% Subsubsection  %%%%%%%%%%%%%%%%
%%%%%%%%%%%%%%%%%%%%%%%%%%%%%%%%%%%
\subsubsection*{Loops $q$: Speed $\abs{z_q^\prime}$ at collisions
-- equal non-zero speed}
\unboldmath

\begin{lemma}\label{le:gjhgh4798}
Pick a collision loop $q\in \Lambda^\times_{\mathrm{coll}}\Qfrak$.
Set $z_q:=q^\frac12\circ \tau_{1/\sqrt{q}}$. Then
\begin{equation}\label{eq:1/q-L1}
   {\color{gray}\norm{\tfrac{1}{q}}_{L^1} =\,}
   \int_0^1\frac{1}{\abs{q_t}}\, dt
   =\frac{1}{\norm{z_q}^2}
   .
\end{equation}
\end{lemma}

\begin{proof}
Variable substitution $\sigma(s)=t_{1/\sqrt{q}}(s)$
and~(\ref{eq:bhjbhjbhj7458}) for $t^\prime_{1/\sqrt{q}}(s)$ yield
\begin{equation*}
\begin{split}
   \norm{z_q}^2
   =\int_0^1\abs{q\circ\underbrace{\tau_{1/\sqrt{q}}(\sigma)}_{=s}}\, d\sigma
   =\int_{s(0)=0}^{s(1)=1}\abs{q_s}\,\frac{ds}{\abs{q_s}\cdot\norm{\tfrac{1}{q}}_{L^1}}
   =\frac{1}{\norm{\tfrac{1}{q}}_{L^1}}
\end{split}
\end{equation*}
and this proves Lemma~\ref{le:gjhgh4798}.
\end{proof}

\begin{lemma}[first derivative]\label{le:1st-deriv}
Pick a collision loop $q\in \Lambda^\times_{\mathrm{coll}}\Qfrak$ and set
$z_q:=q^\frac12\circ \tau_{1/\sqrt{q}}$.
Then all collisions $\tau_*$ of $z_q$ happen at equal speed
\begin{equation*}%\label{eq:coll-77}
   \forall \tau_*\in z_q^{-1}(0)\colon
   \quad
   \lim_{\tau\to\tau_*}
   \Abs{z_q^\prime(\tau)}
   =\frac{1}{\sqrt{2}\norm{z_q}^2}
   \stackrel{\text{\rm(\ref{eq:1/q-L1})}}{=}
   \frac{1}{\sqrt{2}}\norm{\tfrac{1}{q}}_{L^1} 
   .
\end{equation*}
\end{lemma}

\begin{proof}
At any non-collision time $\tau$, simplifying $f(t)g(t)=:(fg)_t$, we have
\begin{equation}\label{eq:joker-1}
\begin{split}
   \Abs{z_q^\prime(\tau)}
%1
   &=\Abs{
   \frac{\dot q_{\tau_{1/\sqrt{q}}(\tau)}\cdot{\color{brown}\dot\tau_{1/\sqrt{q}}(\tau)}}
   {2q^{\frac12}_{\tau_{1/\sqrt{q}}(\tau)}}
   }
%2
   \stackrel{{\color{gray}2}}{=}
   \Biggl(
   \frac{
   \Abs{\sqrt{2e_q +\tfrac{2}{\abs{q}}}}
   \cdot{\color{brown}\Abs{q}}
   }
   {2\abs{q}^{\frac12}}
   \Biggr)_{\tau_{1/\sqrt{q}}(\tau)} \cdot
   \color{brown}\norm{\tfrac{1}{q}}_{L^1}
\\
%3
   &\stackrel{{\color{gray}3}}{=}
   \left(
   \frac{\sqrt{{\color{red}e_q \Abs{q}}+1}}{\sqrt{2}}
   \right)_{{\color{red}\tau_{1/\sqrt{q}}(\tau)}}
   \cdot \frac{1}{\norm{z_q}^2}
   .
\end{split}
\end{equation}
Equality 2, also 3, uses that $\abs{\sqrt{\cdot}}=\sqrt{\abs{\cdot}}$ on $\C$.
Equality 2 also uses the energy identity~(\ref{eq:class-coll-loops})
resolved for $\abs{\dot q}$ and~(\ref{eq:bhjbhjbhj7458}) for
{\color{brown}$\dot\tau_{1/\sqrt{q}}$}.
In equality 3 the common factor $\abs{q}^{\frac12}$ in nominator and
denominator cancels and the remaining $\abs{q}^{\frac12}$ we multiply
with the sum under the square root.
We replaced $\norm{\tfrac{1}{q}}_{L^1}$ by~(\ref{eq:1/q-L1}).

In the limit as time $\tau$ approximates a collision time $\tau_*$,
while the energy is bounded $q$ approaches zero, hence
\begin{equation}\label{eq:e-limit}
{\color{red}
   e_q\circ\tau_{1/\sqrt{q}}(\tau)\cdot\Abs{q\circ\tau_{1/\sqrt{q}}(\tau)}
}
   \stackrel{\tau\to\tau_*}{\longrightarrow} {\color{red}0}
\end{equation}
and this proves Lemma~\ref{le:1st-deriv}.
\end{proof}

\boldmath
%%%%%%%%%%%%%%%%%%%%%%%%%%%%%%%%%%%
%%%%%%% Subsubsection  %%%%%%%%%%%%%%%%
%%%%%%%%%%%%%%%%%%%%%%%%%%%%%%%%%%%
\subsubsection*{Solutions $q$: Second derivative $z_q^{\prime\prime}$
vanishes at collisions}
\unboldmath

\begin{lemma}[second derivative]\label{le:coll-7887}
If $q\in \mathrm{Coll}\,\Ss$
and $z_q:=q^\frac12\circ \tau_{1/\sqrt{q}}$, then
$$
   \lim_{\tau\to\tau_*}
   z_q^{\prime\prime}(\tau)
   =0
$$
at any collision time $\tau_*$.
\end{lemma}

\begin{proof}
Let $\tau$ be a non-collision time.
We simplify notation $f(t)g(t)=:(fg)_t$.
For $z_q^\prime$ use the first line in~(\ref{eq:joker-1})
to get equality one
\begin{equation}\label{eq:z-prime-prime=0}
\begin{split}
   z_q^{\prime\prime}(\tau)
%1
   &=\frac{d}{d\tau}
   \left(
   \frac{\dot q_{\tau_{1/\sqrt{q}}(\tau)}\cdot \dot\tau_{1/\sqrt{q}}(\tau)}
   {2q^{\frac12}_{\tau_{1/\sqrt{q}}(\tau)}}
   \right)
%2
   \stackrel{{\color{gray}2}}{=}
   \tfrac12 \frac{d}{d\tau}
   \left(
   \frac{\dot q\Abs{q}}
   {q^{\frac12}}
   \right)_{\tau_{1/\sqrt{q}}(\tau)}
   \norm{\tfrac{1}{q}}_{L^1}
\\
%3
   &\stackrel{{\color{gray}3}}{=}
   \tfrac12 \norm{\tfrac{1}{q}}_{L^1}
   \tfrac{d}{d\tau}
   \left(
{\color{brown} 
   \dot q\circ\tau_{1/\sqrt{q}}
}
   \,\cdot\,
{\color{cyan} 
   (\bar q)^{\frac12}\circ\tau_{1/\sqrt{q}}
}
   \right)_\tau
\\
%4
   &\stackrel{{\color{gray}4}}{=}
   \tfrac12 \norm{\tfrac{1}{q}}_{L^1}
   \Bigl(
{\color{brown}
   \ddot q_{\tau_{1/\sqrt{q}}(\tau)}
\overbrace{
   \abs{q_{\tau_{1/\sqrt{q}}(\tau)}}\cdot\norm{\tfrac{1}{q}}_{L^1}
}^{=\dot\tau_{1/\sqrt{q}}(\tau)}
}
{\color{cyan}
   (\bar q)^{\frac12}_{\tau_{1/\sqrt{q}}(\tau)}
}
\\
   &\qquad
   +
{\color{brown} 
   \dot q_{\tau_{1/\sqrt{q}}(\tau)}
}
{\color{cyan}
   \tfrac12 (\bar q)^{-\frac12}_{\tau_{1/\sqrt{q}}(\tau)}
   \dot {\bar q}_{\tau_{1/\sqrt{q}}(\tau)}
\overbrace{
   \abs{q_{\tau_{1/\sqrt{q}}(\tau)}}\cdot\norm{\tfrac{1}{q}}_{L^1}
}^{=\dot\tau_{1/\sqrt{q}}(\tau)}
}
   \Bigr)
\\
%5
   &\stackrel{{\color{gray}5}}{=}
   \tfrac12 \norm{\tfrac{1}{q}}_{L^1}^2
   \Bigl(
\underbrace{
   \ddot q
}_{\text{(\ref{eq:q})}}
\underbrace{
   \Abs{q}
   (\bar q)^{\frac12}
}_{=\bar q q^{\frac12}}
   +
\underbrace{
   \tfrac12\abs{\dot q}^2
}_{e_q+\tfrac{1}{\abs{q}}}
   q^{\frac12}
  \Bigr)_{\tau_{1/\sqrt{q}}(\tau)}
\\
%6
   &\stackrel{{\color{gray}6}}{=}
   \tfrac12 \norm{\tfrac{1}{q}}_{L^1}^2
   \Bigl(
   -
\underbrace{
   B|_q
}_{\text{bd. $\SS^1\times\im q$}}
\underbrace{
{\color{red}\;
   (j_0 \dot q)
\;}
   \bar q q^{\tfrac12}
}_{\text{(\ref{eq:joker-2})}}
   -
\underbrace{
   \AAAdot|_q
}_{\text{bd. $\SS^1\times\im q$}}
\underbrace{
   \bar q q^{\tfrac12}
}_{\to 0}
\\
   &\qquad\qquad\qquad\qquad\qquad\qquad
{\color{gray}
\underbrace{
   -\tfrac{q}{\abs{q}^3} \,\bar q q^{\tfrac12}
   +\tfrac{1}{\abs{q}} q^{\tfrac12}
}_{=0}
}
   +
\underbrace{
   e_q  q^{\tfrac12}
}_{\stackrel{\tau\to\tau_*}{\longrightarrow} 0}
  \Bigr)_{\tau_{1/\sqrt{q}}(\tau)}
\\
   &\stackrel{\tau\to\tau_*}{\longrightarrow}
   0
   .
\end{split}
\end{equation}
Equality~2 uses~(\ref{eq:bhjbhjbhj7458}) for $\dot\tau_{1/\sqrt{q}}$.
Equality 3 replaces the complex absolute value
$\abs{q}$ by $q^{1/2}(\bar q)^{1/2}$,
then we cancel factor one with the denominator.
Equality~4 is by the chain rule and~(\ref{eq:bhjbhjbhj7458}) for
$\dot\tau_{1/\sqrt{q}}$.
Equality~5 just summarized in simplified notation.
Equality~6 is the key step, here we bring in the energy function $e_q$
from~(\ref{eq:class-coll-loops}) and we use that $q$ solves the classical
ODE~(\ref{eq:q}) replacing $\ddot q$ accordingly.
As a consequence a marvelous elimination of the two singular potential
terms takes place. 

Both maps $B$ and $\AAAdot$ are
continuous on $\SS^1\times \im q$, hence bounded.
The product $e_q q^{\frac12}$ tends to zero, because energy
$e_q\colon\SS^1\to\R$ is continuous, hence bounded, and $q$
collides with the origin, hence is zero, at collision times.

Furthermore, we profit from the additional factor $\abs{\bar q}=\abs{q}$
in ${\color{red}(j_0\dot q)} \bar q q^{\frac12}$, as compared
to~(\ref{eq:joker-1}), to obtain zero in the limit
\begin{equation}\label{eq:joker-2}
\begin{split}
   \Abs{{\color{red}(j_0 \dot q)} \bar q q^{\frac12}}_{\tau_{1/\sqrt{q}}(\tau)}
%1
   &\stackrel{{\color{gray}1}}{=}
   \left(
   \Abs{\sqrt{2e_q +\tfrac{2}{\abs{q}}}}
   \Abs{q}^{\frac12}\Abs{q}
   \right)_{\tau_{1/\sqrt{q}}(\tau)}
\\
%2
   &\stackrel{{\color{gray}2}}{=}
   \sqrt{2}
\underbrace{
   \left(
   \sqrt{{\color{red}e_q \Abs{q}}+1}
   \right)_{{\color{red}\tau_{1/\sqrt{q}}(\tau)}}
}_{\stackrel{\tau\to\tau_*}{\longrightarrow} 1\text{, by~(\ref{eq:e-limit})}}
\underbrace{
   \abs{q _{\tau_{1/\sqrt{q}}(\tau)}}
}_{\stackrel{\tau\to\tau_*}{\longrightarrow} 0}
\\
   &\stackrel{\tau\to\tau_*}{\longrightarrow}
   0
   .
\end{split}
\end{equation}
Equality~1 uses that the rotation $j_0\simeq i$ is an isometry,
that $\abs{\bar q}=\abs{q}$,
that $\abs{\sqrt{\cdot}}=\sqrt{\abs{\cdot}}$ on $\C$.
It also uses the energy identity~(\ref{eq:class-coll-loops})
resolved for $\dot q$.
In equality~2 we multiply out.
This concludes the proof of Lemma~\ref{le:coll-7887}.
\end{proof}

\begin{lemma}[$C^2$]\label{le:z_q-C2}
Let $q\in \mathrm{Coll}\,\Ss$.
Then $z_q:=q^\frac12\circ \tau_{1/\sqrt{q}}$
is of class $C^2$.
\end{lemma}

\begin{proof}
By Lemma~\ref{le:coll-7887},
it suffices to show that the first derivative of $z_q$ is continuous
at collisions.

\smallskip\noindent
The map $\tau\mapsto\abs{z_q^\prime(\tau)}$ is continuous,
by~(\ref{eq:joker-1}), and at each collision time we have
$\abs{z_q^\prime(\tau_*)}=\frac{1}{\sqrt{2}}\norm{q^{-1}}_{L^1}=:c_q>0$,
by Lemma~\ref{le:1st-deriv}.
Hence there exists $\eps_0>0$ such that for each collision time
$\tau_*$ there is the bound $\abs{z_q^\prime(\tau_*+\eps)}>c_q/2>0$
whenever $\eps\in(-\eps_0,\eps_0)$.
In particular, there exists a well defined angle function
$\phi\colon(-\eps_0,\eps_0)\to \SS^1_\C$ such that
$$
   z_q^\prime(\tau_*+\eps)
   =\abs{z_q^\prime(\tau_*+\eps)}\cdot\phi(\eps)
$$
whenever $\eps\in(-\eps_0,\eps_0)$
and it remains to check that $\phi$ is continuous at zero.
Since $z_q(\tau_*)=0$, by Taylor's theorem, there exists
$r_\eps\in[0,1]$ such that $z_q(\tau_*+\eps)=z_q^\prime(\tau_*+r_\eps\eps)$.
This and the previous displayed equation yield equality one
\begin{equation*}
\begin{split}
   \phi(r_\eps\cdot\eps)
   &=\frac{z_q(\tau_*+\eps)}{\abs{z_q^\prime(\tau_*+r_\eps\cdot\eps)}}
\\
   &=\frac{q^{\frac12}\circ \tau_{1/\sqrt{q}}(\tau_*+\eps)}
   {\abs{z_q^\prime(\tau_*+r_\eps\cdot\eps)}}
\\
   &\stackrel{\color{gray}3}{=}
   {\alpha_q}^{\frac12}\circ \tau_{1/\sqrt{q}}(\tau_*+\eps)
   \cdot \frac{\abs{q^{\frac12}\circ \tau_{1/\sqrt{q}}(\tau_*+\eps)}}
   {\abs{z_q^\prime(\tau_*+r_\eps\cdot\eps)}}
\\
   &\stackrel{\color{gray}4}{=}
   {\alpha_q}^{\frac12}\circ \tau_{1/\sqrt{q}}(\tau_*+\eps)
   .
\end{split}
\end{equation*}
Equality~3 holds by definition~(\ref{eq:class-coll-loops}) of $\alpha$.
Equality~4 uses the fact that $\phi$ and $\alpha$ take values in
$\SS^1_\C\subset\C$, hence the real quotient factor must be $1$.
Taking the limit $\eps\to0$ in the above equation and using that
$\tau_{1/\sqrt{q}}$
is continuous by Corollary~\ref{cor:ghjg67t87gh}
and $\alpha_q$ is continuous by assumption,
we deduce
$
   \phi(0)
   ={\alpha_q}^{\frac12}\circ \tau_{1/\sqrt{q}}(\tau_*)
$.
Hence at collision times $\phi$ is continuous since
$\alpha$ is.
This proves Lemma~\ref{le:z_q-C2}.
\end{proof}

\begin{lemma}\label{cor:ghjg67t77}
Let $q\in \Lambda^\times_{\mathrm{coll}}\Qfrak$.
Then Lemma~\ref{le:t_z-coll}
applies to $z_q$ and
\begin{equation}\label{eq:rescaling-sqrt-q-coll}
   t_{z_q}
   \stackrel{{\color{gray}1}}{=}
   \tau_{z_q}^{-1}{\color{gray}\colon\SS^1_\tau\to\SS^1_t}
   ,\quad
   \tau_{1/\sqrt{q}}
   \stackrel{{\color{gray}2}}{=} 
   t_{z_q}{\color{gray}\colon\SS^1_\tau\to\SS^1_t}
   ,\quad
   \tau_{z}
   \stackrel{{\color{gray}3}}{=}
   t_{1/\sqrt{q_z}}{\color{gray}\, \colon\SS^1_t\to\SS^1_\tau}
   .
\end{equation}
\end{lemma}

\begin{proof}
Identity~1 holds by Lemma~\ref{le:t_z-coll}.
Identity~2 holds by the computation~(\ref{eq:rescaling-sqrt-q}).
Identity~3 holds by the computation~(\ref{eq:tau_z=t1/sqrtq})
with~(\ref{eq:1/q^2-C}) replaced by~(\ref{eq:hgghh5677}) which applies
since $z\in\Crit\,\Bb\subset \bar\Ll^\times_{\mathrm{coll}}\Zfrak$
by Lemma~\ref{le:crit->coll}.
\end{proof}

\boldmath
%%%%%%%%%%%%%%%%%%%%%%%%%%%%%%%%%%%
%%%%%%% Subsubsection  %%%%%%%%%%%%%%%%
%%%%%%%%%%%%%%%%%%%%%%%%%%%%%%%%%%%
\subsubsection*{Image of $\Zz$ lies in $\Crit\Bb$}
\unboldmath

\begin{lemma}[energy identity]\label{le:energy-id}
Any classical collisional $q\in \mathrm{Coll}\,\Ss$ satisfies
$$
   e_q(t)
   =\tfrac12\norm{\dot q}^2-\norm{\tfrac{1}{q}}_{L^1}
   -
   \int_0^1s \inner{\AAAdot_{s}|_{q_s}}{\dot q_s}_0\; ds
   +
   \int_{s=t}^1\inner{\AAAdot_{s}|_{q_s}}{\dot q_s}_0\; ds
$$
at every time $t$ and where at non-collision times
$e_q(t)=\tfrac12\abs{\dot q_t}^2-\tfrac{1}{\abs{q_t}}$ 
by~(\ref{eq:class-coll-loops}).
\end{lemma}

\begin{proof}
At a non-collisional time $t$ we differentiate the energy
\begin{equation*}
\begin{split}
   \tfrac{d}{dt} e_q(t)
   =\tfrac{d}{dt}\left(\tfrac12\abs{\dot q_t}^2-\tfrac{1}{\abs{q_t}}\right)
   &=\inner{\dot q_t}{\ddot q_t+\tfrac{q_t}{\abs{q_t}}}_0
\\
   &=\inner{\dot q_t}{-B_t|_{q_t} i\dot q_t-\AAAdot_t|_{q_t}}_0
\\
   &=-\inner{\dot q_t}{\AAAdot_t|_{q_t}}_0
\end{split}
\end{equation*}
where equality two is by the classical ODE~(\ref{eq:q})
and equality three by orthogonality $\dot q_t\perp i\dot q_t$.
By definition of the $L^2$ and $L^1$ norms we get equality one
\begin{equation*}
\begin{split}
   \tfrac12\norm{\dot q}^2-\norm{\tfrac{1}{q}}_{L^1}
   &=\int_0^1\left(\tfrac12\abs{\dot
       q_s}^2-\tfrac{1}{\abs{q_s}}\right)ds
   =\int_0^1 e_q(s)\, ds
\\
%2
   &\stackrel{{\color{gray}2}}{=}
   \int_{t-1}^t 1\cdot e_q(s)\, ds
\\
%3
   &\stackrel{{\color{gray}3}}{=}
   \left. s\cdot e_q(s)\right|_{t-1}^t
   +\int_{t-1}^t s\,\inner{\dot q_s}{\AAAdot_s|_{q_s}}_0\, ds
\\
%4
   &\stackrel{{\color{gray}4}}{=}
   e_q(t)
   +\left(\int_{t-1}^0+\int_0^1-\int_t^1\right)
   s\,\inner{\dot q_s}{\AAAdot_s|_{q_s}}_0\, ds
\\
%5
   &\stackrel{{\color{gray}5}}{=}
   e_q(t)
   +\int_{t}^{1}({\color{gray}s}-1)
      \inner{\dot q_{s-1}}{\AAAdot_{s-1}|_{q_{s-1}}}_0\,ds
   +\int_0^1 s\,\inner{\dot q_{s}}{\AAAdot_{s}|_{q_{s}}}_0\,ds
   \\
   &\quad
{\color{gray}
   -\int_t^1 s\,\inner{\dot q_{s}}{\AAAdot_{s}|_{q_{s}}}_0\,ds
}
   .
\end{split}
\end{equation*}
Equality~2 is by periodicity of $e_q$.\footnote{
  schematically $\int_0^1 e_s=\int_0^t e_s+\int_t^1 e_s
  =\int_0^t e_s+\int_{t-1}^0 e_{s+1}=\int_{t-1}^t e_s$
  as $e_{s+1}=e_s$ is periodic
  }
Equality~3 is by integration by parts
and the previous displayed identity for the energy derivative.
Equality~4 uses again periodicity $e_q(t-1)=e_q(t)$
and additivity of the integral.
Equality~5 shifts the integration interval by $1$, so $s$ becomes $s-1$.
Then we use periodicity $q_{s-1}=q_s$ and, by~(\ref{eq:AAA}), also
$\AAAdot_{s-1}=\AAAdot_s$.
Then the two {\color{gray}grayed out} summands cancel
and we get equality~6
which proves Lemma~\ref{le:energy-id}.
\end{proof}

\begin{proposition}\label{prop:Coll=>Crit}
Let $q\in \mathrm{Coll}\,\Ss$.
Then $z_q:=q^\frac12\circ \tau_{1/\sqrt{q}} \in \Crit\,\Bb$.
\end{proposition}

\begin{proof}
Pick $q\in \mathrm{Coll}\,\Ss$ and abbreviate $z=z_q:=
q^\frac12\circ \tau_{1/\sqrt{q}}\colon\SS^1\to\Zfrak$
which is $C^2$ by Lemma~\ref{le:z_q-C2}.
We represented $z^{\prime\prime}$ at time $\tau$
via equality~(\ref{eq:z-prime-prime=0})$_6$
in terms of $q$ and $\dot q$ at time $\tau_{1/\sqrt{q}}(\tau)$.
While $q_{\tau_{1/\sqrt{q}}(\tau)}=z^2_\tau$, in order to express
$\dot q$ in terms of $z$ and $z^\prime$ observe first that
\begin{equation*}%\label{eq:tau-prime}
   \abs{q_{\tau_{1/\sqrt{q}}(\tau)}}=\bar z_\tau z_\tau
   ,\qquad
   \tau_{1/\sqrt{q}}^\prime(\tau)
   %1
   \stackrel{\text{(\ref{eq:bhjbhjbhj7458})}}{=}
   \abs{q_{\tau_{1/\sqrt{q}}(\tau)}}\cdot\norm{\tfrac{1}{q}}_{L^1}
   %2
   =\frac{z_\tau \bar z_\tau}{\norm{z}^2}
   .
\end{equation*}
Use identity two and the definition of $z=z_q$ to get equality 2 of what
follows
$$
   z_\tau^{\prime}
   \stackrel{\text{(\ref{eq:joker-1})}}{=}
   \frac{\dot q_{\tau_{1/\sqrt{q}}(\tau)}\cdot\tau^\prime_{1/\sqrt{q}}(\tau)}
   {2q^{1/2}_{\tau_{1/\sqrt{q}}(\tau)}}
   \stackrel{{\color{gray}2}}{=}
   \frac{\dot q_{\tau_{1/\sqrt{q}}(\tau)}}
   {2z_\tau}
   \cdot
   \frac{z_\tau \bar z_\tau}{\norm{z}^2}
   =
   \frac{\dot q_{\tau_{1/\sqrt{q}}(\tau)}\bar z_\tau}
   {2\norm{z}^2}
   .
$$
Equivalently
\begin{equation}\label{eq:dot-q}
   \dot q_{\tau_{1/\sqrt{q}}(\tau)}
   =2\norm{z}^2 \frac{z_\tau^{\prime}}{\bar z_\tau}
   .\qquad
   \text{Thus $\tfrac12 \abs{\dot q_{\tau_{1/\sqrt{q}}(\tau)}}^2
   =2\norm{z}^4 
{\color{cyan}
   \frac{\abs{z_\tau^\prime}^2}{\abs{z_\tau}^2}
}$
   }
   .
\end{equation}
Use the previous identity at time $t:=\tau_{1/\sqrt{q}}(\tau)$, in
which case $\tau=(\tau_{1/\sqrt{q}})^{-1}(t)=(t_z)^{-1}(t)=\tau_z(t)$
by~(\ref{eq:rescaling-sqrt-q-coll}),
in equality 2 of what follows
\begin{equation}\label{eq:gvhg645}
\begin{split}
   \tfrac12\norm{\dot q}^2:
   &=\int_0^1\tfrac12\abs{\dot q_t}^2\, dt
   \stackrel{{\color{gray}2}}{=}
   2\norm{z}^4\int_0^1
{\color{cyan}
   \frac{\abs{z^\prime_{\tau_{z}(t)}}^2}{\abs{z_{\tau_{z}(t)}}^2}
}
   \, dt
   \stackrel{{\color{gray}3}}{=}
   2\norm{z}^4\int_0^1
{\color{cyan}
   \frac{\abs{z^\prime_\tau}^2}{\abs{z_\tau}^2}
}
   \frac{\abs{z_\tau}^2}{\norm{z}^2}\, d\tau
\\
   &\stackrel{{\color{gray}4}}{=}
   2\norm{z}^2\norm{z^\prime}^2
   .
\end{split}
\end{equation}
Equality~3 is by variable substitution
$t=t_z(\tau)$, to $dt=t_z^\prime(\tau)\, d\tau$,
apply~(\ref{eq:t_z-prime-coll}).

\noindent
By~(\ref{eq:z-prime-prime=0}) we obtain equality one
\begin{equation*}
\begin{split}
   z_\tau^{\prime\prime}
%1
   &\stackrel{{\color{gray}1}}{=}
   \tfrac12 \norm{\tfrac{1}{q}}_{L^1}^2
   \Bigl(
   -
   B|_q
   (j_0 \dot q)
   \bar q {\color{cyan}q^{\tfrac12}}
   -
   \AAAdot|_q
   \bar q  {\color{cyan}q^{\tfrac12}}
   +
   e_q   {\color{cyan}q^{\tfrac12}}
  \Bigr)_{\tau_{1/\sqrt{q}}(\tau)}
\\
%2
   &\stackrel{{\color{gray}2}}{=}
   \frac{{\color{cyan}z_\tau}}{2\norm{z}^4}
   \biggl(
   -
   \rot\,\AAA|_q\;
   (j_0 \dot q)
   \bar q
   -
{\color{brown}
   \AAAdot|_q
   \bar q
}
   +
{\color{red}
   \tfrac12\norm{\dot q}^2
}
   -
   \norm{\tfrac{1}{q}}_{L^1}
   \\
   &\qquad \quad 
   -
   \int_0^1 {\color{red}s}\inner{\AAAdot_{s}|_{q_s}}{\dot q_s}_0\; ds
   +
   \int_{s={\color{magenta}\tau_{1/\sqrt{q}}(\tau)}}^1
   \inner{\AAAdot_{s}|_{q_s}}{\dot q_s}_0\; ds
   \biggr)_{\tau_{1/\sqrt{q}}(\tau)\stackrel{\text{(\ref{eq:rescaling-sqrt-q-coll})}}{=}t_z(\tau)}
\\
%3
   &\stackrel{{\color{gray}3}}{=}
   \frac{-{\color{cyan}z_\tau}\rot\,\aaa_{t_z(\tau)}|_{z_\tau}}
      {8\bar z_\tau z_\tau\norm{z}^4}
   \left(
   2\norm{z}^2\frac{i z_\tau^\prime}{\bar z_\tau}
   \right)
   \bar z_\tau^2
   -
{\color{brown}
   \frac{{\color{cyan}z_\tau}\bar z_\tau\bar z_\tau}{2\norm{z}^4}
   \frac{\dot\aaa_{t_z(\tau)}|_{z_\tau}}{2\bar z_\tau}
}
   +
{\color{red}
   {\color{cyan}z_\tau}\frac{\norm{z^\prime}^2}{\norm{z}^2}
}
   -
   \frac{{\color{cyan}z_\tau}}{2\norm{z}^6}
   \\
   &\quad
   \qquad \;
   -\frac{{\color{cyan}z_\tau}}{2\norm{z}^4}
   \int_0^1 
   \overbrace{
   \frac{\int_0^\sigma\abs{z_\rho}^2\,d\rho}{\norm{z}^2}
   }^{=t_z(\sigma)={\color{red}s}}
   \Re\left(
   \overline{\frac{1}{2\bar z_\sigma}\dot \aaa_{t_z(\sigma)}|_{z_\sigma}}\cdot
   2\norm{z}^2\frac{z_\sigma^\prime}{z_\sigma}
   \right)
   \overbrace{\frac{\abs{z_\sigma}^2}{\norm{z}^2}}^{t_z^\prime(\sigma)}
   \, d\sigma
   \\
   &\quad
   \qquad \;
   +\frac{{\color{cyan}z_\tau}}{2\norm{z}^4}
   \int_{\sigma=t_z^{-1}\circ{\color{magenta}\tau_{1/\sqrt{q}}(\tau)}=\tau}^1
   \Re\left(
   \overline{\frac{1}{2\bar z_\sigma}\dot \aaa_{t_z(\sigma)}|_{z_\sigma}}\cdot
   2\norm{z}^2\frac{z_\sigma^\prime}{z_\sigma}
   \right)
   \frac{\abs{z_\sigma}^2}{\norm{z}^2}\, d\sigma
\\
\end{split}
\end{equation*}
\begin{equation*}
\begin{split}
%4
   &\stackrel{{\color{gray}4}}{=}
   -
   \frac{\rot\,\aaa_{t_z(\tau)}|_{z_\tau}}{4\norm{z}^2}\;
   i z_\tau^\prime
   -
{\color{brown}
   \frac{\abs{z_\tau}^2}{4\norm{z}^4}\cdot
   \dot\aaa_{t_z(\tau)}|_{z_\tau}
}
   +
{\color{red}
   z_\tau\frac{\norm{z^\prime}^2}{\norm{z}^2}
}
   -
   \frac{z_\tau}{2\norm{z}^6}
   \\
   &\quad
   \qquad \;
   -
   \frac{z_\tau}{2\norm{z}^6}
   \int_0^1
{\textstyle
   \int_0^\sigma\Abs{z_\rho}^2 d\rho
}
   \cdot
   \inner{\dot{\aaa}_{t_z(\sigma)}|_{z_\sigma}}{z^\prime_\sigma}_0\; d\sigma
   \\
   &\quad
   \qquad \;
   +\frac{z_\tau}{2\norm{z}^4}
   \int_{\tau}^1
   \inner{\dot{\aaa}_{t_z(\sigma)}|_{z_\sigma}}{z^\prime_\sigma}_0\;
   d\sigma
   .
\end{split}
\end{equation*}
But this is the regularized critical point DDE~(\ref{eq:z}) for $z$.
To conclude the proof of Proposition~\ref{prop:Coll=>Crit}
it remains to explain equalities 2, 3, and 4.

\smallskip\noindent
Equality~2 uses~(\ref{eq:1/q-L1}), that $B=\rot\,\AAA$ by
definition~(\ref{eq:q}),
that ${\color{cyan}q^{\tfrac12}}$ rescaled is~${\color{cyan}z}$ at $\tau$,
then we substitute the energy $e_q$ according to
Lemma~\ref{le:energy-id}.
\\
Equality~3 uses that $\tau_{1/\sqrt{q}}(\tau)=t_z(\tau)$
according to~(\ref{eq:rescaling-sqrt-q-coll}). 
Summand one:
we replace $\rot\,\AAA$ according to~(\ref{eq:tilde-A-A})
and $\dot q$ according to~(\ref{eq:dot-q})$_1$, for $j_0$ we write $i$,
and $\bar q$ rescaled is $\bar z_\tau^2$.
Summand two: replace $\AAAdot$ according to the dotted version
of~(\ref{eq:barzA}) and $\bar q$ rescaled by $\bar z^2$ at $\tau$.
Summand three: use~(\ref{eq:gvhg645}).
Summand four: use~(\ref{eq:1/q-L1}).
Summands five and six: variable substitution $s=t_z(\sigma)$,
replace $\AAAdot$ using the dotted version
of~(\ref{eq:barzA}), rewrite the inner product in the $\R^2$ picture
in terms of the real part~(\ref{eq:bar-z-zeta}) in the $\C$ picture,
replace $\dot q_s$ for $s=t_z(\sigma)=\tau_{1/\sqrt{q}}(\sigma)$ 
according to~(\ref{eq:dot-q})$_1$.
\\
Now a lot of factors annulate in pairs and (in equality 4) we go back
to the $\R^2$ picture and inner product.
This proves the critical point equation~(\ref{eq:z}) for $z$., 

Since $z\in C^2$, bootstrapping~(\ref{eq:z})
we see that $z$ is smooth.
This proves Proposition~\ref{prop:Coll=>Crit}.
\end{proof}

\boldmath
%%%%%%%%%%%%%%%%%%%%%%%%%%%%%%%%%%%
%%%%%%% Subsubsection  %%%%%%%%%%%%%%%%
%%%%%%%%%%%%%%%%%%%%%%%%%%%%%%%%%%%
\subsubsection{Proof of Main Theorem~\ref{thm:bijection} (bijection)}
\label{sec:main-bij}
\unboldmath

Pick $z\in \Crit\,\Bb$ and $q\in \mathrm{Coll}\,\Ss$.
Then with the definitions $q_z:=\Qq(z):=z^2\circ\tau_z$
and $z_q:=\Zz(q):=q^{\frac12}\circ\tau_{1/\sqrt{q}}$
we obtain
$$
   \Qq(\Zz(q))
%1
   =z_q^2\circ\tau_{z_q}
%2
   =q\circ \tau_{1/\sqrt{q}}\circ\tau_{z_q}
$$
and
$$
   \Zz(\Qq(z))
%1
   =q_z^{\frac12}\circ\tau_{1/\sqrt{q_z}}
%2
   =z\circ \tau_z\circ \tau_{1/\sqrt{q_z}}
   .
$$
This yields the following two equivalences
\begin{equation*}
\begin{split}
   \Qq\circ\Zz=\id
   \qquad&\Leftrightarrow\qquad
   \id
   =
   \tau_{1/\sqrt{q}}\circ\tau_{z_q}
   \stackrel{\text{(\ref{eq:rescaling-sqrt-q-coll})}_2}{=}
   t_{z_q}\circ\tau_{z_q}
(
   \stackrel{3}{=}
   \id
)
   ,
\\
   \Zz\circ\Qq=\id
   \qquad&\Leftrightarrow\qquad
   \id
   =
   \tau_z\circ \tau_{1/\sqrt{q_z}}
   \stackrel{\text{(\ref{eq:rescaling-sqrt-q-coll})}_3}{=}
   t_{1/\sqrt{q_z}}\circ \tau_{1/\sqrt{q_z}}
(
   \stackrel{3}{=}
   \id
)
   .
\end{split}
\end{equation*}
The right hand side of the first equivalence holds true (eq.~3) by
(\ref{eq:rescaling-sqrt-q-coll})$_1$.
The right hand side of the second equivalence holds true (eq.~3) by
Corollary~\ref{cor:ghjg67t87gh} for $q_z$ which indeed lies in
$\Lambda^\times_{\mathrm{coll}}\Qfrak$ by
Proposition~\ref{prop:extend-Qq}.
This proves Theorem~\ref{thm:bijection}.

\boldmath
%%%%%%%%%%%%%%%%%%%%%%%%%%%%%%%%%%%
%%%%%%%%%%%%%%%%%%%%%%%%%%%%%%%%%%%
%%%%%%% Section  %%%%%%%%%%%%%%%%%%%%%%
%%%%%%%%%%%%%%%%%%%%%%%%%%%%%%%%%%%
%%%%%%%%%%%%%%%%%%%%%%%%%%%%%%%%%%%
\section[Non-local Hamiltonian mechanics -- periodic case]{Non-local Hamiltonian mechanics -- periodic}
\label{sec:non-loc-Ham}
\unboldmath

Throughout Section~\ref{sec:non-loc-Ham}
we specialize to \underline{periodic} $1$-forms
$$
   \theta=\{\theta_t\}_{t\in\SS^1}
   ,\quad
   \theta_{t+1}=\theta_t
   ,\quad
   f=0
   .
$$
The reason is that in this case there exists a non-local magnetic
potential, more precisely a magnetic $1$-form $\Theta$ on loop space,
which induces the magnetic function $\Mm$, see~(\ref{def:Mm}),
and whose exterior derivative, since it is automatically closed,
induces a twisted symplectic form $\Omega_\Theta$ on the cotangent
bundle of loop space.

\boldmath
%%%%%%%%%%%%%%%%%%%%%%%%%%%%%%%%%%%
%%%%%%% Subsection:  %%%%%%%%%%%%%%%%%%%
%%%%%%%%%%%%%%%%%%%%%%%%%%%%%%%%%%%
\subsection{Cotangent bundle of twisted loop space}
\label{sec:cotangent-bdls}
\unboldmath

Considering twisted loop space allows for defining square roots, see
Remark~\ref{rem:cx-square-root}.
The open subsets $\Zfrak^\times\subset\C\setminus\{0\}$
and $\Zfrak=\Zfrak^\times\cup\{0\}$
are defined by~(\ref{eq:Zfrak}).

\begin{definition}[cotangent bundle of regularization loop space ${\color{red}\bar\Ll^\times\Zfrak}$]\label{def:T*Ll}
\mbox{}
\\
Both cotangent bundles
\begin{equation*}%\label{eq:twisted-loop-space-*}
   T^*\Ll_\pm\Zfrak
   :=\{(z,\eta)\in C^\infty (\R,\Zfrak\times\C)\mid\forall \tau\in\R\colon
   \left(z_{\tau+1},\eta_{\tau+1}\right)=\pm \left(z_{\tau},\eta_{\tau}\right)\}
   ,
\end{equation*}
are invariant under the $I$-induced, see~(\ref{eq:inv-I}),
\textbf{sign involution} 
$$
   T^*I=-\Id_2
   \colon T^*\Ll_\pm\Zfrak\to T^*\Ll_\pm\Zfrak
   ,\quad
   (z,\eta)\mapsto -(z,\eta)
   ,\qquad
{\color{gray}
   \Id_k:=\Id_{\C^k}
   ,
}
$$
because the target $\Zfrak\times\C$ is, see~(\ref{eq:I-Zfrak}).
The intersection of both cotangent bundles consists of the zero
element $(0,0)$, the unique fixed point of $T^*I$, in symbols
$
   T^*\Ll_+\Zfrak\cap T^*\Ll_-\Zfrak=\{(0,0)\}=\Fix\, T^*I
$.
Taking away the fiber over the zero loop we get the two cotangent
bundles
$$
   T^*\Ll_\pm^\times\Zfrak
   :=\{(z,\eta)\in T^*\Ll_\pm\Zfrak\colon \norm{z}\not=0\}
   ,\qquad
{\color{gray}
   T^*\Ll_+^\times\Zfrak\cap T^*\Ll_-^\times\Zfrak=\emptyset
   .
}
$$
To put it differently
$$
   T^*\Ll_+^\times\Zfrak
   =\Ll_+^\times\Zfrak \times \Ll_+\C
   ,\qquad
   T^*\Ll_-^\times\Zfrak
   =\Ll_-^\times\Zfrak \times \Ll_-\C
   .
$$
These are disjoint and the sign involution $T^*I$ acts freely
on each one of them, so on their union.
The \textbf{base point projection}, defined by
$$
   \pi\colon T^*\Ll_\pm^\times\Zfrak\to \Ll_\pm^\times\Zfrak
   ,\quad
   (z,\eta)\mapsto z
   ,\qquad
   \pi\circ T^*I=I\circ \pi
   ,
$$
is equivariant with respect to the sign involutions $T^*I$ and
$I$ in~(\ref{eq:inv-I}).
We introduce the two quotient spaces (indicated by a bar)
\begin{equation}\label{eq:T*-bar-Ll}
   T^*\bar\Ll_\pm^\times\Zfrak
   :=\frac{\Ll_\pm^\times\Zfrak \times \Ll_\pm\C}{T^*I}
   ,\qquad
   T^*{\color{red}\bar\Ll^\times\Zfrak}
   :=T^*\bar\Ll_+^\times\Zfrak\cup
   T^*\bar\Ll_-^\times\Zfrak
   .
\end{equation}
For $\bar\Ll^\times\Zfrak$
see~(\ref{eq:involution-quotients}).
The elements of quotient space $T^*\bar\Ll^\times\Zfrak$
are still denoted by~$(z,\eta)$
keeping in mind that each element has two representatives
$\pm (z,\eta)$ and our constructions must be independent of choosing
$(z,\eta)$ or $-(z,\eta)$.
The base point projection descends to quotient spaces,
still denoted by
\begin{equation}\label{eq:proj-involution-quotients-*}
   \pi\colon T^*\bar\Ll^\times\Zfrak
   \to \bar\Ll^\times\Zfrak
   ,\quad
   (z,\eta)\mapsto z
   .
\end{equation}
This concludes Definition~\ref{def:T*Ll}.
\end{definition}

\begin{remark}[tangent spaces]
At "Upsilon" $\Upsilon=(z,\eta)$ the tangent spaces
$$
   T_\Upsilon T^*\Ll_\pm^\times\Zfrak
   =\Ll_\pm\C^2
   :=\{\hat\Upsilon=(\hat z,\hat \eta)\in C^\infty (\R,\C^2)\mid
   \forall \tau\in\R\colon\hat\Upsilon_{\tau+1}=\pm \hat\Upsilon_\tau\}
$$
are disjoint and invariant under sign involution
$dT^*I|_\Upsilon=-\Id$ which acts freely on each of them.
We introduce the two quotient spaces (indicated by a bar)
\begin{equation*}%\label{eq:bar-Ll-C-*}
   T_\Upsilon T^*\bar\Ll_\pm^\times\Zfrak
   :=\bar\Ll_\pm\C^2
   =\frac{\Ll_\pm\C^2}{-\Id}
   ,\qquad
   T_\Upsilon T^*{\color{red}\bar\Ll^\times\Zfrak}
   :=T_\Upsilon T^*\bar\Ll_+^\times\Zfrak
   \cup T_\Upsilon T^*\bar\Ll_-^\times\Zfrak
%   =\bar\Ll_\pm\C^2
   .
\end{equation*}
The elements of quotient space
$T_\Upsilon T^*\bar\Ll^\times\Zfrak
=\bar\Ll_+\C^2\cup \bar\Ll_-\C^2$
are still denoted by~$\hat\Upsilon=(\hat z,\hat \eta)$
keeping in mind that each element has two representatives
$\pm \hat\Upsilon$ and our constructions must be independent of
choosing $\hat\Upsilon$ or $-\hat\Upsilon$.

The linearization of the projection $\pi(z,\eta)=z$
in~(\ref{eq:proj-involution-quotients-*}) acts by
\begin{equation}\label{eq:hghjghy77897-2}
   d\pi|_{(z,\eta)}
   \colon T_{(z,\eta)} T^*\bar\Ll^\times\Zfrak
   \to
   T_z \bar\Ll^\times\Zfrak
   ,\qquad
   (\hat z,\hat \eta)\mapsto \hat z
   .
\end{equation}
\end{remark}

\begin{remark}[sign involution on tangent bundle]
The tangent map of $T^*I=-\Id_{2}\colon
T^*\bar\Ll^\times\Zfrak\to
T^*\bar\Ll^\times\Zfrak$ 
acts on the tangent bundle
$T T^*\bar\Ll^\times\Zfrak$
as $TT^*I=-\Id_{4}$.
\end{remark}

\boldmath
%%%%%%%%%%%%%%%%%%%%%%%%%%%%%%%%%%%
%%%%%%% Subsubsection  %%%%%%%%%%%%%%%%
%%%%%%%%%%%%%%%%%%%%%%%%%%%%%%%%%%%
\subsubsection{Canonical vector field and induced flow}
\unboldmath

The \textbf{canonical vector field} is generated by time derivative
\begin{equation}\label{eq:can-vf-loops-T*}
   \Vv\colon T^*\bar\Ll^\times\Zfrak
   \to TT^*\bar\Ll^\times\Zfrak
   ,\quad
   \Upsilon
   \mapsto(\Upsilon,\Upsilon^\prime)
   .
\end{equation}
For the principal part we use the same notation
$\Vv(\Upsilon)=\Upsilon^\prime$.
Note that $\Vv$ is well defined on quotients, indeed
$\Vv(-\Upsilon)=(- \Upsilon,(-\Upsilon)^\prime)
=-(\Upsilon, \Upsilon^\prime)=-\Vv(\Upsilon)$.

\begin{remark}[induced flows and Lie derivative]
\label{rem:Vv-flow}
The flow induced by $\Vv=\p_\tau$ on
$T^*\bar\Ll^\times\Zfrak\ni\Upsilon$ is time shift
$\Phi^r_\Vv\Upsilon=r_*\Upsilon:=\Upsilon (\cdot+r)$
whenever $\tau\in\R$.
Given $\hat\Upsilon\in T_\Upsilon T^*\bar\Ll^\times\Zfrak$,
analogous as for the flow $\phi^r\nu$ on
$\bar\Ll^\times\Zfrak$ in~(\ref{eq:shift-nu}),
one shows that $\bigl(d \Phi^r_\Vv|_z \hat\Upsilon\bigr)_\tau
=\hat\Upsilon_{\tau+r}$.
The base point projection~(\ref{eq:proj-involution-quotients-*})
relates the two flows
\begin{equation}\label{eq:P-Phi-comm}
   \pi\circ\Phi^r_\Vv
   =\phi^r_\nu\circ \pi
   .
\end{equation}
Indeed $\pi\circ\Phi^r_\Vv(z_\tau,\eta_\tau)
=\pi(z_{\tau+r},\eta_{\tau+r})=z_{\tau+r}=\phi^r_\nu z_\tau
=\phi^r_\nu\circ \pi(z_\tau,\eta_\tau)$.

\smallskip
Pull-back under base point projection intertwines
the two \textbf{Lie derivatives}
\begin{equation}\label{eq:hghjghy777-2}
\begin{split}
   L_\Vv \pi^*:
   &=\left.\tfrac{d}{dr}\right|_{r=0}(\Phi^r_\Vv )^* \pi^*
\\
   &=\pi^*\left.\tfrac{d}{dr}\right|_{r=0}(\phi^r_\nu)^*
\\
   &=:\pi^* L_\nu
   .
\end{split}
\end{equation}
Here the first and the last equality are by definition of the Lie derivative.
Equality two is by~(\ref{eq:P-Phi-comm}) and *-functoriality
$(\Phi^r_\Vv )^* \pi^*=(\pi\circ\Phi^r_\Vv)^*
=(\phi^r_\nu\circ \pi)^*=\pi^*(\phi^r_\Vv )^*$
and since $\pi^*$ does not depend on $r$.
This concludes Remark~\ref{rem:Vv-flow}.
\end{remark}

\boldmath
%%%%%%%%%%%%%%%%%%%%%%%%%%%%%%%%%%%
%%%%%%% Subsubsection  %%%%%%%%%%%%%%%%
%%%%%%%%%%%%%%%%%%%%%%%%%%%%%%%%%%%
\subsubsection{Canonical $1$-form $\Lambda$  and symplectic form
$\Omega$}
\unboldmath

Let $\lambda=\eta\, dz$ be the canonical $1$-form
on the cotangent bundle of $\Zfrak^\times\subset\C$.
The \textbf{canonical \boldmath$1$-form} on
$T^*\bar\Ll^\times\Zfrak$ 
is defined by assigning to an element
$\hat\Upsilon=(\hat z,\hat \eta)$ of the tangent space to
$T^* \bar\Ll^\times\Zfrak$
at a point $\Upsilon=(z,\eta)$ the real value
\begin{equation*}%\label{eq:Lambda}
   \Lambda_\Upsilon \hat\Upsilon
   =\Lambda_{(z,\eta)} (\hat z,\hat\eta)
   :=\int_0^1\eta_\tau(\hat z_\tau)\, d\tau
   =\INNER{\eta}{\hat z}
   ,\quad
{\color{gray}
   \Lambda
   =\int_0^1\lambda\, d\tau
   ,
}
\end{equation*}
cf Definition~\ref{def:Liouville-form}.
Note that $\Lambda$ is invariant under sign involution. Indeed
$$
   (T^*I)^*\Lambda \;(\Upsilon,\hat\Upsilon)
=\Lambda_{T^*I\,\Upsilon} TT^*I\, \hat\Upsilon
=\Lambda_{-\Upsilon} (-\hat\Upsilon)
=\INNER{-\eta}{-\hat z}
=\INNER{\eta}{\hat z}
=\Lambda(\Upsilon,\hat\Upsilon)
   .
$$
We evaluate $\Lambda$ on the canonical vector field
$\Vv$ along $T^*\bar\Ll^\times\Zfrak$,
see~(\ref{eq:can-vf-loops}).
At any point $\Upsilon=(z,\eta)\in T^*\bar\Ll^\times\Zfrak$,
for later use in~(\ref{eq:Aa_Hh}), we compute
\begin{equation*}%\label{eq:Lambda-Vv}
  (i_\Vv\Lambda)_\Upsilon
  =\Lambda_\Upsilon \Upsilon^\prime =\int_0^1 \eta_\tau(z_\tau^\prime)\, d\tau
  =\INNER{\eta}{z^\prime}
   .
\end{equation*}

\begin{definition}[canonical symplectic form]
At a point $\Upsilon=(z,\eta)\in T^*\bar\Ll^\times\Zfrak$
and for tangent vectors $\hat\Upsilon^j=(\hat z^j,\hat\eta^j)\in
T_{\Upsilon}T^*\bar\Ll^\times\Zfrak$
for $j=1,2$ we define and compute the
\textbf{canonical symplectic form \boldmath$\Omega$}
on $T^*\bar\Ll^\times\Zfrak$ as follows
\begin{equation}\label{eq:Omega_can}
\begin{split}
   \Omega_\Upsilon(\hat\Upsilon^1, \hat\Upsilon^2) :
   &=d\Lambda_\Upsilon(\hat\Upsilon^1, \hat\Upsilon^2)
   \stackrel{{\color{gray}2}}{=}
   \int_0^1 d\lambda_{(z,\eta)}
   (\hat\Upsilon^1_\tau, \hat\Upsilon^2_\tau)\, d\tau
\\
   &=\int_0^1\left(\inner{\hat\eta^1_\tau}{\hat z^2_\tau}_0
   -\inner{\hat\eta^2_\tau}{\hat z^1_\tau}_0\right) d\tau
   =\inner{\hat\eta^1}{\hat z^2}
   -\inner{\hat\eta^2}{\hat z^1}
   .
\end{split}
\end{equation}
Here equality~2 uses~(\ref{eq:exterior-derivative}),
and $\inner{\cdot}{\cdot}$ is the $L^2$-inner product,
$d\lambda=d\eta\wedge dz$,
and $\inner{\cdot}{\cdot}_0$ is the Euclidean inner
product~(\ref{eq:bar-z-zeta}) on $\C\simeq\R^2$.
Similarly as for $\Lambda$ one checks that
$\Omegacan$ is invariant under sign involution, namely
$(T^*I)^*\Omegacan=\Omegacan$.
\end{definition}

\boldmath
%%%%%%%%%%%%%%%%%%%%%%%%%%%%%%%%%%%
%%%%%%% Subsubsection  %%%%%%%%%%%%%%%%
%%%%%%%%%%%%%%%%%%%%%%%%%%%%%%%%%%%
\subsubsection{Exact twisted symplectic form $\Omega_\Theta$}
\unboldmath

\begin{definition}[non-local magnetic $1$-form]\label{def:Theta-periodic}
Let $\vartheta=\varsigma^*\theta$ be a periodic $1$-form on
$\Zfrak$ as in (\ref{eq:Zfrak}).
Define the reparametrization $t_z\colon\SS^1_\tau\to\SS^1_t$
by formula~(\ref{eq:t_z-prime-coll}).
We get a $1$-form on loop space
$\bar\Ll^\times\Zfrak$, cf~(\ref{eq:Theta}),
defined for $(z,\xi)\in T \bar\Ll^\times\Zfrak$ by
\begin{equation}\label{eq:Theta-periodic}
   \Theta_z\xi=\Theta_z^{\vartheta,T}\xi
   :=\int_{\SS^1} \vartheta_{t_z}|_z \xi
   :=\int_0^1 \vartheta_{t_z(\tau)}|_{z_\tau} \xi_\tau\, d\tau
   .
\end{equation}
\end{definition}

Pull-back $\Theta$ by the base point projection $\pi$
in~(\ref{eq:proj-involution-quotients-*}),
then insert the canonical vector field $\Vv$
from~(\ref{eq:can-vf-loops-T*}),
see also $\nu$ in~(\ref{eq:can-vf-loops}),
to recover the magnetic func\-tional $\Mm$ in~(\ref{eq:Mm-BOV}):
At $\Upsilon=(z,\eta)\in T^*\bar\Ll^\times\Zfrak$,
for use in~(\ref{eq:Aa_Hh}), 
we compute
\begin{equation*}\label{eq:bhjvbhj7745}
   (i_\Vv \pi^*\Theta)_\Upsilon
   =\Theta_{\pi(\Upsilon)} d\pi|_\Upsilon\Upsilon^\prime
   \stackrel{\text{(\ref{eq:hghjghy77897-2})}}{=} 
   \Theta_z z^\prime
   =i_\nu\Theta_z
{\color{gray}\;
   \stackrel{f=0}{=}
   \Mm(z)
   .
}
\end{equation*}

\begin{lemma}[twisted non-local symplectic form]
\label{le:tw-sympl-form}
For $d\Theta$ in~(\ref{eq:dTheta})
the
$2$-forms
$$
   \Omega_\Theta
   :=\Omega+\pi^*d\Theta
   =d\left(\Lambda+\pi^*\Theta\right)
   ,\qquad
   \Omega:=d\Lambda
$$
are weak\footnote{
  see Remark~\ref{rem:reak-symplectic}
  }
symplectic forms on $T^*\bar\Ll^\times\Zfrak$
where $d$ is defined by~(\ref{eq:exterior-derivative}).
\end{lemma}

\begin{proof}
Since $dd=0$, cf. Remark~\ref{rem:dd=0},
Corollary~\ref{cor:omega_sigma} applies.
\end{proof}

\boldmath
%%%%%%%%%%%%%%%%%%%%%%%%%%%%%%%%%%%
%%%%%%%%%%%%%%%%%%%%%%%%%%%%%%%%%%%
%%%%%%% Subsection  %%%%%%%%%%%%%%%%%%%
%%%%%%%%%%%%%%%%%%%%%%%%%%%%%%%%%%%
\subsection[Hamiltonian vector field $\Xx$ and Lorentz force $\Zz$]
{Hamiltonian vector field $\Xx_\Hh^{\Omega_\Theta}$ and
Lorentz force $\Zz_\Theta$}
\unboldmath

\begin{definition}
On the cotangent bundle $T^*\bar\Ll^\times\Zfrak$,
see~(\ref{eq:T*-bar-Ll}),
the \textbf{mechanic} (kinetic plus potential energy)
\textbf{Hamiltonian} is defined by
\begin{equation*}%\label{eq:Hh}
\begin{split}
   \Hh\colon
   T^*\bar\Ll^\times\Zfrak
   &\to\R
   \qquad\qquad\qquad\quad
{\color{gray}
   \inner{\cdot}{\cdot}^z:=\tfrac{1}{4 \norm{z}^2} \inner{\cdot}{\cdot}
}
\\
   (z,\eta)
   &\mapsto
   \tfrac12 \INNER{\eta}{\eta}^z
   -\tfrac{1}{\norm{z}^2}
   =
{\color{gray}
   \underbrace{{\color{black}\tfrac{\norm{\eta}^2}{8\norm{z}^2}}}_{\Kk^*(z,\eta)}
   \underbrace{{\color{black}-\;\tfrac{1}{\norm{z}^2}}}_{\Uu^*(z)}
}
   .
\end{split}
\end{equation*}
The non-local \textbf{Hamiltonian vector field} is determined and
abbreviated by
$$
   d\Hh
   =\Omega_\Theta(\cdot, \Xx_\Theta)
   =\Omega(\cdot, \Xx_\Theta)
   +d\Theta(d\pi \cdot, \Xx_\Hh^{\Omega_\Theta})
   ,\qquad
   \Xx_\Theta=\Xx_\Hh^{\Omega_\Theta} .
$$
This guarantees uniqueness, but not necessarily existence
since we only have a weak symplectic
form; cf. Remark~\ref{rem:reak-symplectic}.
The next lemma guarantees existence.
\end{definition}

\begin{lemma}[Non-local Hamiltonian vector field]
\label{le:Xx_Theta}
At $(z,\eta)\in T^*\bar\Ll^\times\Zfrak$ we have
\begin{equation}\label{eq:X}
\begin{split}
   \Xx_\Theta(z,\eta)
   &
   =\tfrac{1}{4\norm{z}^2}
   \begin{pmatrix}
      \eta\\
      \tfrac{\norm{\eta}^2-8}{\norm{z}^2}\, z
   \end{pmatrix}
   +
   \tfrac{1}{4\norm{z}^2}
   \begin{pmatrix}
      0\\
      \Zz_z \eta
   \end{pmatrix}
   .
\end{split}
\end{equation}
The \textbf{non-local Lorentz force}
$\Zz_z\in\Ll( T_z\bar\Ll^\times\Zfrak)$
is uniquely determined\footnote{
  Formula~(\ref{eq:dTheta}) 
  shows that $d\Theta$ is a continuous bilinear form on $L^2$,
  hence by the Riesz representation theorem it determines an
  isomorphism.
  }
by
\begin{equation}\label{eq:Zz}
   \INNER{\Zz_z\eta}{\cdot}
   =-d\Theta_z(\eta,\cdot)
   .
\end{equation}
\end{lemma}

\begin{proof}
The Hamiltonian vector field of $\Hh$ for the untwisted symplectic form
$\Omega$ is at a point $\Upsilon=(z,\eta)$ given by
\begin{equation*}
  \left\{
   \begin{aligned}
     \Xx_0^1&=\p_\eta\Hh  &&=\tfrac{1}{4\norm{z}^2}\;\eta\\
     \Xx_0^2&=-\p_z\Hh&&=
        \tfrac{\norm{\eta}^2}{4\norm{z}^4}\, z
        -\tfrac{2}{\norm{z}^4}\, z
   \end{aligned}
   \right.
   ,\quad
   \Xx_0(z,\eta)
   =\tfrac{1}{4\norm{z}^2}\begin{pmatrix}
      \eta\\
      \tfrac{\norm{\eta}^2}{\norm{z}^2}\, z
      -\tfrac{8}{\norm{z}^2}\, z
   \end{pmatrix}
   .
\end{equation*}
Pick $\Yy\in T_\Upsilon T^*\bar\Ll^\times\Zfrak$,
by~(\ref{eq:Omega_can})
and then by definition of $\Xx_0$ and of $\Xx_\Theta$, we get
\begin{equation}\label{eq:hgjhghjbhj}
\begin{split}
   \inner{\Xx_0^2}{\Yy^1}-\inner{\Xx_0^1}{\Yy^2}
   &=i_\Yy i_{\Xx_0}\Omega
\\
   &=-i_\Yy d\Hh
\\
   &= i_\Yy i_{\Xx_\Theta}\Omega_\Theta
\\
   &=\Omega(\Xx_\Theta,\Yy)+d\Theta(\pi_*\Xx_\Theta,\pi_*\Yy)
\\
   &=\inner{\Xx_\Theta^2}{\Yy^1}-\inner{\Xx_\Theta^1}{\Yy^2}
   +d\Theta(\Xx_\Theta^1,\Yy^1) .
\end{split}
\end{equation}
The last but one identity uses $\Omega_\Theta:=\Omega+\pi^*d\Theta$,
the last one~(\ref{eq:Omega_can})
and the linearized projection $\pi_*=d\pi$ in~(\ref{eq:hghjghy77897-2}).
For vectors of the form $\Yy=(0,\Yy^2)$ the displayed formula yields that
$
   -\inner{\Xx_0^1}{\Yy^2}
   =-\inner{\Xx_\Theta^1}{\Yy^2}
$
for every $\Yy^2$. Thus
$\Xx_\Theta^1=\Xx_0^1=\tfrac{\eta}{4\norm{z}^2}$, so in
the displayed formula two terms cancel and \underline{it reduces~to}
$$
\underline{
   \inner{\Xx_\Theta^2-\Xx_0^2}{\Yy^1}
   =-d\Theta(\Xx_\Theta^1,\Yy^1)
}
   =-d\Theta(\tfrac{\eta}{4\norm{z}^2},\Yy^1)
   \stackrel{\text{(\ref{eq:Zz})}}{=}
   \inner{\Zz_z \tfrac{\eta}{4\norm{z}^2}}{\Yy^1}
$$
for every $\Yy^1$.
Thus the difference is a vertical vector field, the Lorentz force
$$
   (\Xx_\Theta-\Xx_0)|_{(z,\eta)}
   =
   \begin{pmatrix}
      \Xx_\Theta^1-\Xx_0^1\\
      \Xx_\Theta^2-\Xx_0^2
   \end{pmatrix}|_{(z,\eta)}
   =\tfrac{1}{4\norm{z}^2}
   \begin{pmatrix}
      0\\
      \Zz_z\eta
   \end{pmatrix}
   .
$$
In~(\ref{eq:dTheta}) we calculate $d\Theta$.
The proof of Lemma~\ref{le:Xx_Theta} is complete.
\end{proof}

\boldmath
%%%%%%%%%%%%%%%%%%%%%%%%%%%%%%%%%%%
%%%%%%% Subsection  %%%%%%%%%%%%%%%%%%%
%%%%%%%%%%%%%%%%%%%%%%%%%%%%%%%%%%%
\subsection[Euler vector field $\Yy$]
{Euler vector field $\Yy_\Vv^{\Omega_\Theta}$}
\unboldmath

Let $\Theta$ be defined by~(\ref{eq:Theta-periodic}).
The base-point projection $\pi$ intertwines the
upstair and the downstair Lie derivative
by~(\ref{eq:hghjghy777-2}).

\begin{definition}\label{def:Euler}
The \textbf{non-local Euler vector field} $\Yy_\Theta$ is determined
by
$$
   i_{\Yy_\Theta}\Omega_\Theta
   =\pi^* L_\nu\Theta 
{\color{gray}
\quad\bigl(
   =
   L_\Vv \pi^*\Theta
\bigr)
}
   .
$$
\end{definition}

\begin{lemma}\label{le:Yy_Theta}
At $(z,\eta)\in T^*\bar\Ll^\times\Zfrak$
the Euler vector field is time-wise given
by
\begin{equation}\label{eq:Yy}
   \Yy_\Theta|_z(\tau)
   =
   \tfrac{1}{4\norm{z}^2}
   \begin{pmatrix}
      0\\
      -4 \Abs{z_0}^2\dot\vartheta_{t_z(\tau)}|_{z_\tau}
   \end{pmatrix}
   .
\end{equation}
\end{lemma}

\begin{proof}[Proof of Lemma~\ref{le:Yy_Theta}]
Pick $\Xx=(\Xx^1,\Xx^2)\in
T_{(z,\eta)} T^*\bar\Ll^\times\Zfrak$, then by
definition of $\Yy_\Theta$ and similar to the last two lines
of~(\ref{eq:hgjhghjbhj}), we obtain
\begin{equation*}
\begin{split}
   i_\Xx L_\Vv \pi^*\Theta
   &=i_\Xx i_{\Yy_\Theta}\Omega_\Theta
\\
   &=\Omega(\Yy_\Theta,\Xx)+d\Theta(\pi_*\Yy_\Theta,\pi_*\Xx)
\\
   &=\inner{\Yy_\Theta^2}{\Xx^1}-\inner{\Yy_\Theta^1}{\Xx^2}
   +d\Theta(\Yy_\Theta^1,\Xx^1) .
\end{split}
\end{equation*}
Now at a point $\Upsilon=(z,\eta)$ the left hand side evaluates to
$$
   \left(i_\Xx \underline{L_\Vv \pi^*}\Theta\right)_\Upsilon
   \stackrel{\text{(\ref{eq:hghjghy777-2})}}{=}
   \left(i_\Xx \underline{\pi^* L_\nu}\Theta\right)_\Upsilon
   \stackrel{\color{gray}3}{=}
   \left(L_\nu \Theta\right)_{\pi(\Upsilon)} d\pi|_\Upsilon \Xx
   \stackrel{\text{(\ref{eq:hghjghy77897-2})}}{=} 
   (L_\nu \Theta)_z\Xx^1
$$
where equality~3 is by definition of pull-back.
So the earlier computation yields
$$
   \inner{\Yy_\Theta^2}{\Xx^1}-\inner{\Yy_\Theta^1}{\Xx^2}
   +d\Theta(\Yy_\Theta^1,\Xx^1)
   =(L_\nu \Theta |_z) \Xx^1
   .
$$
For vectors of the form $\Xx=(0,\Xx^2)$ this
implies $-\inner{\Yy_\Theta^1}{\Xx^2}=0$ for every $\Xx^2$,
thus $\Yy_\Theta^1=0$. Hence
$\inner{\Yy_\Theta^2}{\Xx^1}=(L_\nu \Theta |_z) \Xx^1$, equivalently
$$
   \int_0^1\inner{\Yy_\Theta^2|_{(z,\eta)}(\tau)}
   {\Xx^1|_{(z,\eta)}(\tau)}_0\,d\tau
   \stackrel{\text{\rm(\ref{eq:BOV-Lie})}}{=}
   -
   \int_0^1 \tfrac{\Abs{z_0}^2}{\norm{z}^2}
   \dot\vartheta_{t_z(\tau)}|_{z_\tau}
   \Xx^1|_{(z,\eta)}(\tau)
   \; d\tau
   .
$$
This proves Lemma~\ref{le:Yy_Theta}.
\end{proof}

\boldmath
%%%%%%%%%%%%%%%%%%%%%%%%%%%%%%%%%%%
%%%%%%% Subsection  %%%%%%%%%%%%%%%%%%%
%%%%%%%%%%%%%%%%%%%%%%%%%%%%%%%%%%%
\subsection[Action $\Aa$ and Euler-Hamilton equation]
{Action $\Aa_\Hh^{\Lambda_\Theta}$ and
Euler-Hamilton equation}
\unboldmath

\begin{definition}
Motivated by~\cite[\S 4]{Frauenfelder:2026a}, we define
the \textbf{perturbed symplectic action functional}
on $T^*\bar\Ll^\times\Zfrak$ as
follows\footnote{$\Aa(\Upsilon)
  =\Lambda|_\Upsilon\Vv(\Upsilon)
  +\Theta_{\pi(\Upsilon)} d\pi|_\Upsilon \Vv(\Upsilon)
  -\Hh(\Upsilon)$
  }
$
   \Aa=\Aa_\Hh^{\Lambda_\Theta}
   :=i_\Vv\left(\Lambda+\pi^*\Theta\right)-\Hh
$.
More precisely, the action is of the form
\begin{equation}\label{eq:Aa_Hh}
\begin{split}
   \Aa\colon T^*\bar\Ll^\times\Zfrak
   &\to\R
\\
   \Upsilon=(z,\eta)
   &\mapsto\Lambda_\Upsilon \Upsilon^\prime
   +\Theta_z z^\prime
   -\Hh(z,\eta)
\\
   &\quad=
   \INNER{\eta+\Theta_z}{z^\prime}
   -\tfrac{\norm{\eta}^2}{8\norm{z}^2} 
   +\tfrac{1}{\norm{z}^2}
   =\Aa(z,\eta)
   .
\end{split}
\end{equation}
\end{definition}

We compute $d\Aa$ using Cartan's formula in equation two
\begin{equation*}
\begin{split}
   d\Aa
   &=di_\Vv\left(\Lambda+\pi^*\Theta\right)-d\Hh
\\
   &=L_\Vv\left(\Lambda+\pi^*\Theta\right)
   -i_\Vv d\left(\Lambda+\pi^*\Theta\right)+i_{\Xx_\Theta}\Omega_\Theta
\\
   &=L_\Vv\Lambda+L_\Vv \pi^*\Theta -i_\Vv\Omega_\Theta
   +i_{\Xx_\Theta}\Omega_\Theta
\\
%4
   &=L_\Vv \pi^*\Theta -i_\Vv\Omega_\Theta
   +i_{\Xx_\Theta}\Omega_\Theta
\\
   &= i_{\Yy_\Theta}\Omega_\Theta
   -i_\Vv\Omega_\Theta
   +i_{\Xx_\Theta}\Omega_\Theta
\\
   &= i_{(\Xx_\Theta+\Yy_\Theta-\Vv)}\,\Omega_\Theta .
\end{split}
\end{equation*}
Equation two is by definition of $\Xx_\Theta$.
Equation four uses that $L_\Vv\Lambda=0$ by rotation invariance of
$\Lambda$.
Equation five is Definition~\ref{def:Euler} of the Euler vector field.
This, together with the formulas~(\ref{eq:X}) for $\Xx_\Theta$
and~(\ref{eq:Yy}) for $\Yy_\Theta$, proves the first two displayed
formulas in the next lemma.

\begin{theorem}[critical points of $\Aa$]\label{thm:Crit-Aa_Theta}
The critical points of the action functional $\Aa$ are the solutions
of the \textbf{delay Euler-Hamilton equation}
$$
   \Vv
   =\Xx_\Theta+\Yy_\Theta .
$$
The delay Euler-Hamilton equation for
$(z,\eta)\in T^*\bar\Ll^\times\Zfrak$ is of the form
\begin{equation}\label{eq:EH-eqs-Theta}
\begin{split}
   \begin{pmatrix}
      z^\prime\\\eta^\prime
   \end{pmatrix}
   &=
   \tfrac{1}{4\norm{z}^2}
   \begin{pmatrix}
      \eta\\
      \tfrac{\norm{\eta}^2-8}{\norm{z}^2}\, z 
   \end{pmatrix}
   +
   \tfrac{1}{4\norm{z}^2}
   \begin{pmatrix}
      0\\
      \Zz_z\eta
      -4\Abs{z_0}^2\dot\vartheta_{t_z(\tau)}|_{z_\tau}
   \end{pmatrix}
   .
\end{split}
\end{equation}
Along the image of the canonical injection~(\ref{eq:proj-inj}),
i.e. $\eta=4\norm{z}^2z^\prime$, it holds that
$$
   \tfrac{1}{4\norm{z}^2}
   \left(
   \Zz_z\eta-4\Abs{z_0}^2\dot\vartheta_{t_z(\tau)}|_{z_\tau}
   \right)
   \stackrel{\text{\rm(\ref{eq:Yy})}}{=}
   \Zz_z z^\prime+\Yy_\Theta^2|_z
   \stackrel{\text{\rm(\ref{eq:Lorentz-BOV})}}{=}
   \grad\Mm|_z
   .
$$
Equivalently to~(\ref{eq:EH-eqs-Theta})
the delay Euler-Hamilton equation is the
first order DDE
\begin{equation}\label{eq:EH-eqs}
\begin{split}
   z^\prime
   &=\tfrac{1}{4\norm{z}^2}\;\eta
\\
   \eta^\prime
%1
   &=
   \tfrac{\norm{\eta}^2-8}{4\norm{z}^4} z
   -\tfrac{2z}{\norm{z}^4}
   \int_0^1
{\textstyle
   \int_0^s\Abs{z_\sigma}^2 d\sigma
}
   \cdot
   \inner{\dot{\aaa}_{t_z(s)}|_{z_s}}{z^\prime_s}_0\; ds
   \\
   &\quad
% T4 magnetic M1
   +\tfrac{2 z}{\norm{z}^2}
   \int_{\sigma=\tau}^1
   \inner{\dot{\aaa}_{t_z(\sigma)}|_{z_\sigma}}{z^\prime_\sigma}_0\;
   d\sigma
% T5 magnetic M3
   -\tfrac{\Abs{z}^2}{\norm{z}^2}\dot{\aaa}_{t_z}|_z
% T6 magnetic M4
   -\bigl(\rot\,\aaa_{t_z}|_z\bigr)\; j_0 z^\prime
   .
\end{split}
\end{equation}
\end{theorem}

\begin{proof}
The formula for $\grad\Mm$ in Lemma~\ref{le:gradient-Mm}
proves~(\ref{eq:EH-eqs}), so Theorem~\ref{thm:Crit-Aa_Theta}.
\end{proof}

\boldmath
%%%%%%%%%%%%%%%%%%%%%%%%%%%%%%%%%%%
%%%%%%% Subsection  %%%%%%%%%%%%%%%%%%%
%%%%%%%%%%%%%%%%%%%%%%%%%%%%%%%%%%%
\subsection{Euler-Hamilton and Lagrange solutions correspond}
\label{sec:EH=Lag}
\unboldmath

In this section we prove Theorem~\ref{thm:B}
from the introduction.
We need to show that there is a $1$-to-$1$ correspondence
between the critical point equation~(\ref{eq:EH-eqs-Theta}) of $\Aa$
in Theorem~\ref{thm:Crit-Aa_Theta}
and the critical point equation~(\ref{eq:z}) of $\Bb$ in
Theorem~\ref{thm:crit-Ss_Ll}.

%%%%%%%%%%%%%%%%%%%%%%%%%%%%%%%%%%%
%%%%%%% Subsubsection  %%%%%%%%%%%%%%%%
%%%%%%%%%%%%%%%%%%%%%%%%%%%%%%%%%%%
\subsubsection*{Euler-Hamilton implies Lagrange}

\begin{lemma}\label{le:Ham-Lag}
If $(z,\eta)$ solves the first order Euler-Hamilton
DDE~(\ref{eq:EH-eqs-Theta}), then the base part $z$ solves
the regularized second order Lagrangian DDE~(\ref{eq:z}).
\end{lemma}

\begin{proof}
Suppose $(z,\eta)$ solves~(\ref{eq:EH-eqs-Theta}).
Differentiate component one with respect to time $\tau$ and then, in
equality~2, substitute $\eta^\prime$ according to component two,
use in addition that $\eta=4\norm{z}^2 z^\prime$ by component one
\begin{equation*}
\begin{split}
   z^{\prime\prime}
%1
   &=\frac{1}{4\norm{z}^2} \eta^\prime
%2
   \stackrel{{\color{gray}2}}{=}
   \frac{\norm{z^\prime}^2 z}{\norm{z}^2}
   -\frac{z}{2\norm{z}^6}
   +\frac{\grad\Mm|_z}{4\norm{z}^2}.
\end{split}
\end{equation*}
This proves~(\ref{eq:z}), hence Lemma~\ref{le:Ham-Lag}.
\end{proof}

%%%%%%%%%%%%%%%%%%%%%%%%%%%%%%%%%%%
%%%%%%% Subsubsection  %%%%%%%%%%%%%%%%
%%%%%%%%%%%%%%%%%%%%%%%%%%%%%%%%%%%
\subsubsection*{Lagrange implies Euler-Hamilton}

Here we need to inject appropriately a Lagrange solution $z$
to the Hamiltonian side which requires a pair. 
We define the injection by
$$
   \iota(z):=(z,\eta_z)
   ,\quad
   \eta_z:=4\norm{z}^2 z^\prime
   .
$$

\begin{lemma}\label{le:Lag-Ham}
If $z$ solves the regularized second order Lagrangian DDE~(\ref{eq:z}),
then $(z,\eta_z)$ solves the first order Euler-Hamilton
DDE~(\ref{eq:EH-eqs-Theta}).
\end{lemma}

\begin{proof}
Suppose $z$ solves~(\ref{eq:z}).
First component of~(\ref{eq:EH-eqs-Theta}): Indeed
$z^\prime=\frac{1}{4\norm{z}^2} \eta_z$ by definition of $\eta_z$.
Second component:
Differentiate $\eta_z$ and then, in equality~2, replace $z^{\prime\prime}$
according to~(\ref{eq:z}), using in addition that
$z^\prime=\frac{1}{4\norm{z}^2} \eta_z$ to get
\begin{equation*}
\begin{split}
   \eta_z^\prime
%1
   &=4\norm{z}^2 z^{\prime\prime}
%2
   \stackrel{{\color{gray}2}}{=}
   4\norm{z}^2
   \left(
   \frac{\norm{\frac{1}{4\norm{z}^2} \eta_z}^2 z}{\norm{z}^2}
   -\frac{z}{2\norm{z}^6}
   +\frac{\grad\Mm|_z}{4\norm{z}^2}
   \right)
\end{split}
\end{equation*}
which is exactly the second component of~(\ref{eq:EH-eqs-Theta}).
This proves Lemma~\ref{le:Lag-Ham}.
\end{proof}

\boldmath
%%%%%%%%%%%%%%%%%%%%%%%%%%%%%%%%%%%
%%%%%%% Subsection  %%%%%%%%%%%%%%%%%%%
%%%%%%%%%%%%%%%%%%%%%%%%%%%%%%%%%%%
\subsection{Lagrangian action dominates Hamiltonian action}
\unboldmath

\begin{lemma}[Lagrangian domination]\label{le:Lag-domin}
There are the identities
\begin{equation*}
\begin{split}
   \Bb(z)
   &=\Aa(z,\eta)
   +\frac12\Norm{2\norm{z}\, z^\prime-\frac{\eta}{2\norm{z}}}^2\\
   &=\Aa(z,\eta)
   +
   \frac12\Norm{{\color{brown} 4\norm{z}^2}\, z^\prime-\eta}^2
   {\color{brown}\tfrac{1}{ 4\norm{z}^2}}
\end{split}
\end{equation*}
for every pair of loops
$(z,\eta)\in T^*\bar\Ll^\times\Zfrak$.
\end{lemma}

\begin{proof}
By definition~(\ref{eq:Aa_Hh}) of $\Aa$, just multiplying out the
inner product we get
\begin{equation*}
\begin{split}
   &\Aa(z,\eta)
   +\frac12\Norm{2\norm{z}z^\prime-\frac{\eta}{2\norm{z}}}^2
\\
   &=\inner{\eta}{z^\prime}
   -\frac{\norm{\eta}^2}{8\norm{z}^2}+\frac{1}{\norm{z}^2}
   +\Theta_z z^\prime
   +\frac12\INNER{2\norm{z}z^\prime-\frac{\eta}{2\norm{z}}}
   {2\norm{z}z^\prime-\frac{\eta}{2\norm{z}}}
\\
   &=2\norm{z}^2\norm{z^\prime}^2
   +\frac{1}{\norm{z}^2}
   +\Theta_z z^\prime
\\
   &=\Bb(z)
\end{split}
\end{equation*}
where the last step is by definition~(\ref{eq:Ss_Ll}) of $\Bb$.
This proves Lemma~\ref{le:Lag-domin}.
\end{proof}

\begin{corollary}[equal values on critical points]
\label{cor:crit-values-coincide}
Both functionals coincide
  \[
     \Bb(z)=\Aa(z,\eta_z),\qquad \eta_z
      :={\color{brown} 4\norm{z}^2}\, z^\prime ,
  \]
on critical points: solutions $z$ of~(\ref{eq:z}),
equivalently, zeroes $(z,\eta_z)$ of~(\ref{eq:EH-eqs-Theta}).
\end{corollary}

\begin{proof}
Lemma~\ref{le:Lag-domin} and Section~\ref{sec:EH=Lag}.
\end{proof}

In terms of the projection $\pi$ and injection $\iota$
in~(\ref{eq:proj-inj}) the corollary tells that
\[
   \Bb=\Aa\circ\iota
   ,\qquad
   \Bb\circ\pi=\Aa ,
\]
\underline{along critical points} of $\Bb$, respectively of $\Aa$.

\boldmath
%%%%%%%%%%%%%%%%%%%%%%%%%%%%%%%%%%%
%%%%%%% Subsubsection  %%%%%%%%%%%%%%%%
%%%%%%%%%%%%%%%%%%%%%%%%%%%%%%%%%%%
\subsubsection{The diffeomorphism $\LL$}
\unboldmath

By~(\ref{eq:Ss_Ll}) with $\Mm(z)=\Theta_z z^\prime$,
see~(\ref{eq:bhjvbhj7745}), and
$\inner{\cdot}{\cdot}_z:={\color{brown} 4\norm{z}^2}
\inner{\cdot}{\cdot}$ we have
$$
   \Bb=\Kk-\Uu+\Mm
   \colon\bar\Ll^\times\Zfrak\to\R
   ,\quad
   z\mapsto
   \tfrac12 \inner{z^\prime}{z^\prime}_z
   +\frac{1}{\norm{z}^2}                                         % - Kepler
   +\Theta_z z^\prime
   .
$$
The functional defined by
\begin{equation}\label{eq:Ll}
   \Ll=\Ll^\Theta\colon T\bar\Ll^\times\Zfrak\to\R
   ,\quad
   (z,\xi)\mapsto
   \tfrac12\INNER{\xi}{\xi}_z
   +\frac{1}{\norm{z}^2}
   +\Theta_z \xi
\end{equation}
naturally extends $\Bb(z)=\Ll(z, z^\prime)$ in the same way
as the classical action $\Ss_L(q)=\Ll_L(q,\dot q)$ is extended
by a corresponding functional $\Ll_L(q,v)$.

In the classical case a fiberwise strictly convex Lagrange
function $L$ on the tangent bundle determines a function $H$
on the cotangent bundle: resolve $\frak{p}:=d_{\frak{v}}L(r,\frak{v})$
for $\frak{v}$ and substitute the obtained $\frak{v}=\frak{v}(\frak{p})$
in the Legendre identity
\[
   \inner{\frak{p}}{\frak{v}}=L(r,\frak{v})+H(r,\frak{p})
   .
\]
Returning to the non-local situation where the manifold is loop space
and $(z,\xi)$ and $(z,\eta)$ are pairs of loops, an analogous approach yields
\[
   \eta:=d_\xi\Ll(z,\xi)={\color{brown} 4\norm{z}^2} \xi+\Theta_z
   ,\qquad
   \xi=\tfrac{\eta-\Theta_z}{{\color{brown}4\norm{z}^2}}
   ,\qquad
   d_{\xi\xi}\Ll(z,\xi)={\color{brown} 4\norm{z}^2} >0
   .
\]
The \textbf{non-local Hamiltonian function} is then defined by
\[
   \Hh^\Theta(z,\eta):=\inner{\eta}{\xi}-\Ll^\Theta(z,\xi)
\]
and with
$\inner{\cdot}{\cdot}^z:=\frac{1}{{\color{brown} 4\norm{z}^2}}\inner{\cdot}{\cdot}$
given by the formula
\begin{equation*}
\begin{split}
   \Hh^\Theta\colon T^*\bar\Ll^\times\Zfrak&\to\R
\\
   (z,\eta)&\mapsto
   \frac12\inner{\eta-\Theta_z}{\eta-\Theta_z}^z -\frac{1}{\norm{z}^2}
   .
\end{split}
\end{equation*}

Given the functional $\Ll$ in~(\ref{eq:Ll}), define the non-local analogue of
the diffeomorphism introduced in~\cite[p.\,1891]{Abbondandolo:2015c},
in the local context, by the formula
\begin{equation*}
\begin{split}
   \LL\colon T\bar\Ll^\times\Zfrak
   &\to T^*\bar\Ll^\times\Zfrak
   \\
   (z,\xi)&\mapsto\left(z,d_\xi\Ll(z, z^\prime+\xi)\right)=:(z,\eta)
\end{split}
\end{equation*}
where
\[
   \eta_\xi:=d_\xi\Ll(z, z^\prime+\xi)
   ={\color{brown} 4\norm{z}^2}\,  (z^\prime+\xi)+\Theta_z
\]
Note that since the inverse is given by
\begin{equation*}
\begin{split}
   \LL^{-1}(z,\eta)
   =\left(z,\frac{\eta-\Theta_z}{{\color{brown} 4\norm{z}^2}}-z^\prime\right)
   =:(z,\xi)
   ,
\end{split}
\end{equation*}
the solutions $(z,\eta)$ of the Euler-Hamilton
equations~(\ref{eq:EH-eqs-Theta})
are zeroes of the fiber translated map
$\LL^{-1}(\cdot,\cdot+\tfrac{\Theta_z}{4\norm{z}^2})$.

\medskip
As in the ODE case~\cite{Abbondandolo:2015c}, also in the present delay
equation situation both functionals are related through the maps
$\iota(z)=(z,4\norm{z}^2 z^\prime)$ and $\pi(z,\eta)=z$
in~(\ref{eq:proj-inj}), Lemma~\ref{le:Lag-domin}, in the form
\begin{equation*}
\begin{split}
   \Bb\circ\pi(z,\eta)
   =\Aa(z,\eta)+\Uu^*(z,\eta)
   ,\qquad
   \Uu^*(z,\eta):=\tfrac12\Norm{\iota(z)-\eta}^2{\color{brown} \tfrac{1}{4\norm{z}^2}}
\end{split}
\end{equation*}
for every $(z,\eta)\in T^*\bar\Ll^\times\Zfrak$.
Observe that the map $\Uu^*\ge 0$ vanishes precisely along
the critical points.

With the non-negative functional $\Uu$ defined and given by
\[
   \Uu(z,\xi):=\Uu^*\circ\LL(z,\xi)
   =\tfrac12\INNER{\xi+\tfrac{\Theta_z}{4\norm{z}^2}}
   {\xi+\tfrac{\Theta_z}{4\norm{z}^2}}_z\ge 0
\]
the functionals $\Aa$ and $\Bb$ are related by the formula
\begin{equation*}
\begin{split}
   \Aa\circ\LL(z,\xi)
   &=\Bb(z)-\Uu(z,\xi)\\
   &\;{\color{gray} =\tfrac12\INNER{z^\prime}{z^\prime}_z+\frac{1}{\norm{z}^2}
   +\Theta_z z^\prime
   -\tfrac12\INNER{\xi+\tfrac{\Theta_z}{4\norm{z}^2}}
   {\xi+\tfrac{\Theta_z}{4\norm{z}^2}}_z}
   .
\end{split}
\end{equation*}

\appendix
%%%%%%%%%%%%%%%%%%%%%%%%%%%%%%%%%%%
%%%%%%%%%%%%%%%%%%%%%%%%%%%%%%%%%%%
%%%%%%% Appendix:  %%%%%%%%%%%%%%%%%%%
%%%%%%%%%%%%%%%%%%%%%%%%%%%%%%%%%%%
%%%%%%%%%%%%%%%%%%%%%%%%%%%%%%%%%%%

\newpage 
\boldmath
%%%%%%%%%%%%%%%%%%%%%%%%%%%%%%%%%%%
%%%%%%%%%%%%%%%%%%%%%%%%%%%%%%%%%%%
%%%%%%% Section:  %%%%%%%%%%%%%%%%%%%%%
%%%%%%%%%%%%%%%%%%%%%%%%%%%%%%%%%%%
%%%%%%%%%%%%%%%%%%%%%%%%%%%%%%%%%%%
\section[Twisted symplectic forms on Banach manifolds]
{Twisted symplectic forms}
\label{sec:tw-sympl-form}
\unboldmath

Let $N$ be a manifold, either of finite dimension $n$
or a Banach manifold, see e.g.~\cite{lang:2001a}.
The elements of the cotangent bundle $T^*N$ are the pairs $z=(q,p)$
where $q\in N$ and $p\in T_q^*N$.
The foot-point projection to the base is the map
$\pi \colon T^*N\to N$, $z=(q,p)\mapsto q$.
The derivative is a linear map
$d\pi |_z\colon T_zT^* N\to T_qN$.
The kernel $V_z:=\ker d\pi |_z$ is called the vertical subspace of
$T_zT^*N$.
There is the canonical identification
$$
   T_q^* N\stackrel{\simeq}{\longrightarrow} V_{(q,p)}\subset T_{(q,p)} T^*N
   ,\quad
   \nu\mapsto \left.\tfrac{d}{dt}\right|_{t=0}
   (q,p+t\nu)=(0,\nu)
   .
$$

\begin{definition}\label{def:Liouville-form}
The \textbf{canonical \boldmath$1$-form} at a point
$z=(q,p)\in T^* N$ applied to tangent vector $\zeta\in T_zT^*N$ is
defined by foot-point evaluation $\lambda_z \zeta:=p(d\pi |_z\zeta)$.
\end{definition}

\begin{definition}[exterior derivative]
\label{def:exterior-derivative}
For a $1$-form $\lambda$ and a $2$-form $\omega$
one defines
\begin{equation}\label{eq:exterior-derivative}
\begin{split}
   d\lambda(X,Y):
   &=d(\lambda Y) X
   -d(\lambda X) Y
   -\lambda [X,Y]
\\
   d\omega(X,Y,Z):
   &=d(\omega(Y,Z)) X
   -d(\omega(X,Z)) Y
   +d(\omega(X,Y)) Z
   \\
   &\quad
   -\omega([X,Y],Z)
   +\omega([X,Z],Y)
   -\omega([Y,Z],X)
\end{split}
\end{equation}
for vector fields $X,Y,Z$ along $N$;
see e.g.~\cite[V\,Prop.\,3.2]{lang:2001a}.
The $d$'s on the right hand side denote the differential
of a function, evaluated on a vector field.
A form is called \textbf{closed} if its exterior derivative vanishes.
\end{definition}

\begin{remark}\label{rem:dd=0}
Pull-back commutes with $d$ and $dd=0$;
see e.g.~\cite[V\,\S 3]{lang:2001a}.
\end{remark}

\begin{lemma}\label{le:dlambda-eta}
For $z=(q,p)\in T^*N$ let $\nu\in V_z\simeq T_q^*N$
and $\zeta\in T_zT^*N$, then
$$
   d\lambda(\nu,\zeta)
   =\nu\left(d\pi |_z\zeta\right) .
$$
\end{lemma}

\begin{proof}
Since the question is local we can assume without loss of generality
that $N$ is an open subset of a vector space $W$.
Then the cotangent bundle $T^*N\to N$ can be trivialized so that
$T^*N=N\times W^*$ and $TT^*N=N\times W^*\times W\times W^*$.

Then we have $z=(q,p)$ and $\nu=(q,p;0,v)$
and $\zeta=(q,p;\xi,\eta)$.
Now we extend these vector fields constantly as follows
$\tilde\nu=(\tilde q,\tilde p;0,v)$
and $\tilde\zeta=(\tilde q, \tilde p;\xi,\eta)$.
Set $\tilde z=(\tilde q,\tilde p)$.
The local flows of the vector fields $\tilde\nu$ and $\tilde\zeta$ are
then given by
$
   \phi^t_{\tilde\nu}(\tilde q,\tilde p)
   =(\tilde q,\tilde p+tv)
$
and
$
   \phi^s_{\tilde\zeta}(\tilde q,\tilde p)
   =(\tilde q+s\xi,\tilde p+s\eta)
   .
$
\\
In particular, the flows of $\tilde\zeta$ and $\tilde\nu$ commute, thus
$[\tilde\zeta,\tilde\nu]=0$. By Definition~\ref{def:exterior-derivative}
$
   d\lambda(\nu,\zeta)|_z
   =
   d(\lambda\tilde\zeta)|_z \nu
   -d(\lambda\tilde\nu)|_z \zeta
   -(\lambda[\tilde\nu,\tilde\zeta])|_z
$.
We calculate each summand.
Summand one already represents the desired result
$$
   d(\lambda \tilde\zeta)|_z\nu
%1
   =
   \left.\tfrac{d}{dr}\right|_{r=0}
   \lambda_{\phi^r_{\nu}(z)}\tilde\zeta_{\phi^r_{\nu}(z)}
%2
    \stackrel{\color{gray}2}{=}
   \left.\tfrac{d}{dr}\right|_{r=0}
   (p+rv)
   d\pi|_{(q,p+rv)} \zeta
%3
   \stackrel{\color{gray}3}{=}
   v\xi
%4
   \stackrel{\color{gray}4}{=}
   \nu(d\pi|_z\zeta)
   .
$$
Equality 2 is by Definition~\ref{def:Liouville-form} of the
canonical $1$-form $\lambda$. It also uses that the extension vector
field $\tilde\zeta\equiv(\xi,\eta)=\zeta$ is constant, hence so is the
result under the linearized projection, namely
$d\pi|_{(q,p+rv)} \zeta\equiv\xi$.
Thus the $r$-derivative only hits the first map and we get
$v\xi$ which is equality~3.
But $v$ is the local representative of $\nu$ and $\xi$
the one of $d\pi|_z\zeta$. This is equality~4.
\newline
Summand two is zero, because
$
   \lambda_{\tilde z}\nu
   =\tilde p (d\pi|_{\tilde z}(0,v))
   =\tilde p (0)
   =0
$.
Summand three is zero since $[\tilde\nu,\tilde\zeta]=0$.
This proves Lemma~\ref{le:dlambda-eta}.
\end{proof}

\begin{remark}[weak symplectic form]\label{rem:reak-symplectic}
A \textbf{weak symplectic form} is a closed $2$-form which
is non-degenerate in the sense that for any
tangent vector there exists another vector in the same tangent space
such that plugging in these two vectors in the $2$-form
has a non-zero value.
In infinite dimension this does \emph{not} imply that a weak 
symplectic form gives rise to an \emph{isomorphism}
between tangent and cotangent space.
In finite dimension a weak symplectic form
is a symplectic form.
\end{remark}

\begin{corollary}\label{cor:omega_sigma}
Let $\sigma$ be a closed $2$-form on $N$.
Then $\omega_\sigma:=d\lambda+\pi ^*\sigma$ is a
weak symplectic form on $T^*N$.
\end{corollary}

\begin{proof}
Given that $d\omega_\sigma=dd\lambda+\pi^*d\sigma=0$,
by Remark~\ref{rem:dd=0},
it remains to show non-degeneracy.
Let $z=(q,p)\in T^*N$.
There are two cases.

\smallskip
\noindent
\textbf{Case 1.} Suppose $\zeta\in T_zT^*N$ is such that
$d\pi |_z\zeta\not=0$ is non-zero.

\smallskip
\noindent
To $0\not=d\pi |_z\zeta\in T_q N$ there exists,
by the Hahn-Banach Theorem, see e.g.~\cite[I\,Prop.\,2.3]{lang:2001a},
a dual vector $\nu \in T_q^* N\simeq V_z:=\ker d\pi |_z$
with non-zero pairing
$$
   0\not= \nu(d\pi |_z\zeta)=d\lambda(\nu,\zeta)
   =d\lambda(\nu,\zeta)+\sigma(d\pi |_z\nu,d\pi |_z\zeta)
   =:\omega_\sigma(\nu,\zeta)
$$
where equality one is Lemma~\ref{le:dlambda-eta}
and equality two adds zero since $\nu\in\ker d\pi |_z$.

\smallskip
\noindent
\textbf{Case 2.}
Suppose $\nu\in T_zT^*N$ is non-zero and $d\pi |_z\nu=0$.

\smallskip
\noindent
Since $\nu\in V_z\simeq T_q^*N$ is non-zero
there exists a dual vector $\xi_0 \in T_q N$
with non-zero pairing $\nu(\xi_0)\not= 0$.
Since $d\pi |_z$ is surjective there exists $\zeta\in T_zT^*N$
such that $d\pi |_z\zeta=\xi_0$.
By Lemma~\ref{le:dlambda-eta} and since $d\pi |_z\nu=0$ we obtain
$$
   0
   \not=\nu(\xi_0)
   =\nu(d\pi |_z\zeta)
   =d\lambda(\nu,\zeta)
   =d\lambda(\nu,\zeta) +\sigma(d\pi |_z\nu,d\pi |_z\zeta)
   =:\omega_\sigma(\nu,\zeta) .
$$

This proves Corollary~\ref{cor:omega_sigma}.
\end{proof}

\boldmath
%%%%%%%%%%%%%%%%%%%%%%%%%%%%%%%%%%%
%%%%%%%%%%%%%%%%%%%%%%%%%%%%%%%%%%%
%%%%%%% Section  %%%%%%%%%%%%%%%%%%%%%%
%%%%%%%%%%%%%%%%%%%%%%%%%%%%%%%%%%%
%%%%%%%%%%%%%%%%%%%%%%%%%%%%%%%%%%%
\section[Non-exact magnetic $2$-form $\Sigma$ -- twisted-periodic case]
{Non-exact magnetic $2$-form $\Sigma$}
\label{sec:non-loc-Ham-twisted}
\unboldmath

In the case where a $1$-form family,
say $\vartheta_t=\varsigma^*\theta_t$,
is \emph{twisted-periodic}, and \emph{not periodic}, we do not have
a $1$-form $\Theta$ on loop space
${\color{red}\bar\Ll^\times\Zfrak}$ as in~(\ref{eq:involution-quotients}),
so $d\Theta$ as in~(\ref{eq:dTheta}) is not available
to describe a magnetic $2$-form.
However, the key property of a magnetic $2$-form is closedness.
Therefore we use the right hand side of~(\ref{eq:dTheta})
as definition of a $2$-form $\Sigma$ and then show $d\Sigma=0$
by hand.

\begin{definition}[Magnetic $2$-form of $\vartheta$]
For $z\in \bar\Ll^\times\Zfrak$ and
$X,Y\in T_z \bar\Ll^\times\Zfrak$ set
\begin{equation}\label{eq:Sigma}
\begin{split}
   &\Sigma_z( X, Y)
\\
   :&=
   \int_0^1
   \Bigl(
   d\vartheta_{t_z}|_{z}\left( X, Y\right)
   +\bigl(dt|_z X\bigr)\dot\vartheta_{t_z}|_{z}Y
   -\bigl(dt|_z Y\bigr)\dot\vartheta_{t_z}|_{z}X
   \Bigr)_\tau d\tau
   .
\end{split}
\end{equation}
\end{definition}

Note that $d\vartheta_t$ and $\dot\vartheta_t$ are both periodic in $t$
by~(\ref{eq:df_t}).
Hence the $\SS^1$-output of $t_z\colon\SS^1_\tau\to\SS^1_t$ is well received
and, furthermore, integration $\int_r^{r+1}$ leads to the same value
independent of $r\in\R$.

\begin{theorem}\label{thm:twisted-symplectic}
It holds $d\Sigma=0$ and
$
   \Omega_\Sigma
   :=\Omegacan+ P^*\Sigma
$
is a weak symplectic form on the cotangent bundle
$T^*\bar\Ll^\times\Zfrak$,
said \textbf{twisted weak symplectic}.
\end{theorem}

\begin{proof}
By Corollary~\ref{cor:omega_sigma} it suffices to
prove $d\Sigma=0$.
To ease presentation we change our usual loop notation $z$ to $\zfrak$.
Let $\zfrak\in \bar\Ll^\times\Zfrak$ and
$X,Y,Z\in T_\zfrak \bar\Ll^\times\Zfrak$.
Since our manifold is an open subset of a vector space,
namely $\Zfrak\subset\C$, 
therefore we constantly extend the tangent vectors at $\zfrak$
to tangent vectors at any
$\tilde\zfrak \in \bar\Ll^\times\Zfrak$, 
namely we set time-wise $X_{\tilde\zfrak}(\tau):=X(\tau)$.
In particular, commutators vanish
and the two formulas~(\ref{eq:exterior-derivative}) simplify to
\begin{equation}\label{eq:exterior-derivative-simple}
\begin{split}
   d\lambda(X,Y):
   &=d(\lambda Y) X
   -d(\lambda X) Y
\\
   d\omega(X,Y,Z):
   &=d(\omega(Y,Z)) X
   -d(\omega(X,Z)) Y
   +d(\omega(X,Y)) Z
   .
\end{split}
\end{equation}
To shorten formulas we abbreviate
$x:=(X(\zfrak))_\tau$, $y:=Y_\tau$, and $z:=Z_\tau$. 
To avoid parentheses we shall write $d\vartheta_t(x,y)$
as $d\vartheta_t xy$.
We proceed similar to~(\ref{eq:dTheta}),
but during equality two below we omit the step using the letter $D$.
Notice that $\Sigma$ defined by~(\ref{eq:Sigma})
has three summands which in equality two
are treated one after the other, visualized by three integrals.
Using repeatedly~(\ref{eq:exterior-derivative-simple}) we get
\begin{equation*}
\begin{split}
   &d\Sigma\left(X,Y,Z\right)
%1
   \stackrel{\text{(\ref{eq:exterior-derivative-simple})}}{=}
   d(\Sigma (Y,Z)) X
   -d(\Sigma (X,Z)) Y
   +d(\Sigma (X,Y)) Z
   \\
%2
   &\stackrel{\text{(\ref{eq:Sigma})}}{=}
   \int_0^1\Bigl(
%%%%%%%%%%%%%%%%%%
% Sigma 1 - erste Drei
%%%%%%%%%%%%%%%%%%
{\color{cyan}
   d_x\left(d\vartheta_{t_\zfrak}|_\zfrak yz\right)
}
\qquad\qquad   +d\dot\vartheta_{t_\zfrak}|_\zfrak yz(dt_\zfrak X)
   \\
   &\quad\qquad
{\color{cyan}
   \qquad-d_y\left(d\vartheta_{t_\zfrak}|_\zfrak xz\right)
}
\qquad\quad\;\,   -d\dot\vartheta_{t_\zfrak}|_\zfrak xz(dt_\zfrak Y)
   \\
   &\quad\qquad
{\color{cyan}
   \qquad+d_z\left(d\vartheta_{t_\zfrak}|_\zfrak xy\right)
}
\qquad\quad\;\,   +d\dot\vartheta_{t_\zfrak}|_\zfrak xy(dt_\zfrak Z)
   \Bigr) d\tau
   \\
%   &\quad{\color{white}=}
   &
%%%%%%%%%%%%%%%%%%
% Sigma 2 - erste Drei
%%%%%%%%%%%%%%%%%%
   +\int_0^1\Bigl(
{\color{red}\;
   \dot\vartheta_{t_\zfrak}|_\zfrak z\cdot d_X(dt|_\zfrak Y)
}
   +(dt|_\zfrak Y)\cdot d_x(\dot\vartheta_{t_\zfrak}|_\zfrak z)
{\color{gray}\;
   +\; (dt|_\zfrak Y)\cdot\ddot\vartheta_{t_\zfrak}|_\zfrak z\cdot(dt|_\zfrak X)
}
\\
   &\qquad
{\color{red}\;
   -\;\dot\vartheta_{t_\zfrak}|_\zfrak z\cdot d_Y(dt|_\zfrak X)
}
   -(dt|_\zfrak X)\cdot d_y(\dot\vartheta_{t_\zfrak}|_\zfrak z)
{\color{gray}\;
   -\; (dt|_\zfrak X)\cdot\ddot\vartheta_{t_\zfrak}|_\zfrak z\cdot(dt|_\zfrak Y)
}
   \\
   &\qquad
{\color{red}\;
   +\;\dot\vartheta_{t_\zfrak}|_\zfrak y\cdot d_Z(dt|_\zfrak X)
}
   +(dt|_\zfrak X)\cdot d_z(\dot\vartheta_{t_\zfrak}|_\zfrak y)
{\color{gray}\;
   +\;
\underline{
   (dt|_\zfrak X)\cdot\ddot\vartheta_{t_\zfrak}|_\zfrak y\cdot(dt|_\zfrak Z)
}
}
   \Bigr) d\tau
   \\
%   &\quad{\color{white}=}
   &
%%%%%%%%%%%%%%%%%%
% Sigma 3 - erste Drei
%%%%%%%%%%%%%%%%%%
   +\int_0^1\Bigl(
{\color{red}\;
   -\;\dot\vartheta_{t_\zfrak}|_\zfrak y\cdot d_X(dt|_\zfrak Z)
}
   -(dt|_\zfrak Z)\cdot d_x(\dot\vartheta_{t_\zfrak}|_\zfrak y)
{\color{gray}\;
   -\;
\underline{
   (dt|_\zfrak Z)\cdot\ddot\vartheta_{t_\zfrak}|_\zfrak y\cdot(dt|_\zfrak X)
}
}
\\
   &\qquad
{\color{red}\;
   +\;\dot\vartheta_{t_\zfrak}|_\zfrak x\cdot d_Y(dt|_\zfrak Z)
}
   +(dt|_\zfrak Z)\cdot d_y(\dot\vartheta_{t_\zfrak}|_\zfrak x)
{\color{gray}\;
   +\; (dt|_\zfrak Z)\cdot\ddot\vartheta_{t_\zfrak}|_\zfrak x\cdot(dt|_\zfrak Y)
}
   \\
   &\qquad
{\color{red}\;
   -\;\dot\vartheta_{t_\zfrak}|_\zfrak x\cdot d_Z(dt|_\zfrak Y)
}
   -(dt|_\zfrak Y)\cdot d_z(\dot\vartheta_{t_\zfrak}|_\zfrak x)
{\color{gray}\;
   -\; (dt|_\zfrak Y)\cdot\ddot\vartheta_{t_\zfrak}|_\zfrak x\cdot(dt|_\zfrak Z)
}
   \Bigr) d\tau
\\
%3
   &=0
   .
\end{split}
\end{equation*}
Equality three:
The sum of the cyan terms is
{\color{cyan}$dd\theta_{t_\zfrak}|_\zfrak (x,y,z)$}
by~(\ref{eq:exterior-derivative-simple}), hence it vanishes since 
for the finite dimensional exterior derivative $dd=0$.
The {\color{gray}gray terms} cancel pairwise.
Since the loop space exterior derivative satisfies $dd=0$ by
Remark~\ref{le:dlambda-eta}, the six {\color{red}red terms} cancel
pair-wise. For instance, the two terms with common factor
${\color{red}\dot\vartheta_{t_\zfrak}|_\zfrak z}$ combine,
by~(\ref{eq:exterior-derivative-simple}),
as follows
$
   (d_X(dt|_\zfrak Y)- (d_Y(dt|_\zfrak X)
   =ddt|_\zfrak(X,Y)
   =0
$.
Again by~(\ref{eq:exterior-derivative-simple})
the remaining 9 black terms cancel triple-wise
$
   d\vartheta_{t_\zfrak}|_\zfrak(x,y)
   -d_x\bigl(\dot\vartheta_{t_\zfrak}|_\zfrak y\bigr)
   +d_y\bigl (\dot\vartheta_{t_\zfrak}|_\zfrak x\bigr)
   =0
$.
This proves Theorem~\ref{thm:twisted-symplectic}.
\end{proof}

\newpage
\boldmath
%%%%%%%%%%%%%%%%%%%%%%%%%%%%%%%%%%%
%%%%%%%%%%%%%%%%%%%%%%%%%%%%%%%%%%%
%%%%%%% Section  %%%%%%%%%%%%%%%%%%%%%%
%%%%%%%%%%%%%%%%%%%%%%%%%%%%%%%%%%%
%%%%%%%%%%%%%%%%%%%%%%%%%%%%%%%%%%%
\section{Periodic magnetic $1$-forms on loop spaces}
\label{sec:magn-fctl}
\unboldmath

This section lives in an abstract setting
were $M$ is a manifold of finite dimension, not necessarily a
subset $\Zfrak$ of the plane.
Furthermore, we consider general circle reparametrizations
$T_z$ and $\Tt_z$, not necessarily the Barutello-Ortega-Verzini
reparametrizations $t_z$ and $\tau_z$.
\smallskip\newline
\textbf{Note.} Outside of Appendix~\ref{sec:magn-fctl}
the letters $T$ and $\Tt$ have a different meaning.

\begin{definition}[time reparametrization]
A smooth map from \textbf{loop space}
$\Ll M:=\{z\in C^\infty(\R,M)\mid
\forall t\in\R\colon z(t+1)=z(t)\}$
to the group of orientation preserving circle diffeomorphisms
$$
   T\colon\Ll M\to\Diff_+\SS^1
   ,\quad
   z\mapsto T(z)=:T_z
$$
is called a \textbf{time reparametrization}.
The \textbf{inverse} time reparametrization is the map
$
   \Tt\colon\Ll M\to\Diff_+\SS^1
$,
$
   z\mapsto (T_z)^{-1}=:\Tt_z
$.
\end{definition}

A \textbf{periodic} $1$-form on a manifold $M$,
notation $\vartheta=\{\vartheta_t\}_{t\in\SS^1}$,
is a one-parameter-family of $1$-forms on $M$ such that
$\vartheta_{t+1}=\vartheta_t$ for every $t\in\R$.

\begin{definition}[magnetic $1$-form on loop space]
\label{def:Theta}
A periodic $1$-form $\vartheta$ on a manifold $M$
and a time reparametrization $T$ of loop space $\Ll M$
induce on $\Ll M$ a $1$-form $\Theta$
as follows. At a loop $z$ the
\textbf{magnetic \boldmath$1$-form}
is on a tangent vector $\xi\in T_z\Ll M$,
i.e. a vector field along $z$, defined by
\begin{equation}\label{eq:Theta}
   \Theta_z\xi=\Theta_z^{\vartheta,T}\xi
   :=\int_{\SS^1} \vartheta_{T_z}|_z \xi
   :=\int_0^1 \vartheta_{T_z(\tau)}|_{z_\tau} \xi_\tau\, d\tau
   .
\end{equation}
\end{definition}

To write the second integral we lift $T_z\colon\SS^1_\tau\to\SS^1_t$
to a map, still denoted by $T_z\colon\R\to\R$,
which is equivariant with respect to the $\Z$-action on $\R$
given by $\Z\times\R\to\R$, $(k,\tau)\mapsto \tau+k$
and such that $T_z(0)=0$, in symbols
\begin{equation}\label{eq:equiv}
   T_z(\tau+k)=T_z(\tau)+ k
   ,\qquad
   T_z(0)=0
   .
\end{equation}

Since $\vartheta$ is periodic $\Theta$ is well defined, independent of
the choice of integration interval as long as it covers one period;
cf. Remark~\ref{rem:interval-not-SS^1}.

\begin{remark}[canonical vector field on loop space]
\label{rem:vv-flow}
Along loop space there is a \textbf{canonical vector field}
defined by time derivative  $\nu(z):=z^\prime$,
cf.~(\ref{eq:can-vf-loops}).
The flow on $\Ll M$ of the canonical vector field $\nu=\p_\tau$ is
time shift, in symbols
\begin{equation}\label{eq:shift}
   (\phi^r_\nu  z)_\tau=z_{\tau+r}
   ,\qquad
   \left(d \phi^r_\nu|_z\xi\right)_\tau=\xi_{\tau+r}
   ,
\end{equation}
for every time $\tau\in\R$ and where $z\in\Ll M$ and
$\xi\in T_z\Ll M$; cf.~(\ref{eq:shift-nu}).
\end{remark}

\begin{definition}[magnetic functional]
To a periodic $1$-form $\vartheta$ on $M$
and a time reparametrization $T$ of loop space $\Ll M$
we associate a \textbf{magnetic functional}~by
evaluating the associated magnetic $1$-form $\Theta^{\vartheta,T}$
along the canonical vector field
\begin{equation}\label{eq:Mm}
   \Mm:=i_\nu\Theta^{\vartheta,T}\colon\Ll M\to\R
   ,\quad
   z\mapsto
   \int_{\SS^1} \vartheta_{T_z}|_zz^\prime
   :=\int_0^1\vartheta_{T_z(\tau)}|_{z_\tau} z_\tau^\prime\, d\tau
   .
\end{equation}
The definition is~(\ref{eq:Theta}).
\end{definition}

%%%%%%%%%%%%%%%%%%%%%%%%%%%%%%%%%%%
%%%%%%% Subsubsection  %%%%%%%%%%%%%%%%
%%%%%%%%%%%%%%%%%%%%%%%%%%%%%%%%%%%
\subsubsection*{Differential of magnetic functional}

Motivated by Cartan's formula in finite dimension
the differential of the magnetic functional $\Mm$
at $(z,\xi) \in T\Ll M$ should be given by the difference
\begin{equation*}%\label{eq:dMm-abstract}
\begin{split}
   (d\Mm)_z\xi
   {\color{gray}\,=(di_\nu \Theta)_z\xi\,}
   =(L_\nu \Theta)_z\xi-(i_\nu  d\Theta)_z\xi .
%   =(L_\nu \Theta)_z\xi-(d\Theta)_z(z^\prime,\xi) .
\end{split}
\end{equation*}
We prove this identity in case
of the main player in this article, the magnetic functional
$\Mm$ defined in~(\ref{eq:Mm-BOV})
via the Barutello-Ortega-Verzini
reparametrization $T_z=t_z$ in~(\ref{eq:t_z})
and the pull-back $\vartheta=\varsigma^*\theta$.
To get there we calculate the Lie and exterior derivatives of
the more general magnetic $1$-form~$\Theta$ in~(\ref{eq:Theta}).

%%%%%%%%%%%%%%%%%%%%%%%%%%%%%%%%%%%
%%%%%%% Subsubsection  %%%%%%%%%%%%%%%%
%%%%%%%%%%%%%%%%%%%%%%%%%%%%%%%%%%%
\subsubsection*{Lie derivative -- Euler force}

\begin{lemma}\label{le:Lie-magnetic}
The Lie, or fisherman, derivative is defined and given by
\begin{equation*}
\begin{split}
   L_\nu \Theta |_z :
   =\left.\tfrac{d}{dr}\right|_{r=0} (\phi^r_\nu )^*\Theta |_z
   &
   =
   \int_0^1
{\color{gray}
   \underbrace{
{\color{black}
   \left(
   \bigl(\underline{dT|_z z^\prime}\bigr)_\tau
   -\tfrac{d}{d\tau} T_z(\tau)
   \right)
}
   }_{\stackrel{\text{\rm(\ref{eq:dT-Tprime})}}{=} \left.\tfrac{d}{dr}\right|_{r=0} {T_{\phi^r_\nu (z)}(\tau-r)}\;\in\R}
}
   \;
   \dot\vartheta_{T_z(\tau)}|_{z_\tau}\; d\tau
\end{split}
\end{equation*}
at any loop $z\in\Ll M$.
\end{lemma}

\begin{proof}
Let $z\in\Ll M$ and $\xi\in T_z\Ll M$. As a preparation we
compute\footnote{
  Actually, by the footnote to Remark~\ref{rem:interval-not-SS^1}~(iii)
  we'd get
  $\stackrel{{\color{gray}3}}{=}
   \int_0^1\vartheta_{{\color{red}T_{\phi^r_\nu (z)}(\tau-r)}}|_{z_\tau}
   \xi_{\tau}\, d\tau$
  in the periodic case. But formula~(\ref{eq:hjbhjvgh77})
  arising from $\stackrel{{\color{gray}2}}{=}$ could be useful in the
  twisted-periodic case.
  }
\begin{equation}\label{eq:prep}
\begin{split}
   \left({\phi^r_\nu }^*\Theta\right)_z\xi
   :=\Theta_{\phi^r_\nu (z)}\, d\phi^r_\nu |_z\xi
   &\stackrel{{\color{gray}1}}{=}
   \int_0^1 \vartheta_{T_{\phi^r_\nu (z)}(\sigma)}|_{z_{\sigma+r}}
   \xi_{\sigma+r}\, d\sigma
\\
   &\stackrel{{\color{gray}2}}{=}
   \int_{r}^{r+1}\vartheta_{{\color{red}T_{\phi^r_\nu (z)}(\tau-r)}}|_{z_\tau}
   \xi_{\tau}\, d\tau
    .
\end{split}
\end{equation}
Equality~1 uses~(\ref{eq:shift}) in definition~(\ref{eq:Theta}).
Equality~2 is by change of variables $\tau(\sigma):=\sigma+r$.
\\
Since the expression $\left.\tfrac{d}{dr}\right|_{r=0} T_{\phi^r_\nu (z)}$ is
an element of $T_{T_z}\Diff_+\SS^1=C^\infty(\SS^1,\R)$,
evaluated at $\tau\in\SS^1$ it becomes a real number.
For $r\in\R$ Leibniz yields
\begin{equation}\label{eq:dT-Tprime}
\begin{split}
   \left.\tfrac{d}{dr}\right|_{r=0} {\color{red}T_{\phi^r_\nu (z)}(\tau-r)}
   &=\left(\left.\tfrac{d}{dr}\right|_{r=0} T_{\phi^r_\nu (z)}\right)_{\tau}
   +\left.\tfrac{d}{dr}\right|_{r=0} T_z(\tau-r)
\\
   &=\left(dT|_z z^\prime\right)_\tau
   -\tfrac{d}{d\tau} T_z(\tau)\in\R
   .
\end{split}
\end{equation}
Differentiating equality~2 in~(\ref{eq:prep}) we obtain equality 1
in what follows
\begin{equation}\label{eq:hjbhjvgh77}
\begin{split}
   &\left.\tfrac{d}{dr}\right|_{r=0}
   \left({\phi^r_\nu }^*\Theta\right)_z\xi
\\
   &\stackrel{{\color{gray}1}}{=}
   \left.\tfrac{d}{dr}\right|_{r=0}
   \int_{r}^{r+1}\vartheta_{{\color{red}T_{\phi^r_\nu (z)}(\tau-r)}}|_{z(\tau)}
   \xi_{\tau}\, d\tau
\\
   &\stackrel{{\color{gray}2}}{=}
   \vartheta_{T_z(1)}|_{z_1}\xi_1
   -\vartheta_{T_z(0)}|_{z_0}\xi_0
   +\int_0^1\left.\tfrac{d}{dr}\right|_{r=0}
   \vartheta_{{\color{red}T_{\phi^r_\nu (z)}(\tau-r)}}|_{z(\tau)}
   \xi_{\tau}\, d\tau
\\
   &\stackrel{{\color{gray}3}}{=}
   \vartheta_{1}|_{z_0}\xi_0
   -\vartheta_{0}|_{z_0}\xi_0
   +\int_0^1 \dot \vartheta_{T_z(\tau)}|_{z(\tau)}
   \left(
   \left(dT|_z z^\prime\right)_\tau
   -\tfrac{d}{d\tau} T_z(\tau)
   \right)
   \xi_{\tau}\, d\tau
   .
\end{split}
\end{equation}
Equality~2 is by the Leibniz integral rule and~(\ref{eq:dT-Tprime}).
Equality~3 is by~(\ref{eq:equiv}), by periodicity $z_1=z_0$ and
$\xi_1=\xi_0$, and by the chain rule.
Now use periodicity $\vartheta_1-\vartheta_0=0$.
This proves Lemma~\ref{le:Lie-magnetic}.
\end{proof}

\begin{example}[identity reparametrization]
For the constant reparametrization
$T\equiv \id_{\SS^1}\in\Diff_+\SS^1$ the
value of~(\ref{eq:dT-Tprime}) is
$0-\tfrac{d}{d\tau} \tau=-1$.
So by Lemma~\ref{le:Lie-magnetic}
$$
   L_\nu \Theta |_z
   =-\int_0^1 \dot\vartheta_{T_z(\tau)}|_{z_\tau}\; d\tau .
$$
\end{example}

\begin{example}[BOV-reparametrization -- Euler force]
\label{ex:BOV-Lie}
Let $\vartheta_t=\varsigma^*\theta_t$.
Consider the Barutello-Ortega-Verzini
reparametrization $T_z=t_z$ in~(\ref{eq:t_z-prime-coll})
for $z\in {\color{red}\bar\Ll^\times\Zfrak}$, 
cf.~(\ref{eq:involution-quotients}).
Plug formula~(\ref{eq:dt(z)}) for $(dt|_z\xi)_\tau$
and~(\ref{eq:t_z-prime-coll}) for $t_z^\prime$
into~(\ref{eq:dT-Tprime}) to get
\begin{equation*}
\begin{split}
   \left.\tfrac{d}{dr}\right|_{r=0} t_{\phi^r_\nu (z)}(\tau-r)
   &=
   \frac{1}{\norm{z}^2}\int_0^\tau2\INNER{z_\sigma }{z_\sigma^\prime}_0
   d\sigma 
   -\frac{2\INNER{z}{z^\prime}}{\norm{z}^4}
   \int_0^\tau\Abs{z_\sigma }^2 d\sigma 
   -\frac{\Abs{z_\tau}^2}{\norm{z}^2}
\\
   &\stackrel{{\color{gray}2}}{=}
   \frac{\Abs{z_\tau}^2-\Abs{z_0}^2}{\norm{z}^2}
   -\frac{\Abs{z_1}^2-\Abs{z_0}^2}{\norm{z}^4}
   \int_0^\tau\Abs{z_\sigma }^2 d\sigma
   -\frac{\Abs{z_\tau}^2}{\norm{z}^2}
\\
   &=-\frac{\Abs{z_0}^2}{\norm{z}^2} .
\end{split}
\end{equation*}
Equality~2 uses that $\int_0^s 2\INNER{z_\sigma }{z_\sigma ^\prime}ds
=\int_0^s \tfrac{d}{d\sigma}\Abs{z_\sigma}^2 ds
=\Abs{z_s}^2-\Abs{z_0}^2$.
For $s=1$ the difference vanishes since $z_1=\pm z_0$.
Twisted-periodicity of $\theta_t$ with function $f$ implies such
for $\vartheta_t$ with function $F:=\varsigma^* f=f\circ\varsigma$.
Hence Lemma~\ref{le:Lie-magnetic} 
yields
\begin{equation}\label{eq:BOV-Lie}
\small
   L_\nu \Theta |_z
%1
   = %\stackrel{{\color{gray}1}}{=}
   -\tfrac{\Abs{z_0}^2}{\norm{z}^2}
   \int_0^1 \dot\vartheta_{t_z(\tau)}|_{z_\tau}\; d\tau
%2
   \stackrel{{\color{gray}2}}{=}
   \tfrac{1}{4\norm{z}^2}
   \INNER{-4\abs{z_0}^2\dot\vartheta_{t_z(\tau)}|_{z_\tau}}{\cdot}
%3
   \stackrel{{\color{gray}3}}{=}
   \INNER{\Yy_\Theta^2|_{(z,\eta)}}{\cdot}
\end{equation}
$\forall (z,\eta)\in T^*\bar\Ll^\times\Zfrak$.
Step~2 uses that $\R^2\simeq(\R^2)^*$, step~3 is the Euler
force~(\ref{eq:Yy}).
\end{example}

%%%%%%%%%%%%%%%%%%%%%%%%%%%%%%%%%%%
%%%%%%% Subsubsection  %%%%%%%%%%%%%%%%
%%%%%%%%%%%%%%%%%%%%%%%%%%%%%%%%%%%
\subsubsection*{Exterior derivative -- Lorentz force}

Let $\xi,\eta\in T_z\Ll M$.
In order to apply formula~(\ref{eq:exterior-derivative})
involving commutators,
we extend the tangent vectors $\xi$ and $\eta$ to vector fields
defined in a neighborhood of $z$ in $\Ll M$
which we denote by the same letters $\xi$ and $\eta$.
By definition~(\ref{eq:exterior-derivative})
we get equality one
\begin{equation}\label{eq:dTheta}
\small
\begin{split}
   &(d\Theta)_z(\xi,\eta)
\\
  &=
   \int_0^1
   \Bigl(
   D_{\xi}\left(\vartheta_{T_z(\cdot)}|_{z_\cdot}\eta_\cdot\right)_\tau 
   -D_\eta\left(\vartheta_{T_z(\cdot)}|_{z_\cdot} \xi_\cdot\right)_\tau 
   -\vartheta_{T_z(\tau)}|_{z_\tau}[\xi,\eta]_\tau
   \Bigr) d\tau
\\
   &\stackrel{{\color{gray}2}}{=}\int_0^1
   \Bigl(
   d_{\xi}\left(\vartheta_{T_z(\cdot)}|_{z_\cdot}\eta_\cdot\right)_\tau 
   -d_\eta\left(\vartheta_{T_z(\cdot)}|_{z_\cdot} \xi_\cdot\right)_\tau 
   -\vartheta_{T_z(\tau)}|_{z_\tau}[\xi,\eta]_\tau
   \\
   &\qquad
   +\dot\vartheta_{T_z(\tau)}|_{z_\tau} \left(dT|_z \xi\right)_\tau\eta_\tau
   -\dot\vartheta_{T_z(\tau)}|_{z_\tau} \left(dT|_z\eta\right)_\tau \xi_\tau
   \Bigr) d\tau
\\
   &\stackrel{{\color{gray}3}}{=}
   \int_0^1
   \Bigl(
   d\vartheta_{T_z(\tau)}|_{z_\tau}\left(\xi_\tau,\eta_\tau\right)
   +\dot\vartheta_{T_z(\tau)}|_{z_\tau}
   \bigl(
      \bigl(\underline{dT|_z \xi}\bigr)_\tau\eta_\tau
      -\left(dT|_z\eta\right)_\tau \xi_\tau
   \bigr)
   \Bigr) d\tau
   .
\end{split}
\end{equation}
The letter $D$ indicates differentiation with respect to all variables,
spatial $d$ and time $\p_t$.
Equality~2 unpacks $D$.
Equality~3 is by~(\ref{eq:exterior-derivative}) for $d$ and
$\lambda=\vartheta_T$.

\begin{remark}[Lorentz force]\label{rem:Lorentz}
On a finite dimensional manifold $M$
let $g$ be a Riemannian metric and $\sigma_t$ a family of closed
$2$-forms depending on a parameter $t\in\R$,
e.g. a family of exact forms $\sigma_t=d\vartheta_t$.
Then at any point $q\in M$ a $g$-anti-symmetric
linear map $Z_t|_q\colon T_qM\to T_qM$
is defined by the $1$-form identity
$$
   g_q(Z_t|_q v,\cdot)=-\sigma_t|_q(v,\cdot)
$$
equivalently $(Z_t|_q v)^\flat=-i_v\sigma_t|_q$
where $\flat\colon T_qM\to T_q^*M$ is the metric isomorphism
$\xi\mapsto g_q(\xi,\cdot)$ and $\sharp$ the inverse.
This is the \textbf{Lorentz force}, it is given~by
$$
   Z_t|_q v=-(i_v\sigma_t|_q)^\sharp .
$$
\textbf{Euclidean space.} On $M=\R^3$, given a $1$-form $\theta$,
let $\AAA$ be the dual vector field, cf.~(\ref{eq:tw-per-1-form}),
then the magnetic vector field is given by $\BBB=\rot\, \AAA$.
In terms of differential forms this corresponds to $\sigma=d\theta$
and $\sigma$ and $\BBB$ are related via the Hodge $*$-operator
by $\sigma=*\BBB^\sharp$; see e.g.~\cite{Weber:2017b}.
In this context the Lorentz force on a particle of electric charge
$c\in\R$ at a point $\rrr(t)$ at time $t$, namely $cZ|_{r(t)} \dot r(t)$,
is of the familiar cross product form $c \dot r(t)\times \BBB|_{r(t)}$.
This ends Remark~\ref{rem:Lorentz}.
\end{remark}

\begin{example}[identity reparametrization]
For the constant reparametrization
$T\equiv \id_{\SS^1}\in\Diff_+\SS^1$ term $dT|_z=0$ vanishes
in~(\ref{eq:dTheta}).
The identity
\begin{equation*}
\small
\begin{split}
   \int_0^1 \inner{(\Zz_z\xi)_\tau}{\eta_\tau}_0\, d\tau
   &=\inner{\Zz_z\xi}{\eta}
   \stackrel{\text{(\ref{eq:Zz}))}}{=}
   -(d\Theta)_z(\xi,\eta)
   \stackrel{\text{(\ref{eq:dTheta}))}}{=}
   \int_0^1 -d\vartheta_\tau|_{z_\tau}\left(\xi_\tau,\eta_\tau\right) d\tau
\end{split}
\end{equation*}
$\forall \xi,\eta$
determines the non-local Lorentz force $\Zz_z$.
The equality of Integrands
$
   \inner{(\Zz_z\xi)_\tau}{\cdot}_0
   =-d\vartheta_\tau|_{z_\tau}\left(\xi_\tau,\cdot\right)
$,
the latter is
$
   \inner{Z_\tau|_{z_\tau} \xi_\tau}{\cdot}_0
$, tells
$
   (\Zz_z\xi)_\tau=Z_\tau|_{z_\tau} \xi_\tau
$.
\end{example}

\begin{example}[BOV-reparametrization -- Lorentz force]
\label{ex:BOV-Lorentz}
Let $\vartheta_t=\varsigma^*\theta_t$.
For the Barutello-Ortega-Verzini
reparametrization $t_z$ in~(\ref{eq:t_z-prime-coll})
the \textbf{Lorentz force} at $z\in
{\color{red}\bar\Ll^\times\Zfrak}$,
cf.~(\ref{eq:involution-quotients}),
is the linear map $\Zz_z\colon T_z\bar\Ll^\times\Zfrak
\to T_z \bar\Ll^\times\Zfrak$ determined~by
$$
   \inner{\Zz_z\zeta}{\xi}
   =
   -d\Theta_z(\zeta,\xi)
$$
for all $\xi,\zeta\in T_z \bar\Ll^\times\Zfrak$
where $d\Theta$ is given by~(\ref{eq:dTheta}).\footnote{
  \textbf{Note.}
  In the twisted-periodic, and not periodic, case let's take the same
  formula~(\ref{eq:Sigma}) to define a $2$-form, say $\Sigma$.
  In doing so all we loose is exactness. Define $\Zz$ as above.
  }
In particular for $\zeta=z^\prime$ computation~(\ref{eq:dMm-BOV})
below, read backwards, yields equality one in what follows
\begin{equation}\label{eq:Lorentz-BOV}
\begin{split}
   \INNER{\grad\Mm|_z}{\cdot}
   &=(L_\nu \Theta)_z\cdot-(d\Theta)_z(z^\prime,\cdot)
\\
   &\stackrel{{\color{gray}2}}{=}
   \INNER{\Yy_\Theta^2+\Zz_z z^\prime}{\cdot}
\\
   &\stackrel{{\color{gray}3}}{=}
   \INNER{\tfrac{1}{4\norm{z}^2}
   (-4\abs{z_0}^2\dot\vartheta_{t_z(\tau)}|_{z_\tau})+\Zz_z z^\prime}{\cdot}
   .
\end{split}
\end{equation}
Equality~2 uses the relation~(\ref{eq:BOV-Lie})
between Lie derivative and Euler force $\Yy_\Theta$
as well as the above definition of the Lorentz
force $\Zz$.
Equality~3 is by~(\ref{eq:BOV-Lie}) again.
This concludes Example~\ref{ex:BOV-Lorentz}.
\end{example}

\boldmath
%%%%%%%%%%%%%%%%%%%%%%%%%%%%%%%%%%%
%%%%%%% Subsubsection  %%%%%%%%%%%%%%%%
%%%%%%%%%%%%%%%%%%%%%%%%%%%%%%%%%%%
\subsubsection*{Difference of Lie and exterior derivative -- $d\Mm$}
\unboldmath

\begin{proposition}\label{prop:dMm-integral}
For $z\in\Ll M$ and $\xi\in T_z\Ll M$ there is the 
identity
\begin{equation*}%\label{eq:dMm-integral}
\begin{split}
   &(L_\nu \Theta)_z\xi-(d\Theta)_z(z^\prime,\xi)
   =(L_\nu \Theta)_z\xi
   {\color{cyan}\;-(d\Theta)_z(z^\prime,\xi)}
\\
   &=\int_0^1
   \dot\vartheta_{T_z(\tau)}|_{z_\tau}
   \left(
{\color{cyan}\;
   \left(dT|_z\xi\right)_\tau z_\tau^\prime
}
   -\left(\tfrac{d}{d\tau} T_z(\tau)\right) \xi_\tau
   \right)
{\color{cyan}\; 
  -d\vartheta_{T_z(\tau)}|_{z_\tau}\left(z_\tau^\prime,\xi_\tau\right)
}
   \, d\tau
   .
\end{split}
\end{equation*}
\end{proposition}

\begin{proof}
Lemma~\ref{le:Lie-magnetic}
and~(\ref{eq:dTheta});
the two underlined terms add to zero.
\end{proof}

\begin{corollary}\label{cor:BOV-Cartan-true}
For the BOV-reparametrization $t_z$ in~(\ref{eq:t_z})
and $\vartheta_t=\varsigma^*\theta_t$
Cartan's formula $L_\nu=i_\nu d+di_\nu$
holds for $\Theta$ along loop space $\bar\Ll_+^\times\Zfrak$;
see~(\ref{eq:involution-quotients}).
\end{corollary}

\begin{example}[BOV-reparametrization and
$\vartheta_t=\varsigma^*\theta_t$]
\label{ex:BOV-repara}
In the setting of the Barutello-Ortega-Verzini
reparametrization in Section~\ref{sec:BOV-blow-up}
replace $\Ll M$ by ${\color{red}\bar\Ll^\times\Zfrak}$,
cf.~(\ref{eq:involution-quotients}),
and $T_z$ by $t_z$ from~(\ref{eq:t_z}).
In this case the magnetic functional $\Mm=i_\nu\Theta$
is given~by~(\ref{eq:Mm-BOV}).
For loops $z\in \bar\Ll^\times\Zfrak$ and tangent vector fields
$\xi\in T_z \bar\Ll^\times\Zfrak$,
we compute by Proposition~(\ref{prop:dMm-integral}) the difference
in equality one
\begin{equation*}
\begin{split}
   &(L_\nu \Theta)_z\xi-(d\Theta)_z(z^\prime,\xi)
{\color{brown}\qquad\qquad\quad
   =\inner{\tfrac{1}{4\norm{z}^2}
   (-4\abs{z_0}^2\dot\vartheta_{t_z(\tau)}|_{z_\tau})+\Zz_z z^\prime}{\xi}
}
\\
%1
   &\stackrel{{\color{gray}1}}{=}
   \int_0^1
   \left(dt|_z\xi\right)_\tau\dot\vartheta_{t_z(\tau)}|_{z_\tau} z_\tau^\prime
   -\left(\tfrac{d}{d\tau} t_z(\tau)\right)\dot\vartheta_{t_z(\tau)}|_{z_\tau}\xi_\tau
   -d\vartheta_{t_z(\tau)}|_{z_\tau}\left(z_\tau^\prime,\xi_\tau\right)
   d\tau
\\
%2
   &\stackrel{{\color{gray}2}}{=}
   \frac{2}{\norm{z}^2}
   \int_0^1
   \underline{\textstyle{\int_{\sigma=0}^\tau\INNER{z_\sigma}{\xi_\sigma} d\sigma}}
   \cdot
   \dot\vartheta_{t_z(\tau)}|_{z_\tau} z_\tau^\prime\, d\tau
   \\
   &\quad
{\color{gray}\;
   -\frac{2\INNER{z}{\xi}}{\norm{z}^4}
   \int_0^1
   \underline{\textstyle{\int_0^\tau\Abs{z_\sigma}^2 d\sigma}}
   \cdot
   \dot\vartheta_{t_z(\tau)}|_{z_\tau} z_\tau^\prime\, d\tau
}
   \\
   &\quad
{\color{gray}\;
   -\frac{1}{\norm{z}^2}\int_0^1\Abs{z_\tau}^2
   \dot\vartheta_{t_z(\tau)}|_{z_\tau} \xi_\tau\, d\tau
   -\int_0^1d\vartheta_{t_z(\tau)}|_{z_\tau}\left(z_\tau^\prime,\xi_\tau\right) d\tau
.}
\\
\end{split}
\end{equation*}
Equality 2 inserts the formulas
for $(dt|_z\xi)_\tau$ and $\tfrac{d}{d\tau} t_z(\tau)$.
We {\color{gray} grayed out} the terms which will not be modified
in equality 3 below in which we actually only modify summand one:
We change the order of integration and then interchange the names of
the variables $\tau$ and $\sigma$ in order to obtain
\begin{equation}\label{eq:dMm-BOV}
\begin{split}
%3
   &\stackrel{{\color{gray}3}}{=}
   \frac{2}{\norm{z}^2}
   \int_{0}^1
   \INNER{
\underline{\textstyle{
   \int_{\sigma=\tau}^1
   \dot\vartheta_{t_z(\sigma)}|_{z_\sigma} z_\sigma^\prime\,
   d\sigma
}}
   \cdot z_\tau}{\xi_\tau}d\tau
   \\
   &\quad
   -\frac{2\INNER{z}{\xi}}{\norm{z}^4}
   \int_0^1
   \underline{\textstyle{\int_0^\tau\Abs{z_\sigma}^2 d\sigma}}
   \cdot
   \dot\vartheta_{t_z(\tau)}|_{z_\tau} z_\tau^\prime\, d\tau
   \\
   &\quad
   -\frac{1}{\norm{z}^2}\int_0^1\Abs{z_\tau}^2
   \dot\vartheta_{t_z(\tau)}|_{z_\tau} \xi_\tau\, d\tau
   -\int_0^1d\vartheta_{t_z(\tau)}|_{z_\tau}\left(z_\tau^\prime,\xi_\tau\right) d\tau
\\
   &\stackrel{{\color{gray}4}}{=}
   (d\Mm)_z\xi
{\color{brown}\qquad\qquad \qquad\qquad\quad
   =\INNER{\grad\Mm|_z}{\xi}
}
   .
\end{split}
\end{equation}
Equality~4 is by Lemma~\ref{le:gradient-Mm}.
This proves Cartan's formula in the setting of the present article.
This concludes Example~\ref{ex:BOV-repara}
and proves Corollary~\ref{cor:BOV-Cartan-true}.
\end{example}

\newpage %
%
%\vspace{1cm}
% *The following mandatory declaration statements should be present 
% before the reference section in the manuscript under a section 
% entitled "Declarations". If any of the below statement is not
% applicable, kindly include the respective subheading 
% and mention ‘Not Applicable’.

% Funding
% Conflicts of interest
% Availability of data and materials

%\subsubsection*{Declarations}
%\noindent
%\textbf{Funding.}
%\\
%Urs Frauenfelder: 
%   \todo{\small actualize!}
%DFG (Deutsche Forschungsgemeinschaft) grant FR 2637/5-1
%\\
%Joa Weber: Not Applicable

%\smallskip
%\noindent
%\textbf{Conflicts of interest.}
%None of the authors does have a conflict of interest.

%\smallskip
%\noindent
%\textbf{Availability of data and materials.}
%   \todo{\small actualize! Funding, Conflicts of interest}
%Not Applicable

%%%%%%%%%%%%%%%%%%%%%%%%%
%%%%%%%%% REFERENCES %%%%%%
%%%%%%%%%%%%%%%%%%%%%%%%
%\renewcommand{\bibname}{References}
%\bibliographystyle{plain}
         %   erzeugt:     [1] Joa Weber
%\bibliographystyle{abbrv}
         %  erzeugt:      [1] J. Weber and 
\bibliographystyle{alpha}
         %  article:    [Web05]  J. Weber
         %  book:      [Web05]  Joa Weber
         % more authors: [HZ87]
%%%%%%%%%%%%%%%%%%%%%%%%%
%% include Bibliography in TOC %%
% en.wikibooks.org/wiki/LaTeX/Bibliography_Management#Using_tocbibind
%%%%%%%%%%%%%%%%%%%%%%%%%
% Using hyperref, one should say:
%\cleardoublepage
%\phantomsection
\addcontentsline{toc}{section}{References}
\small
\bibliography{$HOME/Dropbox/0-Libraries+app-data/Bibdesk-BibFiles/library_math,$HOME/Dropbox/0-Libraries+app-data/Bibdesk-BibFiles/library_math_2020,$HOME/Dropbox/0-Libraries+app-data/Bibdesk-BibFiles/library_physics}{}
%$
%%%%%%%%%%%%%%%%%%%%%%%%%
%%%%%%%%% standard %%%%%%%%%
%%%%%%%%%%%%%%%%%%%%%%%%%
%\begin{thebibliography}{00000}
%\small
%\end{thebibliography}

\end{document}